\documentclass[11pt]{article}
\usepackage{graphicx}
\usepackage{amssymb}

\newcommand{\B}{\mathbf}

\newcommand{\bp}{\mathbf{\Phi}}
\newcommand{\bs}{\mathbf{s}}

\newcommand{\al}{\alpha}
\newcommand{\e}{\epsilon}

\newcommand{\ot}{\overline{\theta}}
\newcommand{\ol}[1]{\overline{#1}}
\newcommand{\myref}[1]{(\ref{#1})}

\newenvironment{proof}{\noindent\textbf{Proof\ }}{\hspace*{\fill}$\Box$\medskip}

\newtheorem{theorem}{\textbf{Theorem}}[section]
\newtheorem{lemma}{\textbf{Lemma}}[section]
\newtheorem{remark}{\textbf{Remark}}[section]

\newtheorem{definition}{\textbf{Definition}}[section]
\newtheorem{corollary}{\textbf{Corollary}}[section]
\begin{document}
\title{Data-Driven Time-Frequency Analysis}
\author{Thomas Y. Hou\thanks{Applied and Comput. Math, Caltech,
Pasadena, CA 91125. {\it Email: hou@cms.caltech.edu.}} \and
Zuoqiang Shi\thanks{Mathematical Sciences Center, Tsinghua University, Beijing, China, 100084. 
{\it Email: zqshi@math.tsinghua.edu.cn.}} }
\maketitle

\begin{abstract}
In this paper, we introduce a new adaptive data analysis method
to study trend and instantaneous frequency of nonlinear and 
non-stationary data. This method is inspired by 
the Empirical Mode Decomposition method (EMD) and the recently 
developed compressed (compressive) sensing theory. The main
idea is to look for the sparsest representation of multiscale
data within the largest possible dictionary consisting of
intrinsic mode functions of the form
$\{ a(t) \cos(\theta(t))\}$, where $a \in V(\theta)$, $V(\theta)$
consists of the functions smoother than $\cos(\theta(t))$ and 
$\theta'\ge 0$. This problem can be formulated as a nonlinear 
$L^0$ optimization problem. In order to solve this optimization 
problem, we propose a nonlinear matching pursuit method by 
generalizing the classical matching pursuit for the $L^0$ 
optimization problem. One important advantage of this nonlinear 
matching pursuit method is it can be implemented very efficiently 
and is very stable to noise. Further, we provide a convergence 
analysis of our nonlinear matching pursuit method under certain 
scale separation assumptions. Extensive numerical examples will 
be given to demonstrate the robustness of our method and comparison 
will be made with the EMD/EEMD method. We also apply our method to 
study data without scale separation, data with intra-wave frequency 
modulation, and data with incomplete or under-sampled data. 

\end{abstract}

\section{Introduction}

Developing a truly adaptive data analysis method is important for
our understanding of many natural phenomena. Traditional data analysis
methods, such as the Fourier transform, use
pre-determined basis. They provide an effective tool to process
linear and stationary data. However, there are still some limitations
in applying these methods to analyze nonlinear and nonstationary
data. Time-frequency analysis has been developed to overcome the
limitations of the traditional techniques by representing a signal
with a joint function of both time and frequency. The recent
advances of wavelet analysis have opened a new path for time-frequency
analysis. A significant breakthrough of wavelet analysis is the use
of multi-scales to characterize signals. This technique has led to
the development of several wavelet-based time-frequency analysis
techniques \cite{JP90,Daub92, Mallat09}.

Another important approach in the time-frequency analysis is to
study instantaneous frequency of a signal. Some of the
pioneering work in this area was due to Van der Pol
\cite{VdP46} and Gabor \cite{Gabor46}, who introduced
the so-called Analytic Signal (AS) method that uses the
Hilbert transform to determine instantaneous frequency of
a signal. This Analytic Signal method is one of the most
popular ways to define instantaneous frequency. Until very recently,
this method works mostly for monocomponent signals in which
the number of zero-crossings is equal to the number of local
extreme \cite{Boashash92c}. There were other attempts to
define instantaneous frequency such as the zero-crossing method
\cite{Rice44,Shekel53,Meville83} and the Wigner-Ville distribution
method \cite{Boashash92c,LWB93,QC96,Flandrin99,LT96,Picinbono97}.
However most of these methods are rather restrictive.
More substantial progress has been made only recently
with the introduction of the EMD method \cite{Huang98}.
The EMD method provides an effective tool to decompose a
signal into a collection of intrinsic mode functions (IMF)
that allow well-behaved Hilbert transforms for computation
of physically meaningful time-frequency representation.
We remark that the Hilbert spectral representation based on
the wavelet projection has also been carried in \cite{OW04}.

Inspired by the EMD method and the recently developed compressed
(compressive)
sensing theory, we propose a data-driven time-frequency analysis method.
There are two important ingredients of this method. The
first one is that the basis that is used to decompose the data
is derived from the data rather than determined {\it a priori}.
This explains the name ``data-driven'' in our method. The second
ingredient is to look for the sparsest decomposition of the signal
among the largest possible dictionary consisting of intrinsic mode
functions. The adoption of this data-driven basis and the search
for the sparsest decomposition over this highly redundant basis
make our time-frequency analysis method fully adaptive to the
signal. As we are going to demonstrate later, our method can reveal
some hidden physical information of the signal, such as trend and
instantaneous frequency.

Our data-driven time-frequency analysis method is motivated by the
observation that many multi-scale data often have an intrinsic sparse
structure in the time-frequency plane, although its representation in
the physical domain could be quite complicated. The challenge is that
such sparse representation is valid only for certain multiscale basis
that is adapted to the data and is unknown {\it a priori}.
Finding such nonlinear multiscale basis is an essential ingredient
of our method. In this sense, our problem is more difficult than the 
compressed (compressive) sensing problem in which the basis is 
assumed to be known {\it a priori}. One way to find the adaptive basis
is to learn from the data if we have a large number of data samples that
share the similar physical property. This does not apply to our problem
since we deal with only a single signal. We overcome this difficulty
by reformulating the problem as a nonlinear optimization among the
largest possible dictionary. The trade-off is that such decomposition
is not unique. We need to exploit the intrinsic sparse structure of
the data to select the sparsest one among all the possible decompositions.

In our method, the dictionary is given as following:
\begin{eqnarray}
    \mathcal{D}=\left\{a(t)\cos\theta(t):\; \theta'(t)\ge 0,\; a(t),\theta'(t) \in V(\theta)\right\},
  \end{eqnarray}
where $V(\theta)$ is a linear space consisting of functions smoother than
$\cos\theta(t)$.  The construction of $V(\theta)$ with given $\theta(t)$
is in general an overcomplete Fourier basis given below:
\begin{eqnarray}
\label{2-fold-fourier}
  V(\theta)=\mbox{span}\left\{1, \cos\left(\frac{k\theta}{2L_\theta}\right),\sin\left(\frac{k\theta}{2L_\theta}\right),
\; k=1,\cdots,2\lambda L_\theta\right\},
\end{eqnarray}
where $L_\theta=\lfloor\frac{\theta(1)-\theta(0)}{2\pi}\rfloor$,
$\lfloor\mu \rfloor$ is the largest integer less than $\mu$,
 and $\lambda\le 1/2$ is a parameter to control the smoothness of $V(\theta)$.
We then decompose the signal over this dictionary by
looking for the sparest decomposition. The sparest decomposition can be
obtained by solving a nonlinear optimization problem:
 \begin{eqnarray}
P:\quad &&\mbox{Minimize}\quad\quad\quad\quad\quad\quad\quad M\\
&&\mbox{Subject to:}\quad f(t)=\sum_{k=1}^M a_k(t)\cos\theta_k(t),\quad  a_k(t)\cos\theta_k(t)\in \mathcal{D},\;\quad k=1,\cdots,M\nonumber.
\end{eqnarray}
When the signal is polluted by noise, the equality in the above constraint
is relaxed to be an inequality depending on the noise level.
This optimization problem can be viewed as a nonlinear version of the
$L^0$ minimization problem and is known to be very challenging to solve.
Inspired by the compressed (compressive) sensing theory \cite{BDL09},
we propose a $l^1$-regularized nonlinear matching pursuit method to 
solve this nonlinear optimization problem.

Our nonlinear matching pursuit is inspired by the linear matching pursuit
method \cite{MZ93,TG07}. We first extract an intrinsic mode function
$a(t)\cos\theta(t)\in \mathcal{D}$ from the signal $f(t)$ by looking for
the one which matches the signal $f(t)$ best among all the elements in
$\mathcal{D}$. This would imply the following nonlinear optimization problem:
\begin{eqnarray}
\mbox{Minimize}\quad \gamma \| \widehat{a}(t)\|_{l^1} +  \|f(t)-a(t)\cos\theta(t)\|_{l^2}^2,\quad \mbox{Subject to}\;\; a(t)\cos\theta(t)\in \mathcal{D},\nonumber
\end{eqnarray}
where $\gamma >0$ is a regularization parameter and $\widehat{a}(t)$ is the 
representation of $a(t)$ in $V(\theta)$. In some cases when the 
signals are periodic, we can choose 
$\gamma = 0$.
Denote by $r(t)$ the residual after subtracting $a(t)\cos\theta(t)$ from
$f(t)$, i.e. $r(t)=f(t)-a(t)\cos\theta(t)$. We can then treat $r(t)$ as
a new signal to extract the remaining IMFs. There are two important
advantages of this nonlinear matching pursuit approach. The first one
is that this method is very stable to noise perturbation. The second one
is that it can be implemented very efficiently. In the case of $\gamma = 0$, 
the resulting method can be solved approximately by Fast Fourier Transform, and
the complexity of our algorithm is of order $O(N\log N)$ where $N$ is
the number of data sample points that we use to represent the signal.
The low computational cost and the robustness to noise perturbation
make this method very effective in many applications. Moreover, for
data that satisfy certain scale separation conditions, we prove
that our method recovers the IMFs and their instantaneous frequencies
accurately.

We perform extensive numerical experiments to test the robustness and the
accuracy of our data-driven time-frequency analysis method for both 
synthetic data and some real data. Our results show that the nonlinear 
matching pursuit can indeed decompose a multiscale signal into a sparse 
collection of intrinsic mode functions. We also compare our method with
the original EMD method. For the data without noise, we find that our 
method gives results comparable to those obtained by the EMD method. 
Moreover, for noisy data, our method seems to provide better estimation 
of the instantaneous frequency and IMFs than EMD and recently developed EEMD 
method \cite{WH09}.

A common difficulty in many data analysis methods is the relatively large
error produced near the boundary of the data set. For the EMD method,
this source of error is referred to as the ``end effect'', which is 
primarily caused by the use of cubic spline interpolation in 
constructing the envelope and the median of the signal. Our data-driven 
time-frequency analysis method seems to be less sensitive to this end 
effect, especially when the data satisfy certain scale separation 
property. 

We have also extended our data-driven time-frequency analysis method 
to decompose data that do not have a good scale separation property. 
By incorporating a shape function into our dictionary that is adapted 
to the signal, we can also extend our method to decompose data with 
strong intra-wave modulation. Finally, we demonstrate that our 
data-drive time-frequency analysis  method can be applied to recover 
the original signal with missing data in certain interval. The 
recovered signal as well as their instantaneous frequency seems to 
have reasonably good accuracy. We also apply our method to decompose
under-sampled data. The result is quite encouraging even if the 
under-sampled data are polluted by noise.

We remark that there has been some recent progress in developing
a mathematical framework for an EMD like method using synchrosqueezed
wavelet transforms by Daubechies, Lu and Wu \cite{DLW11}. This
seems to be a very promising approach. We have performed some
preliminary numerical experiments to compare the performance of
our method with the synchrosqueezed wavelet approach.
In many cases, we find that the two methods give comparable and
complementary results. We are currently exploring a hybrid approach
that combines the advantages of our method with those of the
synchrosqueezed wavelet approach. Our preliminary results seem
quite encouraging. We will report this in a forthcoming paper.

The remaining of the paper is organized as follows. In
Section 2, we give a brief review of some existing data analysis
methods such as the matching pursuit, the basis pursuit and
the EMD method. We introduce our adaptive data analysis method.
In Section 3, a simplified version of our data-driven time-frequency data analysis
method is introduced for periodic data which can be
implemented efficiently by using the Fast Fourier Transform.
In Section 4, we present some numerical experiments, include 
incomplete or under-sampled data, to demonstrate 
the performance of our method. In Section 5, we generalize our 
data-driven time-frequency data analysis method to analyze data with 
poor scale separation property 
and data with strong intra-wave modulation. We present some
preliminary error analysis of our data-driven time-frequency data
analysis method in Section 6. The technical proof of the main result 
is deferred to the Appendix. Some conclusions are made in Section 7.

\section{Brief review of the existing sparse decomposition methods}
\label{review}
A considerable focus of activities in the recent
signal processing literature has been the development of the sparse signal representations over
a redundant dictionary. Among these methods, the matching pursuit \cite{MZ93} and the basis pursuit \cite{CDS98} have
attracted a lot of attention in recent years due to the development of the compressed (compressive) sensing.
All these methods consist of two parts: a dictionary to decompose
the signal and a decomposition method to select the sparsest decomposition.

\subsection{Dictionaries}

A dictionary is a collection of parameterized waveforms $\mathcal{D}=\{\phi_\gamma\}_{\gamma\in \Gamma}$.
Many dictionaries have been proposed in the literature.
Here we review a few of them that have been used widely.

\vspace{3mm}
\noindent
{\bf A Fourier Dictionary.}
A Fourier dictionary is a collection of sinusoidal
waveforms. More specifically, the waveforms consist of
the following two families,
\begin{eqnarray}
  \phi_{\omega,0}=\cos(\omega t),\quad \phi_{\omega,1}=\sin(\omega t).
\end{eqnarray}
For the standard Fourier dictionary, $\omega$
 runs through the set of all cosines with
Fourier frequencies $\omega_k=2k\pi/n,\; k=0,1,\cdots,n/2$, and all sines with Fourier frequencies
$\omega_k=2k\pi/n,\; k=1,\cdots,n/2-1$, where $n$ is the number of sample points.
We can also obtain an overcomplete Fourier dictionary by sampling the
frequencies more finely. Let $l>1$. We may choose
$\omega_k=2k\pi/(ln),\; k=0,1,\cdots,ln/2$ for cosines
 and $\omega_k=2k\pi/(ln),\; k=1,\cdots,ln/2-1$ for sines. This is an
$l$-fold overcomplete system. In the algorithm for non-periodic data,
we will use this kind of overcomplete Fourier dictionary.

\vspace{3mm}
\noindent
{\bf A Wavelet Dictionary.}
A wavelet dictionary is a collection of translations
and dilations of the basic mother wavelet $\psi$, together with
translations of the scaling function $\varphi$ defined below:
\begin{eqnarray}
  \phi_{a,b,0}=\frac{1}{\sqrt{a}}\psi\left(\frac{t-b}{a}\right),\quad
\phi_{a,b,1}=\frac{1}{\sqrt{a}}\varphi\left(\frac{t-b}{a}\right) .
\end{eqnarray}
For the standard wavelet dictionary, we let $a, b$
 run through the discrete collection of
mother wavelets with dyadic scales $a_j=2^j/n,\;j=j_0,\cdots,\log_2(n)-1$, and locations
that are integer multiples of the scale $b_{j,k}=ka_j,\;k=0,\cdots,2^j-1$, and the collection
of scaling functions at the coarse scale $j_0$. This dictionary consists of $n$ waveforms, which form an orthonormal basis.
As in the Fourier dictionary, an overcomplete wavelet
dictionary can be obtained by sampling the locations more finely.

\vspace{3mm}
\noindent
{\bf A Time-Frequency Dictionary.}
A typical time-frequency dictionary is the Gabor dictionary
due to Gabor (1946). In this dictionary, we take
$\gamma=(\omega,\tau,\theta,\delta)$, where $\omega\in [0,\pi)$ is
frequency, $\tau$ is a location, $\theta$ is a phase, and $\delta$
is the duration. We define the waveform as follows:
\begin{eqnarray}
  \phi_\gamma(t)=\exp\left(-\frac{(t-\tau)^2}{\delta^2}\right)\cos\left(\omega(t-\tau)+\theta\right).
\end{eqnarray}
Such waveforms consist of frequencies near $\omega$ and essentially vanish far away from $\tau$.

\vspace{3mm}
\noindent
{\bf An EMD Dictionary}.
We can also define a dictionary via the EMD method.
In the EMD method \cite{Huang98}, the dictionary is the collection of all
Intrinsic Mode Functions (IMF), which are functions defined descriptively
by enforcing the following two conditions:
\begin{itemize}
  \item[1.] The number of the extreme and the number of the zero crossings of the function must be equal or differ at most by one;
  \item[2.] At any point of the function, the average of the upper
envelope and the lower envelope defined by the local extreme should
be zero (symmetric with respect to zero).
\end{itemize}

Inspired by the EMD method, we will use a variant of the EMD
dictionary to construct a sparse decomposition of a signal via
nonlinear optimization.

\subsection{Decomposition Methods}

In this subsection, we review a few decomposition methods that
can be used to give a sparse decomposition of a signal by exploiting
the intrinsic sparsity structure of the signal. In recent years,
there have been a lot of research activities in looking for the
sparest representation of a signal over a redundant dictionary
\cite{MZ93,CDS98,Donoho06,Candes-Tao06,CRT06a}, i.e.
look for a decomposition of a signal $f$ over a given dictionary
$\mathcal{D}= \{\phi_\gamma\}_{\gamma\in \Gamma}$ as
 \begin{eqnarray}
\label{decomp-dic}
   f=\sum_{k=1}^m\al_{\gamma_k} \phi_{\gamma_k} + R^m,
 \end{eqnarray}
with the smallest $m$, where $R^m$ is the residual. Whether or not 
a signal can be decomposed into a sparse decomposition depends on 
the choice of the dictionary
that we use to decompose the signal. In general, a more redundant
dictionary tends to give better adaptivity, which implies better
sparsity of the decomposition. However, when the dictionary
is highly redundant, the decompositions are not unique.
We need to give a criterion to pick up the ``best'' decomposition
among all the possible choices.

\vspace{3mm}
\noindent
{\bf Matching Pursuit.}
In \cite{MZ93}, Mallat and Zhang introduced a general decomposition 
method called the matching pursuit that exploits the sparsity of a 
signal. Starting from an initial approximation $\bs^0 = 0$ and a
residual $\B{r}^0 = s$, the matching pursuit builds up a sequence of 
sparse approximations step by step. At stage $k$, the method identifies 
an atom that best matches the residual and then adds it to the current
approximation, so that $s^k = s^{k-1}+\al_k\phi_{\gamma k}$, where 
$\alpha_k=<r^{k-1},\phi_{\gamma k}>$ and $r^k=\bs-\bs^{k}$. After $m$ 
steps, one has a representation of the form \myref{decomp-dic}, with 
residual $R^m=r^m$. A similar algorithm was proposed for Gabor 
dictionaries by S. Qian and D. Chen \cite{QC94}.

An intrinsic feature of this algorithm is that when stopped after a few
steps, it yields an approximate sparse representation using only a
few atoms. When the dictionary is orthogonal, the method works perfectly.
If the dictionary is not orthogonal, the situation is less clear.
Recently, J. Tropp and A. Gilbert proved that under some assumptions
on the basis, the orthogonal matching pursuit can solve
the original $l_0$ minimization problem \cite{TG07}.

\vspace{3mm}
\noindent
{\bf Basis Pursuit.}
Another important class of decomposition methods is the basis pursuit,
which was introduced by S. Chen, D. Donoho and M. Saunders \cite{CDS98}.
First, we reformulate the decomposition problem in the following way.
Suppose we have a discrete dictionary of $p$ waveforms
and we collect all these waveforms as columns of an $n$ by
$p$ matrix $\bp$. The decomposition problem \myref{decomp-dic} can
be reformulated as:
\begin{eqnarray}
  \bs=\bp \alpha ,
\end{eqnarray}
where  $\alpha=(\alpha_\gamma)$ is the vector of coefficients in \myref{decomp-dic}.

The basic idea of the basis pursuit is to find a sparse representation
of the signal whose coefficients have a minimal $l_1$ norm, i.e.
the decomposition is obtained by solving the problem
\begin{eqnarray}
  \min\|\al\|_{l^1},\quad \mbox{subject to}\; \bp \alpha=\bs.
\end{eqnarray}

Recently, the basis pursuit has received a lot of attention,
since it is found that under some conditions the basis pursuit
can recover the exact solution of the original $l_0$ minimization
problem \cite{Candes-Tao06, Donoho06}. There has been extensive
research to obtain a sparse representation by the basis pursuit in
a variety of applications. An essential component of the basis
pursuit is to solve the $l^1$ minimization problem. The computational
cost of solving this $l^1$ minimization is more expensive than
the least-square problem in the matching pursuit, although a powerful
Split Bregman method has been introduced by Goldstein and Osher
to speed up the $l^1$ minimization problem considerably \cite{GO09}.

\vspace{3mm}
\noindent
{\bf The EMD decomposition via a sifting process.}
The EMD method decomposes a signal to its IMFs sequentially.
The basic idea behind this approach is the removal of the
local median from a signal by using a sifting process.
Specifically, for a given signal, $f(t)$, one tries to
decompose it as a sum of the local median $m(t)$, and an IMF.
A cubic spline polynomial is used to interpolate
all the local maxima to obtain an upper envelope, and
to interpolate all the local minima to obtain a lower
envelope. By averaging the upper and lower envelopes, one obtains an
approximate median for $m(t)$. One then decides whether or not to
accept the obtained $m(t)$ as our local median depending on
whether $f(t)-m(t)$ gives an acceptable IMF that satisfies the
two conditions that are specified in the definition of an EMD
dictionary. If $f(t)-m(t)$ does not satisfy these
conditions, one can treat $f(t)-m(t)$ as a new signal and construct
a new candidate for the IMF by using the same procedure
described above. This sifting process continues until we obtain
a satisfactory IMF, which we denote as $f_n(t)$. Now we can treat
$f(t)-f_n(t)$ as a new signal, and apply the same procedure to
generate the second IMF, $f_{n-1}(t)$. This procedure continues
until $f_0(t)$ is either monotone or contains at most one extremum.
For more details of the sifting process, we refer to \cite{Huang98}.

\vspace{3mm}
\noindent
{\bf Decomposition based on a nonlinear $TV^3$ minimization.}
Inspired by the EMD method, we proposed a decomposition method
based on a nonlinear $TV^3$ minimization in our previous paper
\cite{HS11}. Here $TV^3$ is the total variation of the third order
derivative of a function, defined as
$TV^3(f) = \int_a^b |f^{(4)} (t)| dt $. We use a $TV^3$ norm
because the $L^1$ norm or the total variation norm is not strong
enough to enforce the regularity of the median or the envelope of
our decomposition. For example, the use of the total variation norm,
which is very popular in imaging processing community, tends to
give a decomposition whose median or envelope is piecewise constant,
which is referred to as the stair-case effect. The $TV^3$ norm, on
the other hand, gives a much smoother decomposition for both the
median and the envelope. Incidentally, the minimization using the
$TV^3$ norm tends to favor piecewise cubic polynomials such as
cubic splines. Thus, our method gives results that are qualitatively
similar to those obtained by the EMD method which uses cubic splines 
to construct its median and envelope from the local extrema of the signal.

We now give a brief review of our $TV^3$ decomposition method.
In our approach, every element in our dictionary automatically
satisfies the conditions of IMF. There is no need to do any
sifting or use the Hilbert transform in our method.
First, we decompose a signal $f(t)$ into its local median $a_0$ and an
IMF $a_1\cos\theta(t)$ by solving the following nonlinear
optimization problem:
\begin{eqnarray}
\label{opt-math}
(P)\quad\quad\mbox{Minimize}&\;& TV^3(a_0)+TV^3(a_1),\\
\mbox{Subject to:}&\;& a_0(t)+a_1(t)\cos\theta(t)=f(t),\quad \theta'(t)\ge 0.\nonumber
\end{eqnarray}
To solve this nonlinear optimization problem, we proposed the following Newton type of iterative method:
\begin{itemize}
\item[]\hspace{-6mm}\textbf{Initialization:} $\theta^0=\theta_0$.
\item[]\hspace{-6mm}\textbf{Main Iteration:}\hspace{1cm}
\begin{itemize}
\item[Step 1:] Update $a_0^n, \;a_1^n,\;b_1^n$ by solving the
following linear optimization problem:
\begin{eqnarray}
\mbox{Minimize}&\quad& TV^3(a_0^n)+TV^3(a_1^n)+TV^3(b_1^n),\\
\mbox{Subject to }:&\quad& a_0^n+a_1^n\cos\theta^{n-1}(t)+b_1^n\sin\theta^{n-1}(t)=f(t).
\end{eqnarray}
\item[Step 2:] Update the phase function $\theta$:
\begin{eqnarray}
\theta^{n}=\theta^{n-1}-\mu\arctan \left(\frac{b_1^n}{a_1^n}\right),
\end{eqnarray}
where $\mu\in [0,1]$ is chosen to enforce that $\theta^{n}$ is an
increasing function:
\begin{eqnarray}
\mu=\max\left\{\alpha\in [0,1]: \frac{d}{dt}\left(\theta_k^{n-1}-\alpha\arctan \left(\frac{b_1^n}{a_1^n}\right) \right)\ge 0\right\}.
\end{eqnarray}

\item[Step 3:] If $\|\theta^n-\theta^{n-1}\|_2\le \e_0$, stop. Otherwise, go to Step 1.
\end{itemize}
\end{itemize}

In \cite{HS11}, we performed a number of numerical experiments  and
compared the results with those obtained by the EMD (or EEMD) method.
Our results show that this method shares many important
properties with the original EMD method. Moreover, its performance
does not depend on numerical parameters such as the number
of sifting or the stop criterion, which seem to have a major
effect on the original EMD method.

There are two limitations of this approach. The first one is that
the computational cost to solve the $TV^{3}$ minimization problem is
relatively high, even if we use the Split Bregman method of Goldstein
and Osher \cite{GO09}. The second one is that this method is more
sensitive to noise perturbation, although a nonlinear filter was
introduced to alleviate this difficulty. In comparison, the
nonlinear matching pursuit method we introduce in this paper
is very stable to noise perturbation and has a relatively low 
computational cost.

\section{Sparse time-frequency decomposition method based on nonlinear matching pursuit}

Our adaptive data analysis method is based on finding the sparsest
decomposition of a signal by solving a nonlinear optimization problem.
First, we need to construct a large dictionary that can be used to
obtain a sparse decomposition of the signal. In principle, the larger
the dictionary is, the more adaptive (or sparser) the decomposition is.

\subsection{Dictionary}
In our method, the dictionary is chosen to be:
\begin{eqnarray}
  \mathcal{D}=\left\{a(t)\cos\theta(t):\; \theta'(t)\ge 0,\; a(t),\theta'(t) \;\mbox{is smoother than}\; \cos\theta(t)\right\}.
\end{eqnarray}
Let $V(\theta)$ be the collection of all the functions that are smoother
than $\cos\theta(t)$. In general, it is most effective to construct
$V(\theta)$ as an overcomplete Fourier basis given below:
\begin{eqnarray}
  V(\theta)=\mbox{span}\left\{1, \cos\left(\frac{k\theta}{2L_\theta}\right),\sin\left(\frac{k\theta}{2L_\theta}\right),
\; k=1,\cdots,2\lambda L_\theta\right\},
\end{eqnarray}
where $L_\theta=\lfloor\frac{\theta(1)-\theta(0)}{2\pi}\rfloor$,
$\lfloor\mu \rfloor$ is the largest integer less than $\mu$,
 and $\lambda\le 1/2$ is a parameter to control the smoothness of 
$V(\theta)$. Then $\mathcal{D}$ can be written as
\begin{eqnarray}
\label{dic-D}
  \mathcal{D}=\left\{a(t)\cos\theta(t):\; \theta'(t)\ge 0,\; a(t),\theta'(t)\in V(\theta)\right\} .
\end{eqnarray}

In some sense, the dictionary $\mathcal{D}$ defined above can be considered
as a collection of IMFs. This property makes our method as adaptive as
the EMD method. We also call an element of $\mathcal{D}$ as an IMF. 
Since the dictionary $\mathcal{D}$ is highly redundant,
the decomposition over this dictionary is not unique. We need a criterion
to select the ``best'' one among all possible decompositions. We assume
that the data we consider have an intrinsic sparse structure in the
time-frequency plane in some nonlinear and nonstationary basis. However,
we do not know this basis {\it a priori} and we need to derive (or learn)
this basis from the data. Based on this consideration, we adopt sparsity
as our criterion to choose the best decomposition. This criterion yields
the following nonlinear optimization problem:
 \begin{eqnarray}
P:\quad &&\mbox{Minimize}\quad\quad\quad\quad\quad\quad\quad M\\
&&\mbox{Subject to:}\quad f(t)=\sum_{k=1}^M a_k(t)\cos\theta_k(t),\quad  a_k(t)\cos\theta_k(t)\in \mathcal{D},\;\quad k=1,\cdots,M,\nonumber
\end{eqnarray}
or
 \begin{eqnarray}
P_\delta:\quad &&\mbox{Minimize}\quad\quad\quad\quad\quad\quad\quad M\\
&&\mbox{Subject to:}\quad \|f(t)-\sum_{k=1}^M a_k(t)\cos\theta_k(t)\|_{l^2}\le \delta,\quad  a_k(t)\cos\theta_k(t)\in \mathcal{D},\;\quad k=1,\cdots,M,
\nonumber
\end{eqnarray}
if the signal has noise with noise level $\delta$.

After this optimization problem is solved, we get a very clear time-frequency representation:
\begin{eqnarray}
 \mbox{Frequency:}\quad \omega_k(t)=\theta_k'(t),\quad \mbox{Amplitude}: \quad a_k(t).
\end{eqnarray}

\subsection{Nonlinear matching pursuit}
\label{NMP}

The above optimization problem can be seen as a nonlinear $L^0$ minimization problem. 
Thanks to the recent developments of the compressed
(compressive)
sensing, it is now well known that there are two kinds of methods to
solve a $L^0$ minimization problem: the basis pursuit and the matching
pursuit. Since the dictionary we adopt here has infinitely many elements,
the generalization of the basis pursuit is not so straightforward. But
the idea of the matching pursuit can be generalized. Applying the idea
of the matching pursuit to our problem, we obtain the following algorithm:

\begin{itemize}
\item $r^0=f(t)$.
\item[Step 1:] Solve the following $l^1$-regularized nonlinear least-square problem $(P_2)$:
\begin{eqnarray}
\label{opt-greedy}
P_2:\quad &&\mbox{Minimize}\quad 
\gamma \|\widehat{a}_{k}\|_{l^1} +
\|r^{k-1}-a_{k}\cos\theta_{k}\|_{l^2}^2\\
&&\mbox{Subject to:}\quad\theta'_k\ge 0,
\quad a_k(t)\in V(\theta_k), \nonumber.
\end{eqnarray}
where $\gamma >0$ is a regularizing parameter and $\widehat{a}_k$ is the representation of $a_k$ in the $V(\theta_{k})$ space.
\item[Step 2:] Update the residual
\begin{eqnarray}
r^{k}=f(t)-\sum_{j=1}^{k}a_{j}(t)\cos\theta_{j}(t).
\end{eqnarray}
\item[Step 3:] If $\|r^{k}\|_{l^2}<\epsilon_0$, stop. Otherwise, go to Step 1.
\end{itemize}

In the first step of the above iterative algorithm, unlike the standard matching pursuit
, we solve a
$l^1$ regularized nonlinear least-square problem $P_2$. To solve this problem, we
use the following Gauss-Newton type method:
\begin{itemize}
\item $\theta_k^0=\theta_0$.
\item[Step 1:] Solve the following $l_1$ regularized least-square problem:
\begin{eqnarray}
\label{opt-linear-ls-arctan}
P_{2,l_2}\quad&&\mbox{Minimize}\quad
\gamma (\|\widehat{a}_k^{n+1}\|_{l^1} + \|\widehat{b}_k^{n+1}\|_{l^1}) + 
\|r^{k-1}-a_{k}^{n+1}(t)\cos\theta_{k}^n(t)-b_k^{n+1}(t)\sin\theta_k^n(t)\|_{l^2}^2 \nonumber \\
&&\mbox{Subject to}\quad a_{k}^{n+1}(t),\;b_{k}^{n+1}(t)\in V(\theta_{k}^n).\nonumber
\end{eqnarray}
where $\widehat{a}_k^{n+1}, \widehat{b}_k^{n+1}$ are the representations of $a_k^{n+1}, b_k^{n+1}$ in the $V(\theta_{k}^n)$ space.
\item[Step 2:] Update $\theta_k^n$,
\begin{eqnarray}
\theta_k^{n+1}=\theta_k^n-\lambda \arctan\left(\frac{b_k^{n+1}}{a_k^{n+1}}\right),
\end{eqnarray}
where $\lambda\in [0,1]$ is chosen to make sure that
$\theta_k^{n+1}$ is a monotonely increasing function.
\begin{eqnarray}
\lambda=\max\left\{\alpha\in [0,1]: \frac{d}{dt}\left(\theta_k^n-\alpha \arctan\left(\frac{b_k^{n+1}}{a_k^{n+1}}\right)
\right)\ge 0\right\}.
\end{eqnarray}
\item[Step 3:] If $\|\theta_k^{n+1}-\theta_k^n\|_{l^2}<\epsilon_0$, stop. Otherwise, go to Step 1.
\end{itemize}

In the first step of above algorithm, we solve a $l_1$ regularized least square problem. 
The $l^1$ regularization tends to stabilize the least 
square problem using an overcomplete Fourier basis. It also favors a
sparse decomposition of the data.

\subsection{A fast algorithm for periodic data}

In the iterative algorithm given in previous section, we need to solve a $l_1$ regularized least square problem in
each step. This is the most expensive part of the algorithm especially
when the number of the data points is large. 
In this subsection, we consider the special case when the data are periodic. In this case, we can introduce an
method based on Fast Fourier Transform(FFT) instead of solving the $l_1$ regularized least square problem.

One big advantage for periodic data is that we can use the standard Fourier basis to construct 
the $V(\theta)$ space in the following way instead of the overcomplete Fourier basis given in \myref{2-fold-fourier}. 
\begin{eqnarray}
\label{def-V}
  V(\theta)=\mbox{span}\left\{\cos\left(\frac{k\theta}{L_\theta}\right), \sin\left(\frac{l\theta}{L_\theta}\right):
k=0,\cdots,\lambda L_\theta
,\; l=1,\cdots,\lambda L_\theta \right\},
\end{eqnarray}
where $\lambda\le 1/2$ is a parameter to control the smoothness of
functions in $V(\theta)$ and $L_\theta=(\theta(1)-\theta(0))/2\pi$ is a positive integer.

The standard Fourier basis is orthogonal to each other, so the $l_1$ 
regularized term is not necessary in this case. Then, the least-square 
problem that we need to solve in the iterative algorithm is described 
below (recall that we set $\gamma = 0$):
\begin{eqnarray}
\label{ls-only}
&&\mbox{Minimize}\quad\|r(t)-a(t)\cos\theta(t)-b(t)\sin\theta(t)\|_{l^2}^2\\
&&\mbox{Subject to}\quad a(t),\;b(t)\in V(\theta)\nonumber.
\end{eqnarray}
In order to simplify the notations, we drop the subscript and
superscript here.


Notice that in the iterative process, the derivative of the
phase function $\theta(t)$ is always monotonically
increasing. Thus, we can use $\theta(t)$ as a new coordinate.
In this new coordinate, $\cos\theta$,$\sin\theta$ and the basis of $V(\theta)$
are simple Fourier modes, then the least-square problem
can be solved by using the Fourier Transform. More specifically,
for a given increasing phase function $\theta(t)$, we have the
following algorithm to solve the least-square problem \myref{ls-only}
approximately:

\begin{itemize}
\item[Step 1:] Apply an interpolation method to obtain the
value of $r(t)$ at a uniform mesh in the $\theta$-coordinate,
$r_\theta(\theta_j)$:
  \begin{eqnarray}
    r_\theta(\theta_j)=\mbox{Interpolate}\;\left(\theta(t_i), r,\theta_j\right),
  \end{eqnarray}
where $\theta_j, \; j=1,\cdots,N$ are uniformly distributed in the
$\theta$-coordinate.

\item[Step 2:] Apply the low-pass filter $\chi(k)$ to the Fourier Transform of $r_\theta$ 
to compute the envelope $a(t)$ and $b(t)$:
\begin{eqnarray}
a(t)&=&2Re\left\{\mathcal{F}^{-1}\left[\widehat{r}_\theta(k+1)\chi(k)\right]\right\},\\
b(t)&=&2Im\left\{\mathcal{F}^{-1}\left[\widehat{r}_\theta(k+1)\chi(k)\right]\right\},
\end{eqnarray}
where $\chi(k)$ will be given later.
\end{itemize}

The low-pass filter $\chi(k)$ is determined by the choice of $V(\theta)$.
The definition of $V(\theta)$ in \myref{def-V} implies that
$\chi(k)$ a stair function given as following:
\begin{eqnarray}
  \label{cutoff-jump}
\chi(k)=\left\{
  \begin{array}{cl}
    1,& -\lambda<k<\lambda\\
0,& \mbox{otherwise}.
  \end{array}
\right.
\end{eqnarray}

 In the theoretical analysis in the subsequent section, 
 we will see that the stair function is not a good choice as a filter.
We can define a different space for $V(\theta)$ by choosing an
appropriate $\chi(k)$. This opens up many choices for $V(\theta)$. In
this paper, we choose the following low-pass filter $\chi(k)$
to define $V(\theta)$:
\begin{eqnarray}
  \label{cutoff-cosine}
\chi(k)=\left\{
  \begin{array}{cl}
    1+\cos\pi k,& -1<k<1\\
0,& \mbox{otherwise}.
  \end{array}
\right.
\end{eqnarray}



Now, by incorporating the FFT-based solver in our iterative algorithm,
we get the following FFT-based iterative algorithm:
\begin{itemize}
\item $\theta_k^0=\theta_0$.
\item[Step 1:] Interpolate $r(t)$ to a uniform mesh in the
 $\theta^n$-coordinate to get $r_{\theta^n}(\theta^n_j)$:
  \begin{eqnarray}
    r_{\theta^n}(\theta^n_j)=\mbox{Interpolate}\;\left(\theta^n(t_i), r,\theta^n_j\right),
  \end{eqnarray}
where $\theta^n_j, \; j=1,\cdots,N$ are uniformly distributed in the
 $\theta^n$-coordinate.

\item[Step 2:] Apply $\chi(k)$ to the Fourier Transform of
$r_{\theta^n}$ to compute $a^{n+1}$ and $b^{n+1}$ on the mesh of
the $\theta^n$-coordinate:
\begin{eqnarray}
a^{n+1}(\theta^n)&=&2Re\left\{\mathcal{F}^{-1}\left[\widehat{r}_{\theta^n}(k+1)\chi(k)\right]\right\},\\
b^{n+1}(\theta^n)&=&2Im\left\{\mathcal{F}^{-1}\left[\widehat{r}_{\theta^n}(k+1)\chi(k)\right]\right\}.
\end{eqnarray}

\item[Step 3:] Interpolate $a^{n+1}(\theta^n)$ and $b^{n+1}(\theta^n)$ back to the uniform mesh of $t$:
  \begin{eqnarray}
    a^{n+1}(t_i)=\mbox{Interpolate}\;\left(\theta^n_j, a^{n+1}(\theta^n_j),t_i\right),\\
    b^{n+1}(t_i)=\mbox{Interpolate}\;\left(\theta^n_j, b^{n+1}(\theta^n_j),t_i\right).
  \end{eqnarray}

\item[Step 4:] Update $\theta^n$ in the $t$-coordinate:
\begin{eqnarray}
\theta_k^{n+1}=\theta_k^n-\lambda \arctan \left(\frac{b_k^{n+1}}{a_k^{n+1}}\right),
\end{eqnarray}
where $\lambda\in [0,1]$ is chosen to make sure that
 $\theta_k^{n+1}$ is monotonically increasing:
\begin{eqnarray}
\lambda=\max\left\{\alpha\in [0,1]: \frac{d}{dt}\left(\theta_k^n-\alpha
\arctan \left(\frac{b_k^{n+1}}{a_k^{n+1}}\right)\right)\ge 0\right\}.
\end{eqnarray}
\item[Step 5:] If $\|\theta_k^{n+1}-\theta_k^{n}\|_2<\epsilon_0$, stop. Otherwise, go to Step 1.
\end{itemize}

In our implementation, instead updating the phase function $\theta$,
we update the instantaneous frequency $\theta'(t)$:
\begin{eqnarray}
\label{dtheta}
  \left(\theta_k^{n+1}\right)'=\left(\theta_k^{n}\right)'+\lambda \Delta\theta_k^{n+1},\quad \quad \Delta\theta_k^{n+1}=\frac{a_k^{n+1}\left(b_k^{n+1}\right)'-
b_k^{n+1}\left(a_k^{n+1}\right)'}{\left(a_k^{n+1}\right)^2+\left(b_k^{n+1}\right)^2}.
\end{eqnarray}
The phase function can be reconstructed from the instantaneous
frequency by integration. The integration constant is set to 0
since this constant does not change the instantaneous frequency.

In the formula \myref{dtheta} to update $\theta$, at the points where
the denominator $\left(a_k^{n+1}\right)^2+\left(b_k^{n+1}\right)^2$ is small, 
the error maybe amplified, then our algorithm may become unstable. In the real computation,
we first set a threshold value $\alpha$, at the points where the denominator is smaller than $\alpha$,
the value of $\Delta \theta$ is interpolated by the value $\Delta \theta$ at other points, its original 
value is not used. In the computations of this paper, $\alpha$ is set to be $0.1$.

Above algorithm is based on the Gauss-Newton type iteration. It is known that this kind of iterations is sensitive to the initial value. It is not 
practical to assume that we can get a good initial guess especially 
when the signal is polluted by noise. In order to abate the dependence 
on the initial guess, we use the following procedure to gradually improve
our approximation to the phase function so that it converges to the correct
value. 

First, for a given phase function $\theta$, we define a set of spaces, 
such that
\begin{eqnarray}
  V_1(\theta)\subset V_2(\theta)\subset \cdots \subset V_K(\theta)= V(\theta)
\end{eqnarray}
In the computation, we first limit $\Delta \theta$ in corresponding $V_1(\theta)$ space, do the iteration until converge, then relax the restriction on $\Delta \theta$ to make it in $V_2(\theta)$, repeat the iteration until 
converge. Repeat this process until the algorithm converge with the restriction of $\Delta \theta$ in $V(\theta)$.
In the real computations, this process converge even from very rough initial guess although we can not prove that.

This procedure is also very easy to implement. First, we apply a narrow low-pass filter on $\Delta \theta$, then make
the filter broader and broader until it approach the filter given in \myref{cutoff-cosine}.

In the above algorithm, we need to perform the Fourier Transform.
In general, this works well only for periodic data. For a
non-periodic signal with good scale separation property, we can 
extend the signal by a mirror reflection and
treat the extended signal as a periodic signal. The result is still
quite reasonable. But for nonperiodic data with poor scale separation
property, the $l^1$-regularized nonlinear matching pursuit is necessary
to reduce the end effect, as we will see in Sections 4 and 5. 

The initial guess of
$\theta$ can be generated by other time-frequency analysis methods,
such as the synchrosqueezed wavelet transforms \cite{DLW11}.
In the following numerical examples, we obtain our initial guess using
a simple approach based on the Fourier Transform. More precisely, by
estimating the wavenumber by which the high frequency components are
centered around, we can obtain a reasonably good initial guess for
$\theta$. The initial guess for $\theta$ obtained in this way is a
linear function. As we will see in the following numerical examples,
even with these relatively rough initial guesses for $\theta$, our
algorithm still converges to the right answer with accuracy determined
by the noise level.

Last point we want to remark is that our method has a close connection
to the wavelet transform.  In some sense, our method is equivalent to 
employ 
wavelet transform in the $\theta$-coordinate in each step. The low-pass filter $\chi(k)$ used in our method
can be related with the scale function in the multiresolution analysis. And space $V(\theta)$ can also be
constructed following the idea of multiresolution analysis. The relation between our method and the wavelet
transform will be studied more systematically in our future papers. 

\section{Numerical results}
In this section, we will perform extensive numerical studies to 
demonstrate the effectiveness of our nonlinear matching pursuit method.
First we will present numerical results for the FFT-based algorithms
for periodic data or data with a good scale separation property (see
section \ref{analysis} for the definition of the scale separation 
property). In the second section, we will present numerical results for
the $l^1$ regularized nonlinear matching pursuit which gives reasonably
accurate decompositions for non-periodic data and even 
for under-sampled data or data with missing information in some
physical domain.

\subsection{Numerical results for the FFT-based algorithms}

In this section, we will present a number of numerical experiments to
demonstrate the accuracy and robustness of our method. We also compare
the performance of our method with that of EMD or EEMD. A main focus of
our numerical study is the robustness of the decomposition to signals
that are polluted with a significant level of noise. When the signal
is free of noise, we observe that the performance of our method
is comparable to that of the EMD/EEMD method. However, when the
signal is polluted by noise with a significant noise-to-signal ratio,
our nonlinear matching pursuit method tends to give better
performance than that of the EMD/EEMD method.

Throughout this section, we denote $X(t)$ as white noise with zero
mean and variance $\sigma^2=1$. The Signal-to-Noise Ratio (SNR,
measured in dB) is defined by
\begin{eqnarray}
  \mbox{SNR[dB]}=10\log_{10}\left(\frac{\mbox{var} f}{\sigma^2}\right).
\end{eqnarray}
We will apply our method to several different signals with
increasing level of difficulty.

\vspace{3mm}
\textbf{Example 1}
\vspace{3mm}
The first example that we consider is a simple non-stationary signal
consisting of a single IMF, which is given below
\begin{eqnarray}
\label{data-single}
f(t)=\cos(60\pi t+15\sin(2\pi t)) .
\end{eqnarray}
In Fig. \ref{single-IF}, we plot the original signal on the left column
and the instantaneous frequency on the right column. The curve with
the red color corresponds to the exact instantaneous frequency
and the one with the blue color corresponds to the one obtained by
by our method. The top row corresponds to the original signal without
noise, $f(t)$. The middle row corresponds to the same signal with a
moderate noise level ($f(t)+X(t)$, SNR=$-3.01$dB) and
the bottom row corresponds to the signal with large noise ($f(t)+3X(t)$,
SNR=$-12.55$dB). In the case when no noise is added, the instantaneous
frequency obtained by our method is almost indistinguishable from
the exact one, see the first row of Fig. \ref{single-IF}.
When the signal has noise, our method can still extract the
instantaneous frequency and corresponding IMF with reasonable accuracy,
see Fig. \ref{single-IF} and Fig. \ref{IMF-single}.

\begin{figure}

    \begin{center}

\includegraphics[width=0.45\textwidth]{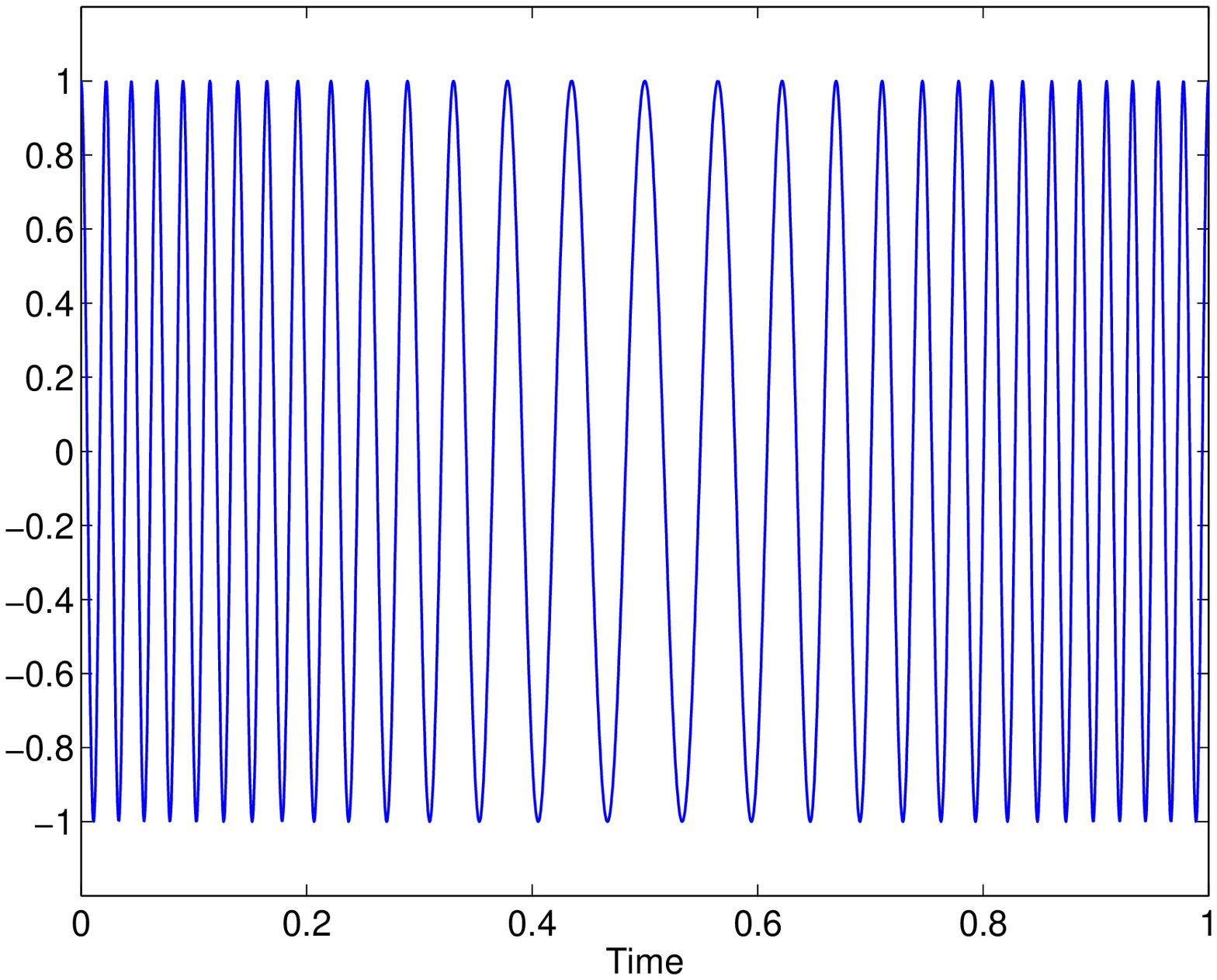}
\includegraphics[width=0.45\textwidth]{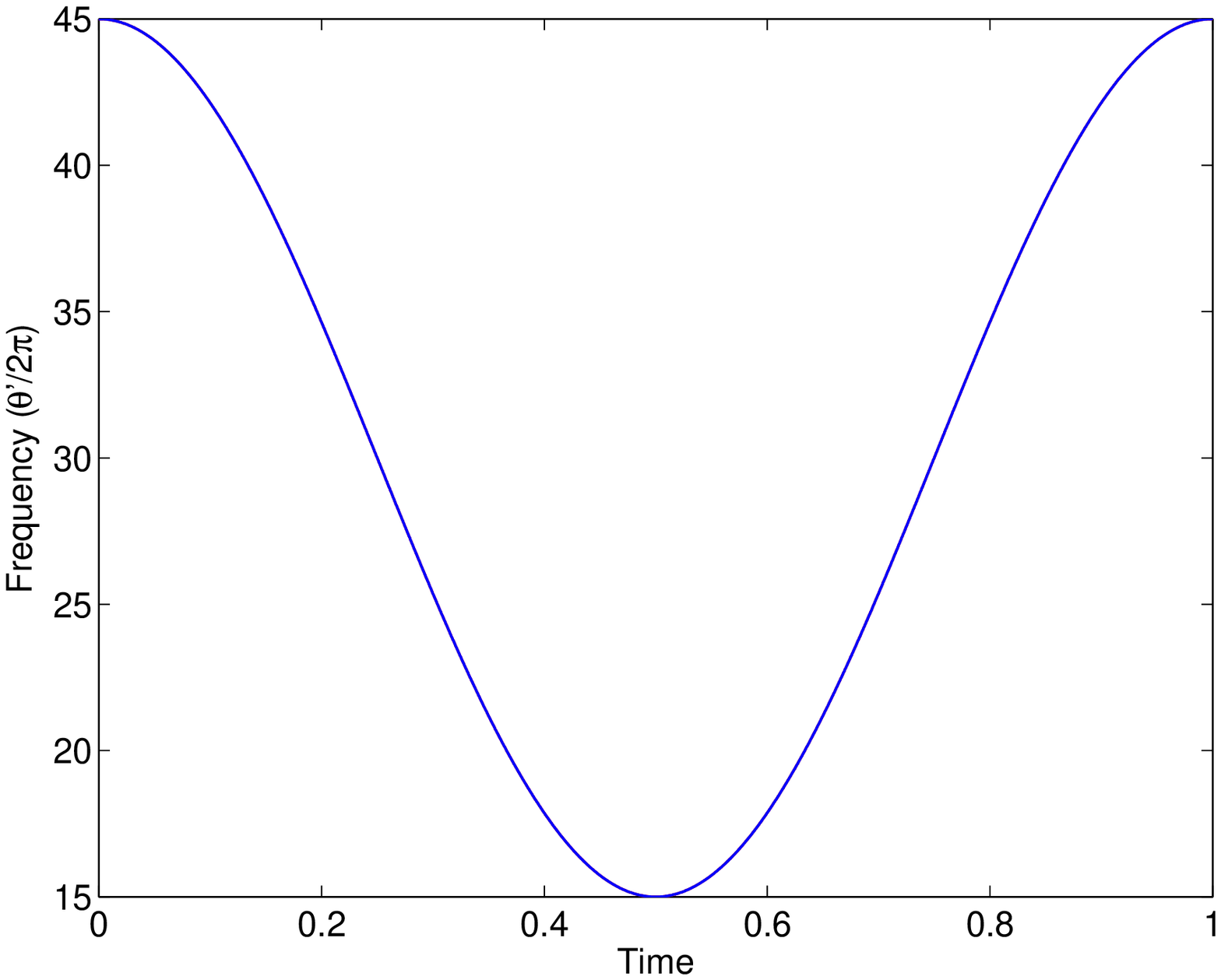}

\includegraphics[width=0.45\textwidth]{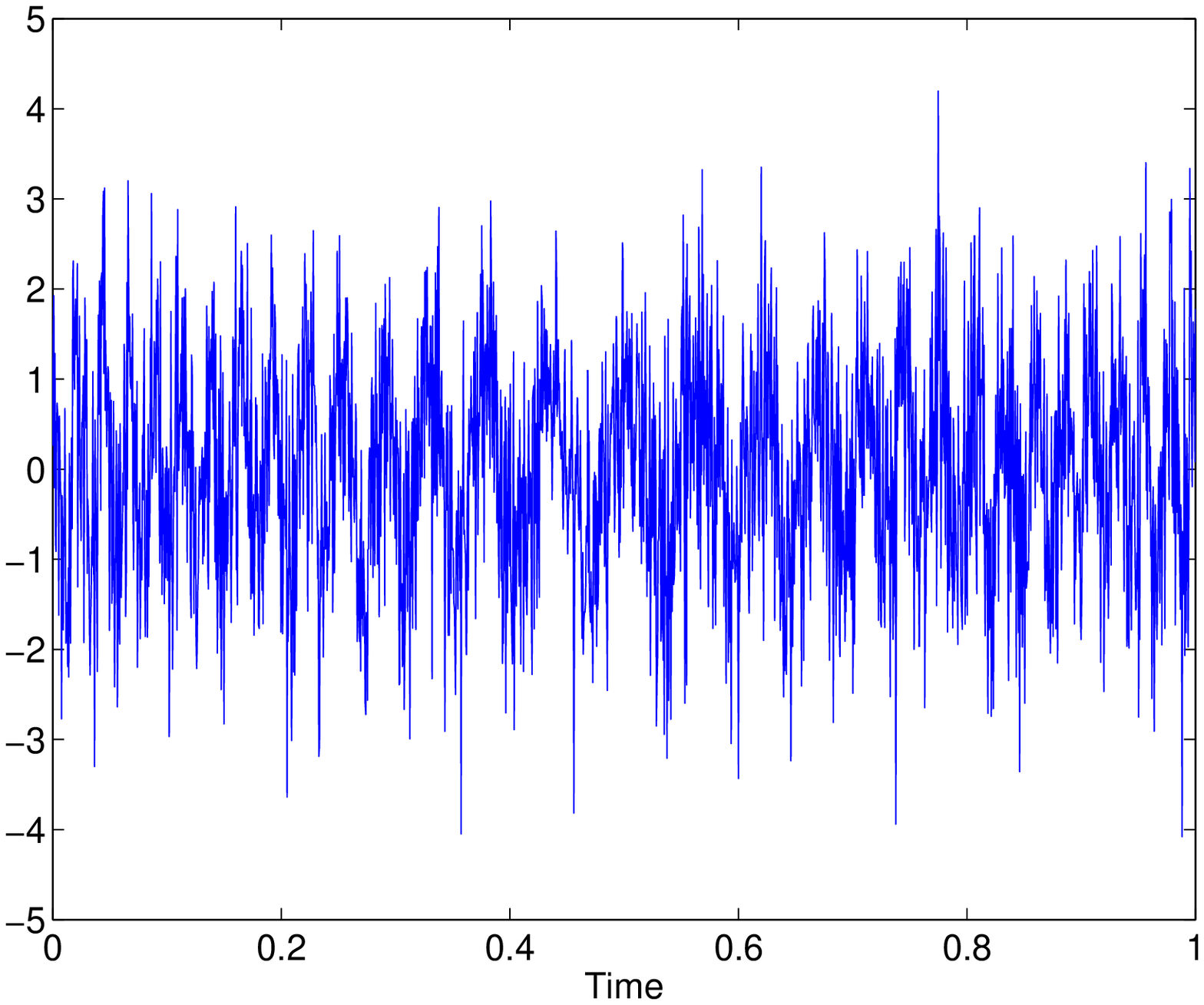}
\includegraphics[width=0.45\textwidth]{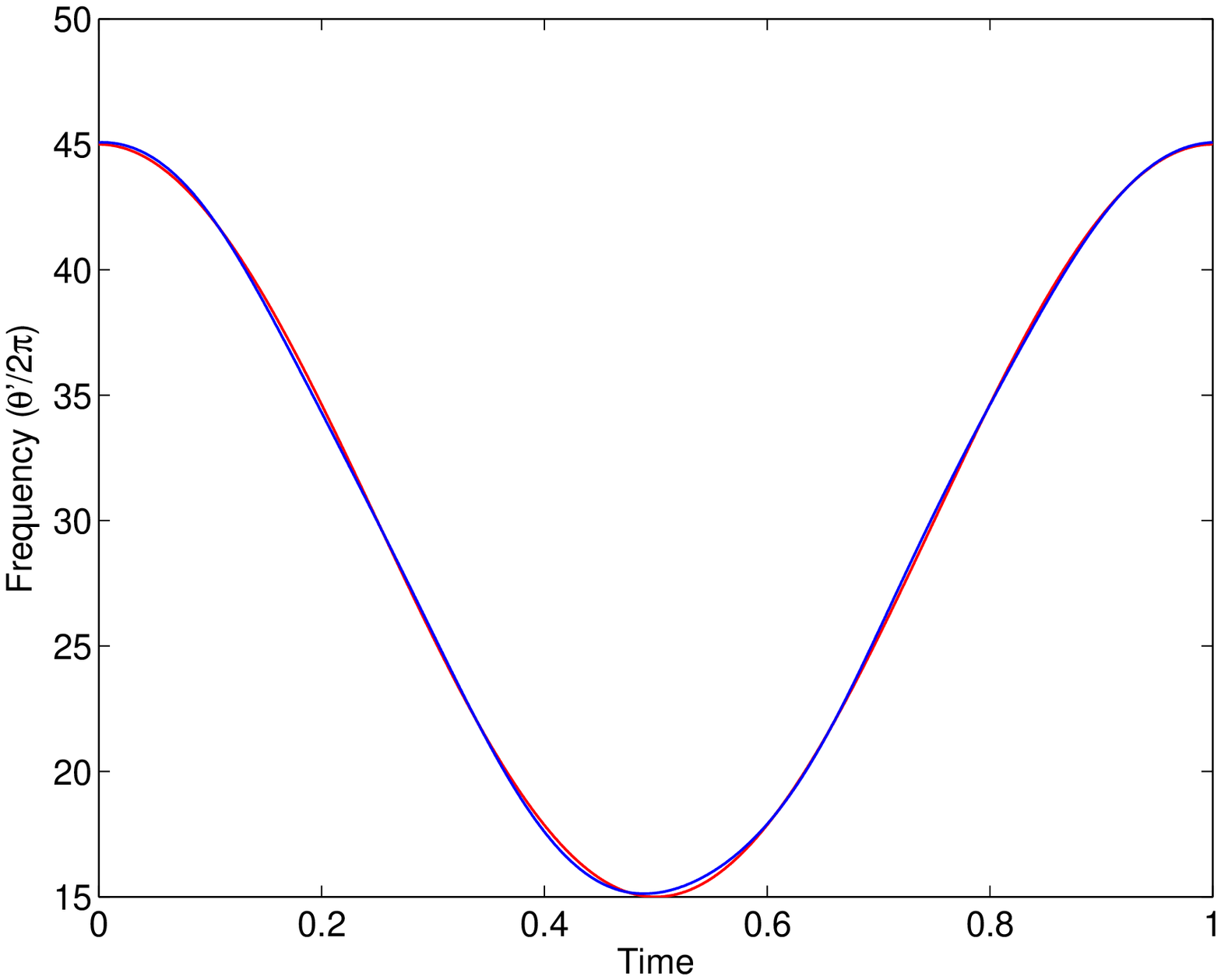}

\includegraphics[width=0.45\textwidth]{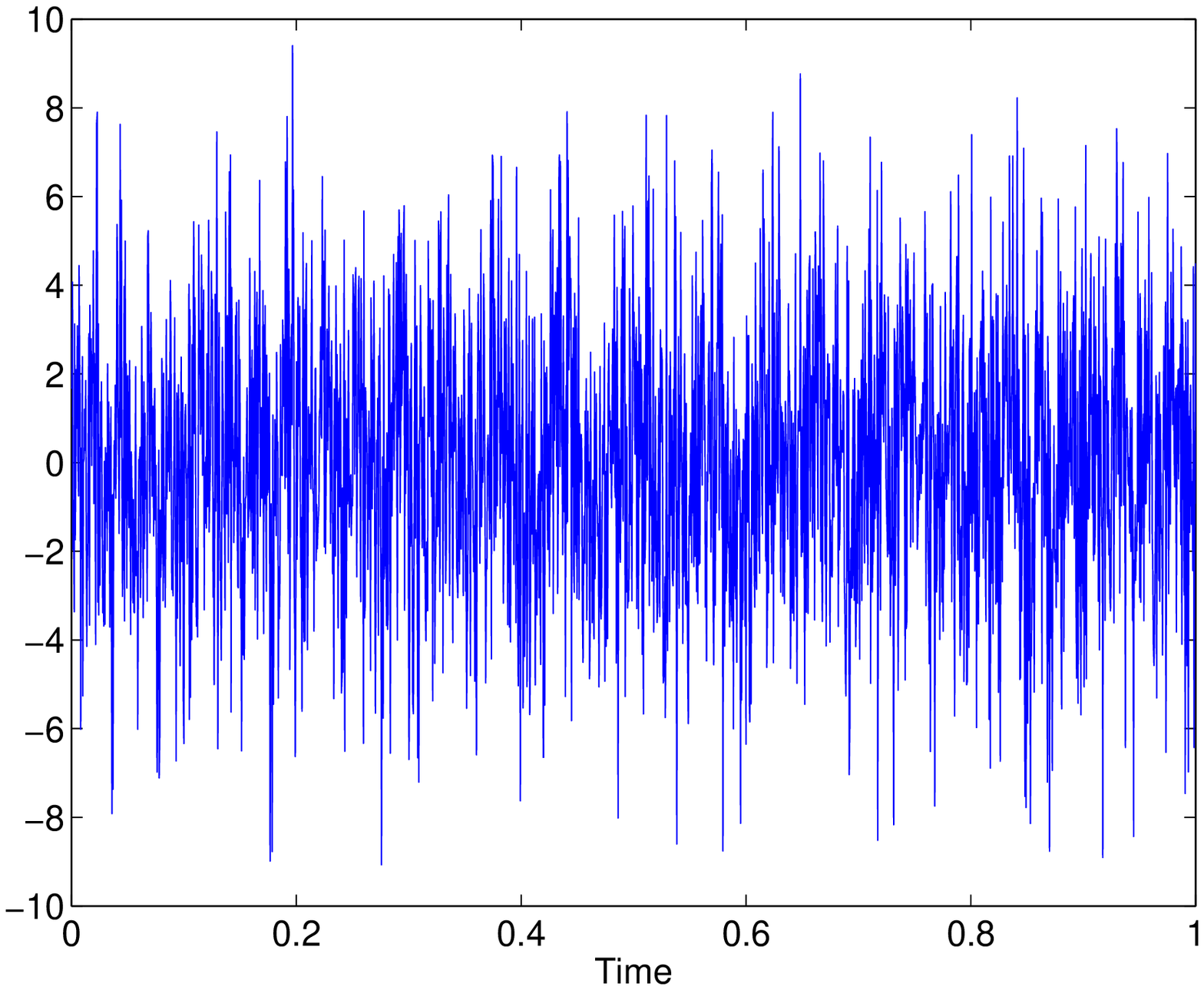}
\includegraphics[width=0.45\textwidth]{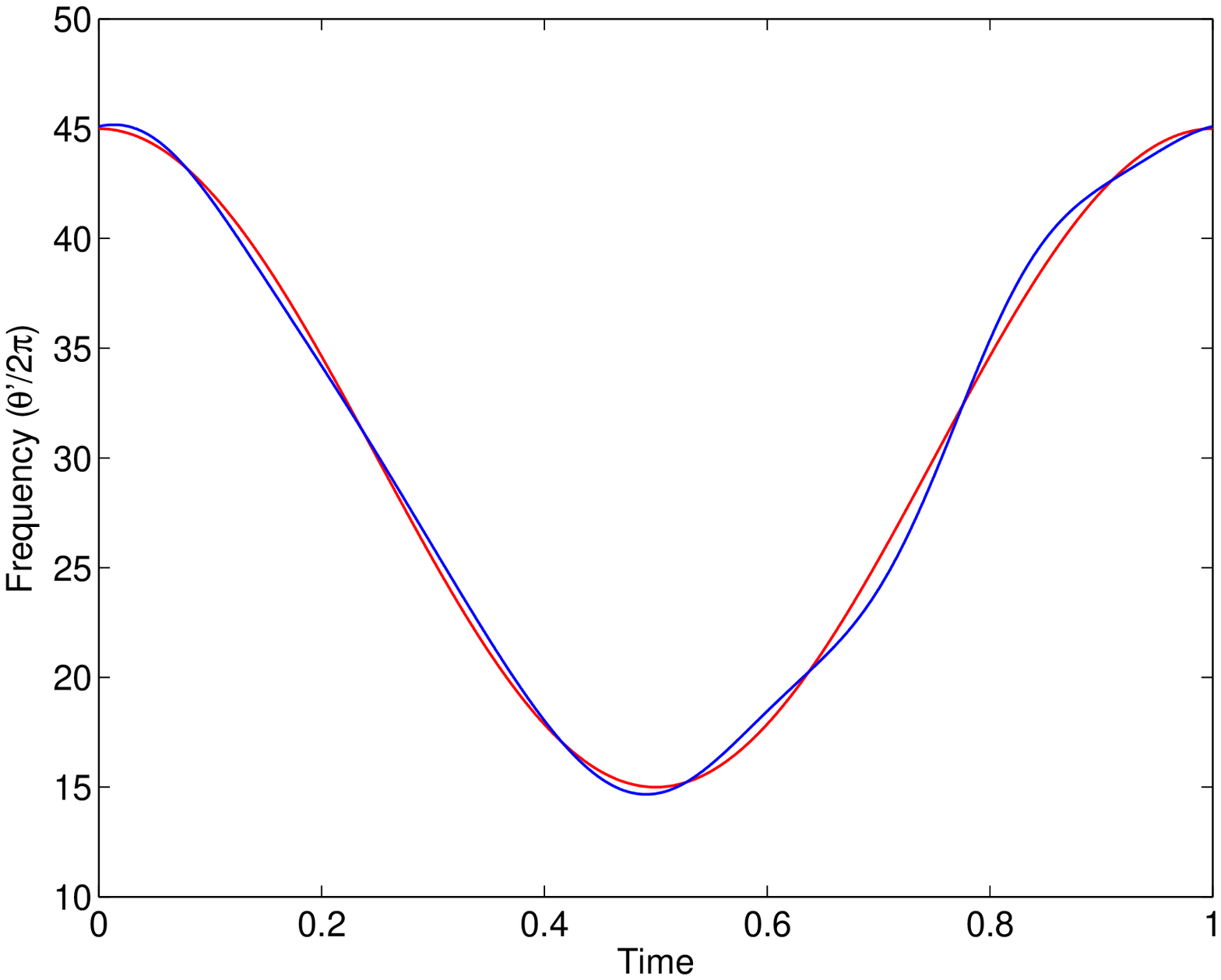}
     \end{center}
    \caption{  \label{single-IF}Top row: left: the original
signal defined by \myref{data-single} without noise; right:
Instantaneous frequencies;
red: exact frequency; blue: numerical results. Middle row:
the same as the top row except white noise $X(t)$ is added to the original
signal, the corresponding SNR is $-3.01$ dB. Bottom row: White noise
$3X(t)$ is added to the original
signal, the corresponding SNR is $-12.55$ dB.}
\end{figure}

In Fig. \ref{IMF-single}, we compare the IMFs extracted by our method
with those obtained by the EMD/EEMD method. For the signal without noise,
we use the EMD method to decompose the signal. For the signal with
noise, we use the EEMD method to decompose the signal. In the
EEMD approach, the number of ensembles is chosen to be 200 and the
standard deviation of the added noise is 0.2. In each ensemble, the
number of sifting is set to 8. Even though the signal we consider has
only a single IMF, the EEMD method still produces several IMFs. Among
different components of IMFs that are produced by the EEMD method, we
select the one that is closest to the exact IMF in $l^2$ norm and
display it in Fig. \ref{IMF-single}.

When the signal does not have noise, both our method and the EMD method
produce qualitatively the same result for this simple signal,
see the first row of Fig. \ref{IMF-single}. When noise is added,
the situation is quite different. The IMFs extracted by our method
still have reasonable accuracy. However, the IMF decomposed by EEMD
fails to capture the phase of the exact IMF in some region. As a
consequence, the accuracy of the instantaneous frequency obtained
by the EEMD method is very poor (not shown here).

\begin{figure}

    \begin{center}

\includegraphics[width=0.45\textwidth]{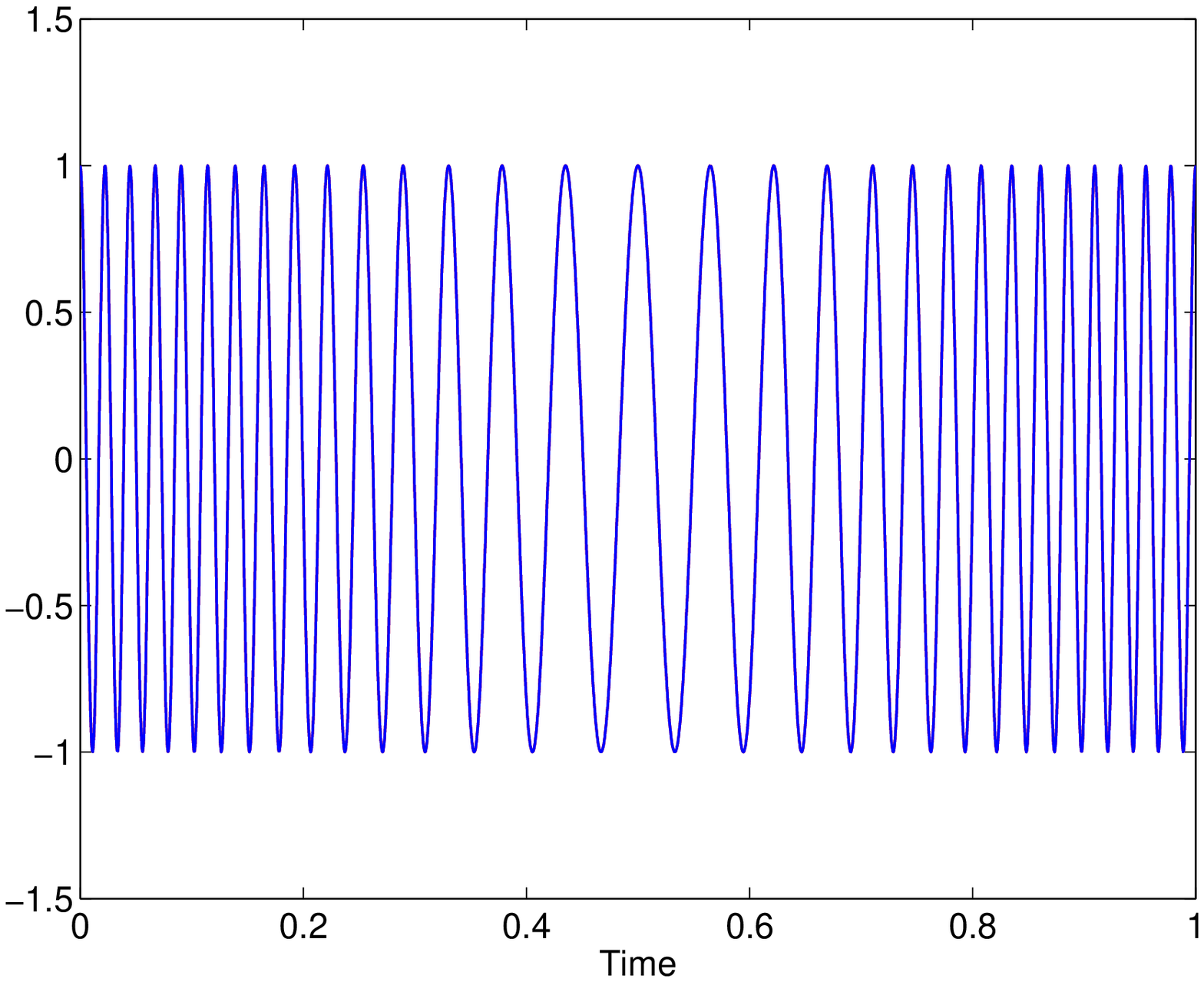}
\includegraphics[width=0.45\textwidth]{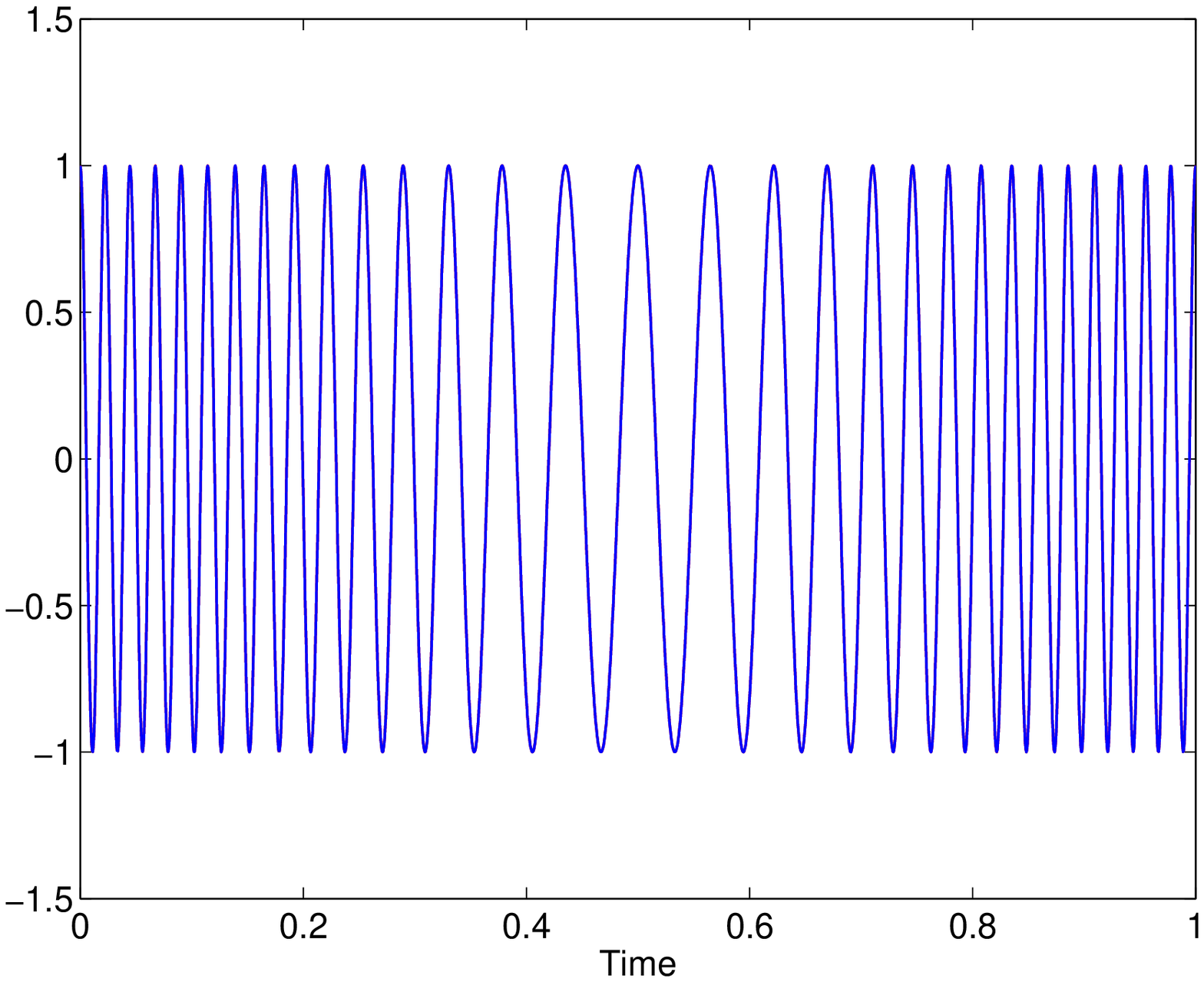}

\includegraphics[width=0.45\textwidth]{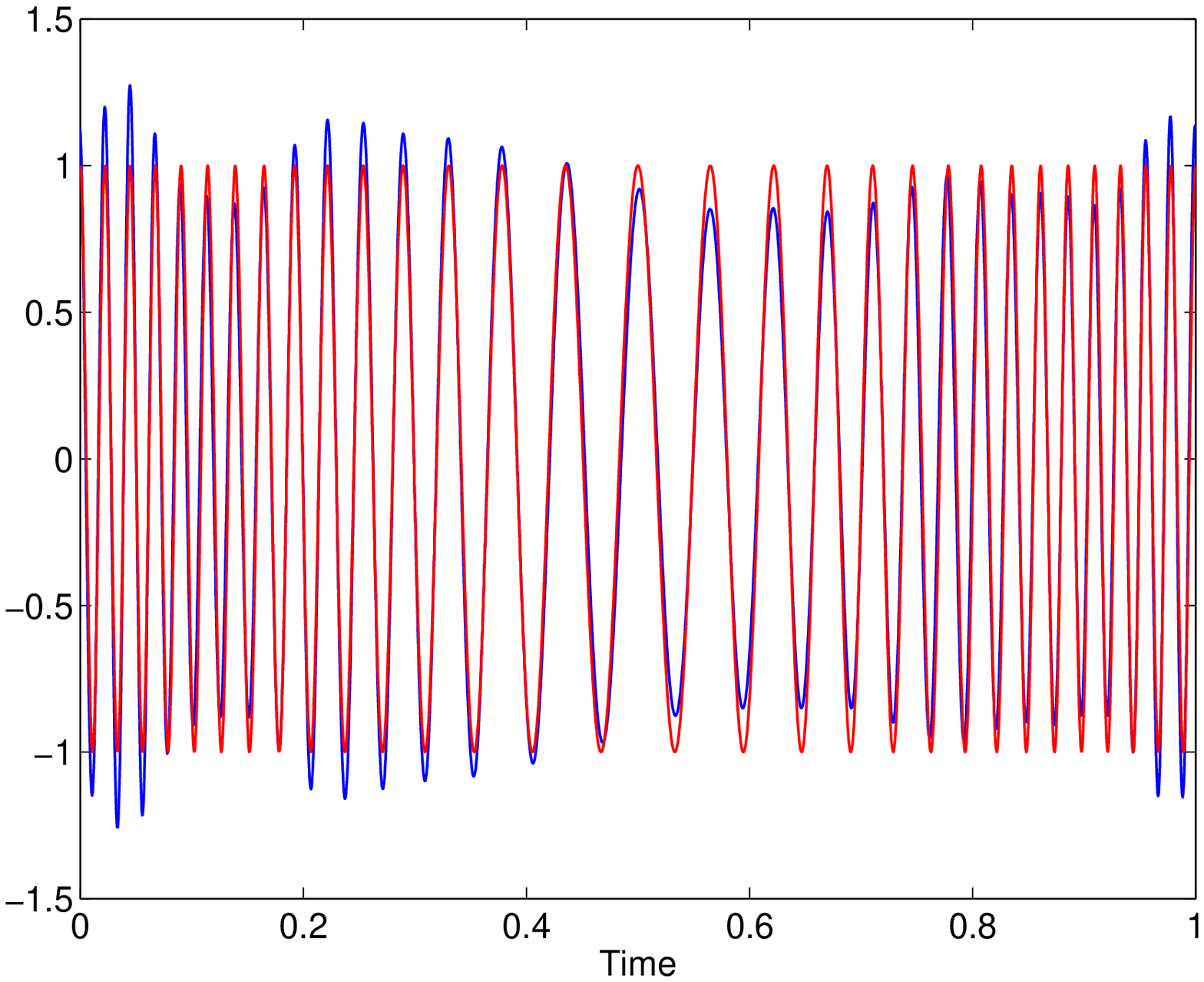}
\includegraphics[width=0.45\textwidth]{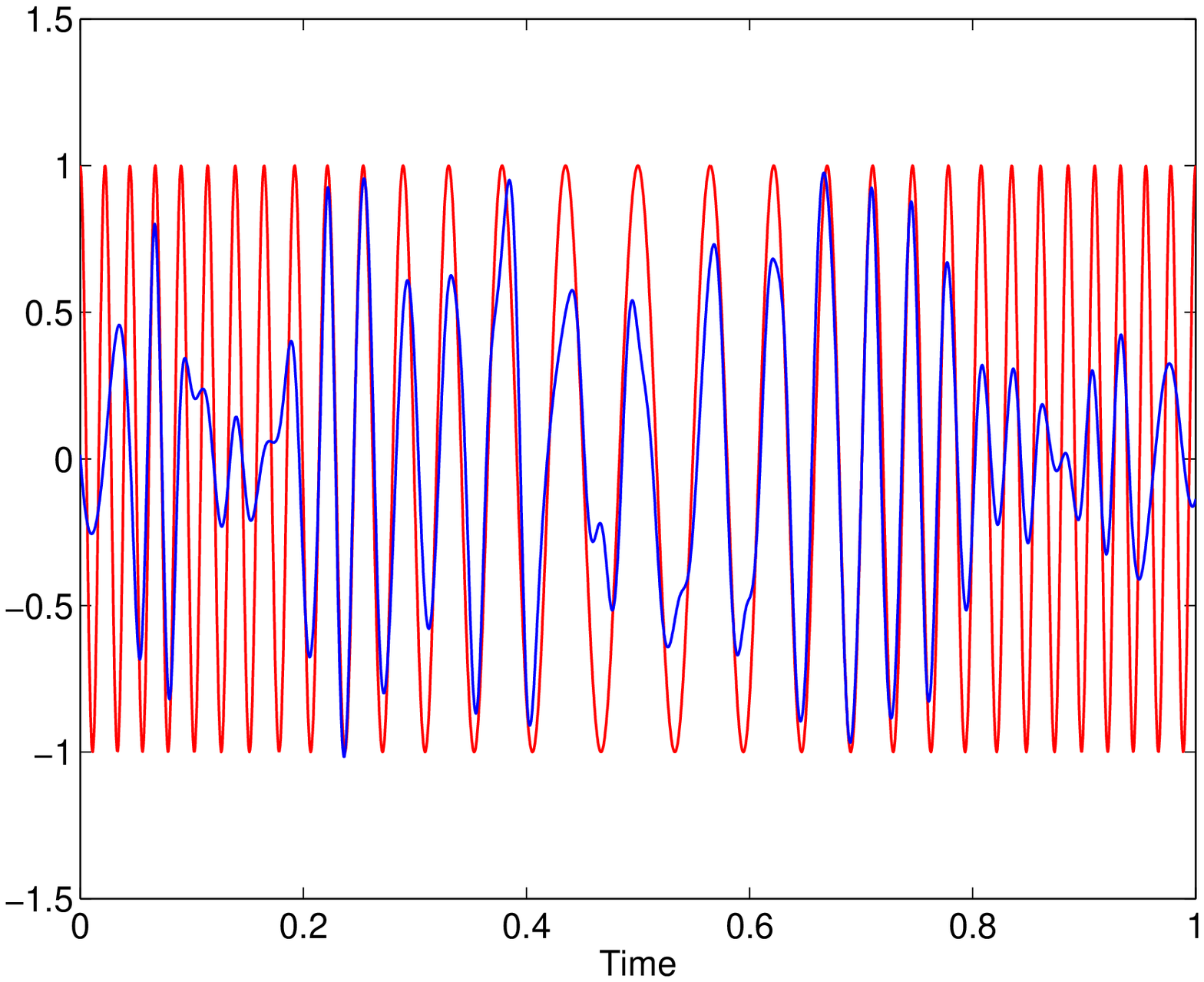}

\includegraphics[width=0.45\textwidth]{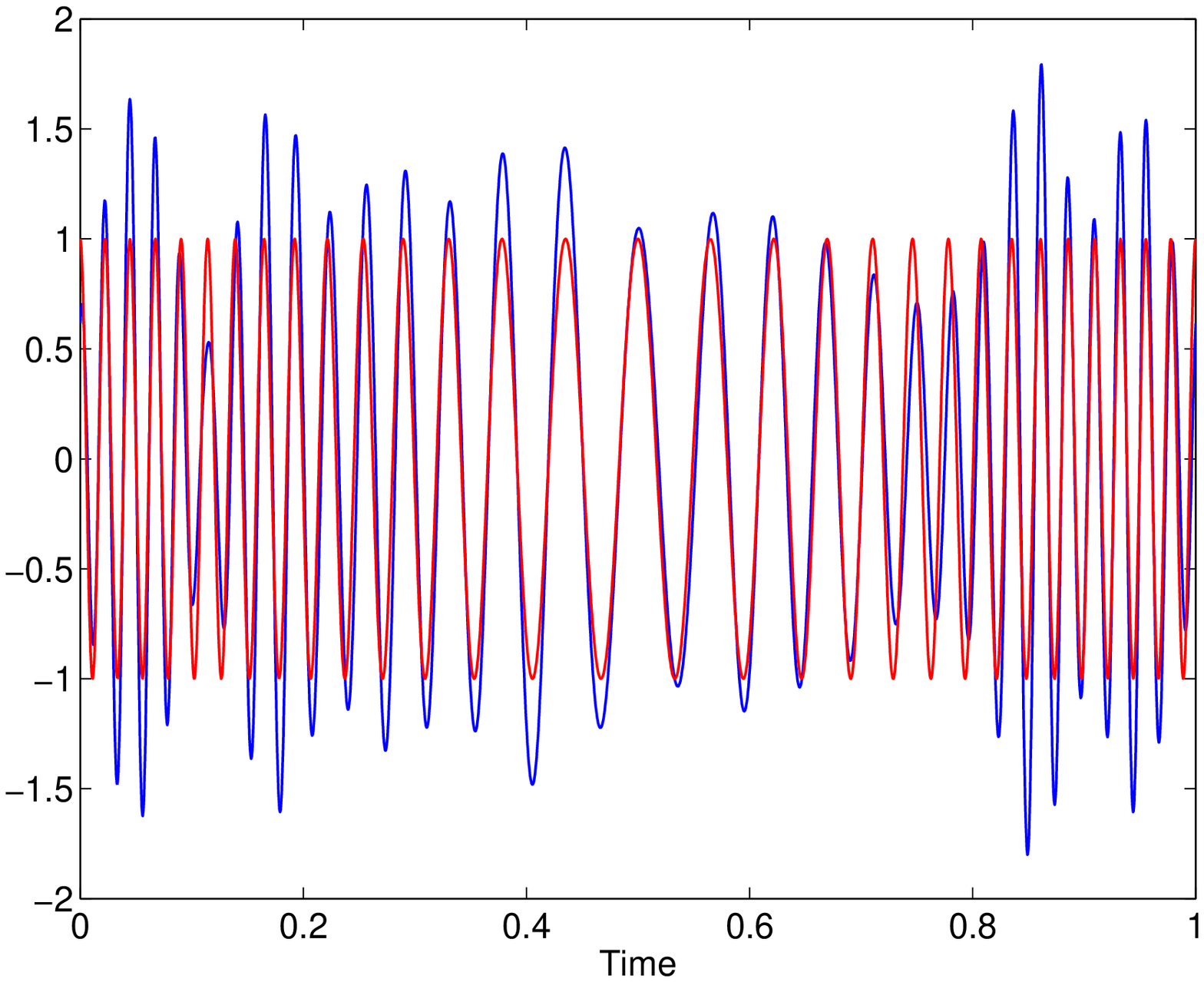}
\includegraphics[width=0.45\textwidth]{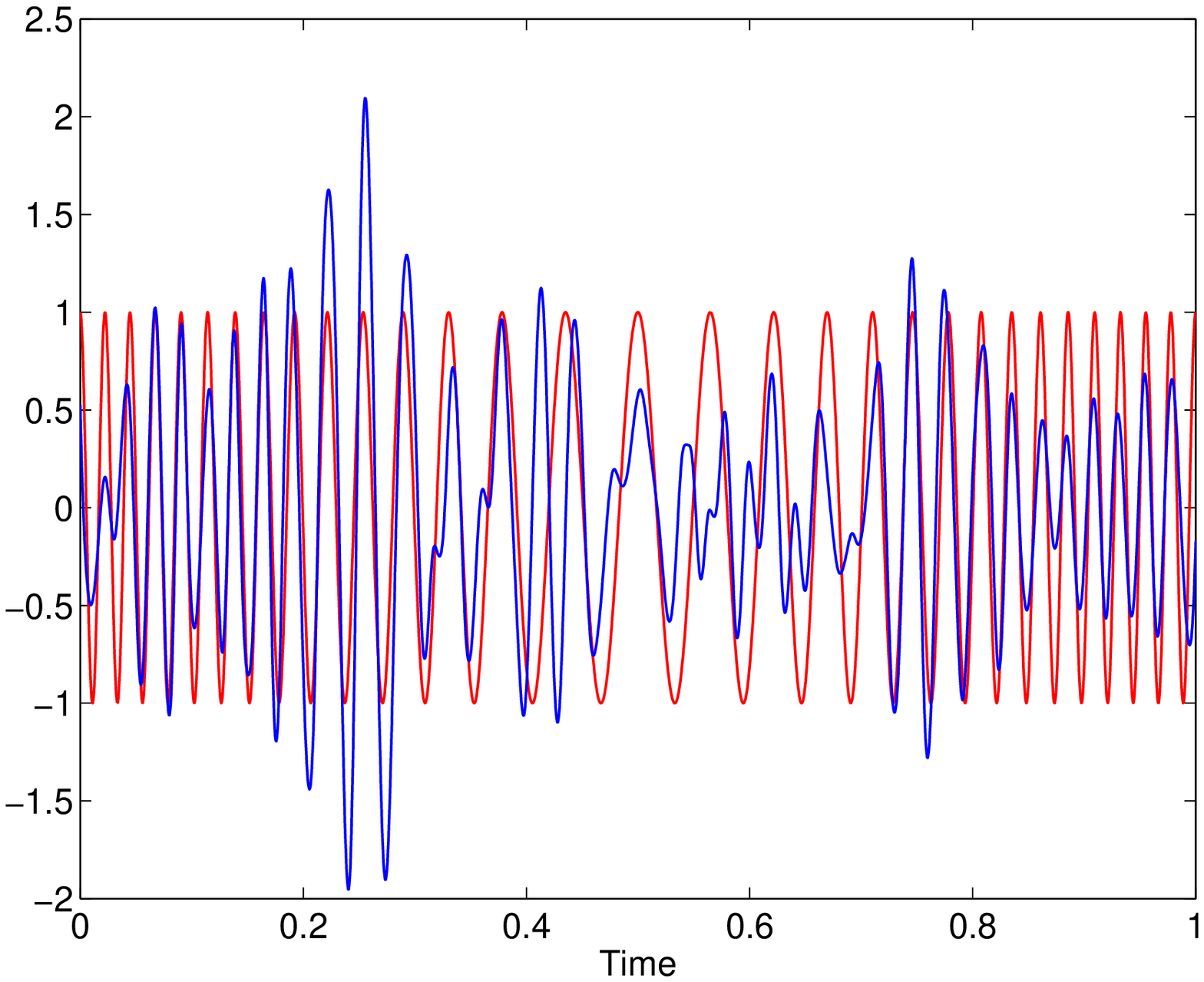}

     \end{center}
    \caption{ \label{IMF-single}The IMFs extracted by our method
and EMD/EEMD method. Left column: IMFs extracted by our method;
Right column: IMFs obtained by EMD/EEMD method. Top row: IMFs from $f(t)$; Middle row: IMFs from $f(t)+X(t)$; Bottom row:
IMFs from $f(t)+3X(t)$. $f(t)$ is defined in \myref{data-single}.}
\end{figure}

\vspace{3mm}
\textbf{Example 2}
\vspace{3mm}

Now, we consider a signal that consists of three IMFs.

\begin{eqnarray}
\label{data-tri}
f(t)&=&\frac{1}{1.5+\cos(2\pi t)}\cos(60\pi t+15\sin(2\pi t))+\frac{1}{1.5+\sin(2\pi t)}\cos(160\pi t+\sin(16\pi t))\nonumber\\
&&
+(2+\cos(8\pi t))\cos(140\pi(t+1)^2) .
\end{eqnarray}
In Fig. \ref{tri-IF}, we study the accuracy of the instantaneous frequencies obtained by our method with the exact instantaneous frequencies. The upper row corresponds to the signal without noise. As we can see, it is hard to tell any hidden structure from this signal even without noise. Our method recovers the three components of the instantaneous frequencies (blue) 
that match the exact instantaneous frequencies (red) extremely well.
They are almost indistinguishable from each other. In the case when noise is added to the original signal, the polluted signal looks really complicated and one can not recognize any hidden pattern from this polluted signal.
It is quite amazing that our method could still recover the three components of the instantaneous frequencies with accuracy comparable to the noise level, see the bottom row of Fig. \ref{tri-IF}.

\begin{figure}

    \begin{center}

\includegraphics[width=0.45\textwidth]{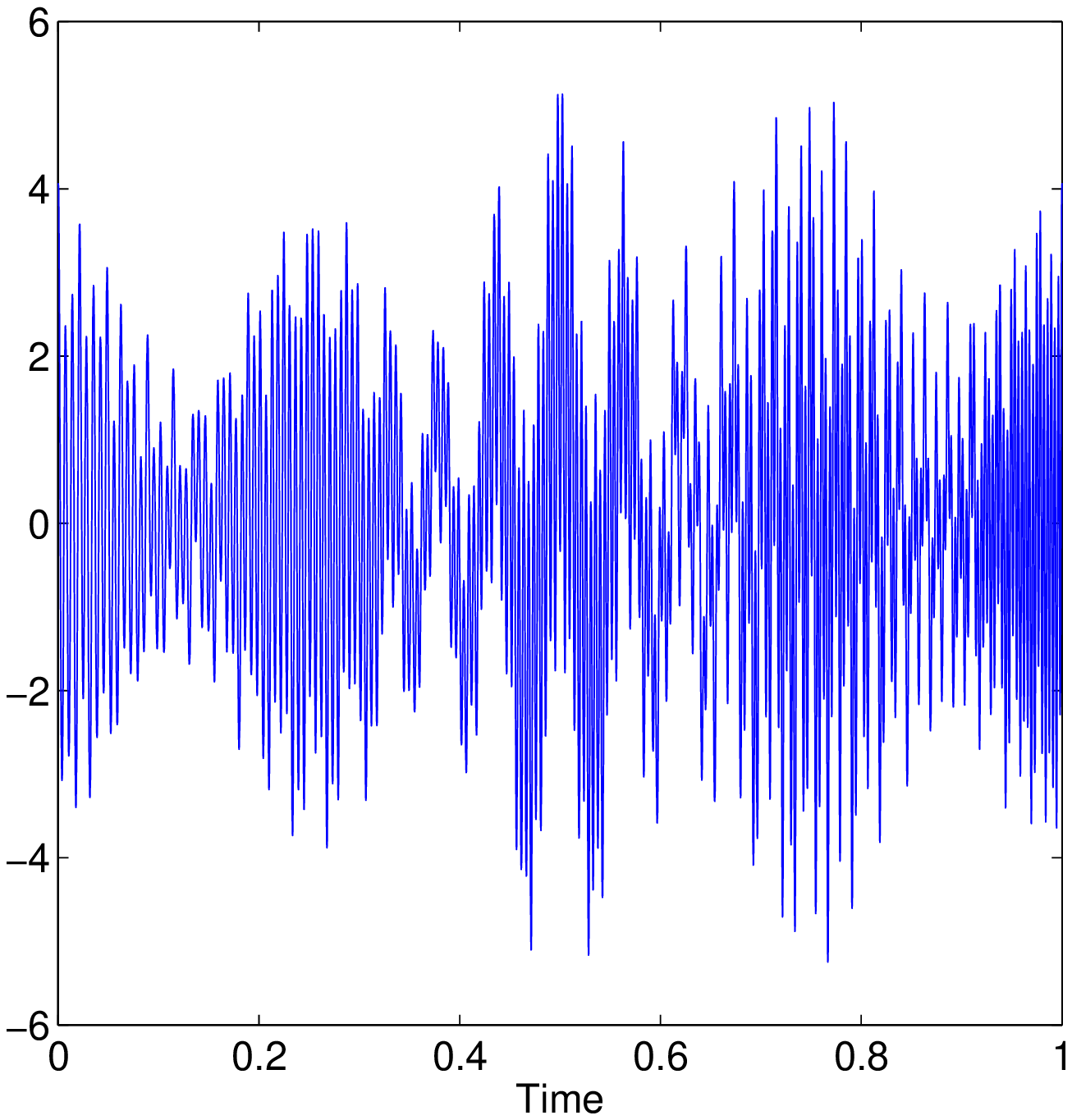}
\includegraphics[width=0.45\textwidth]{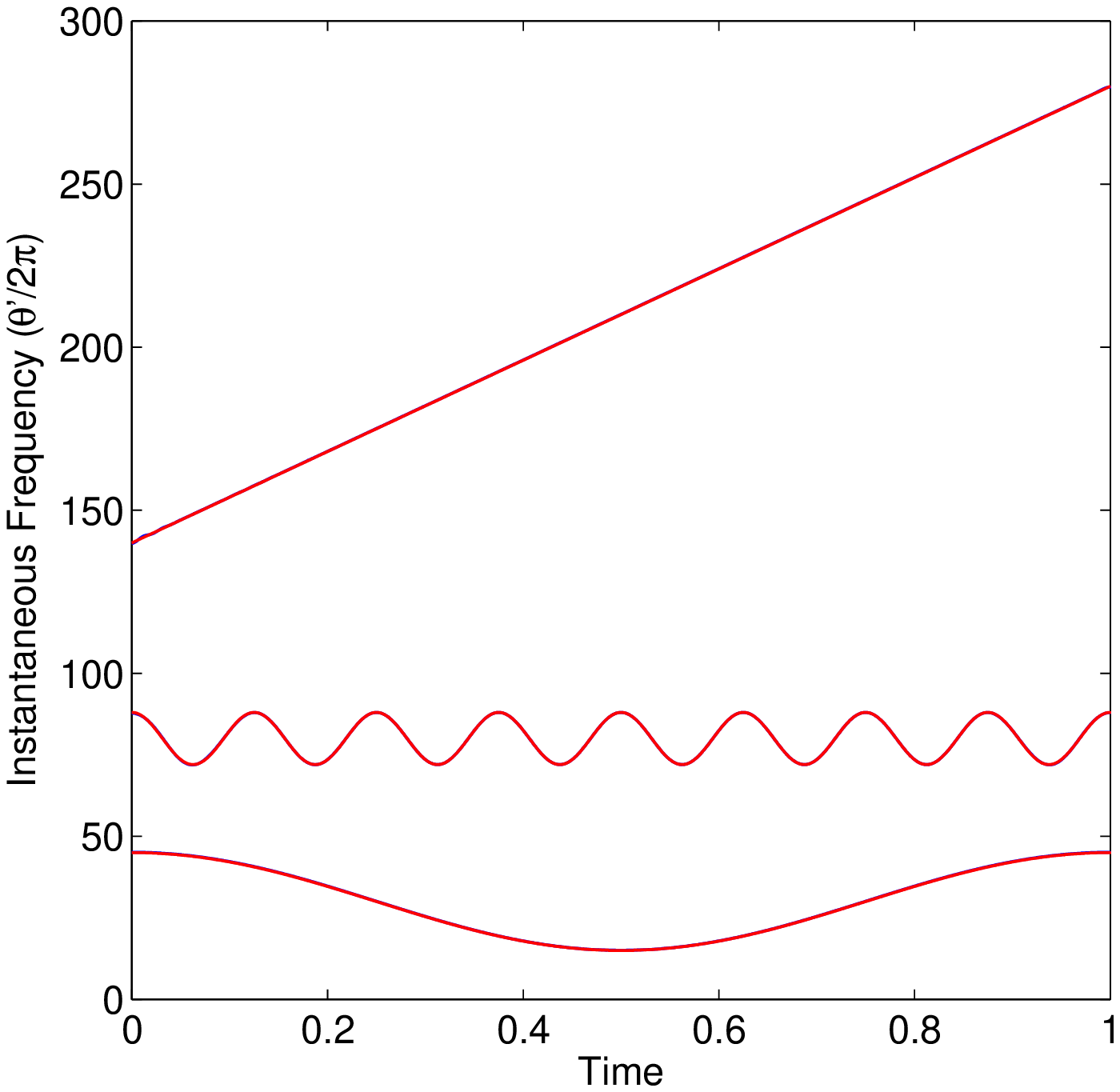}

\includegraphics[width=0.45\textwidth]{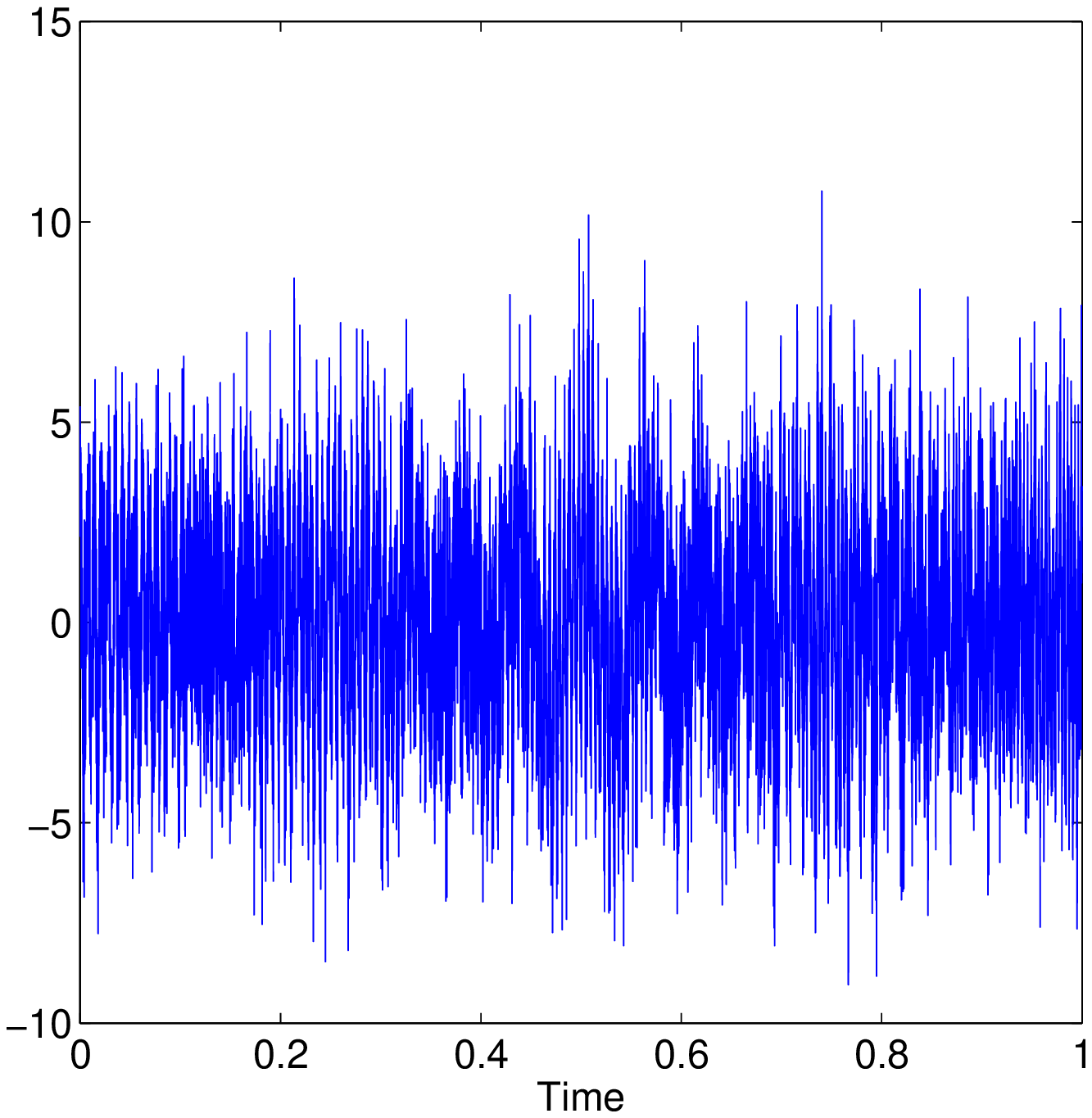}
\includegraphics[width=0.45\textwidth]{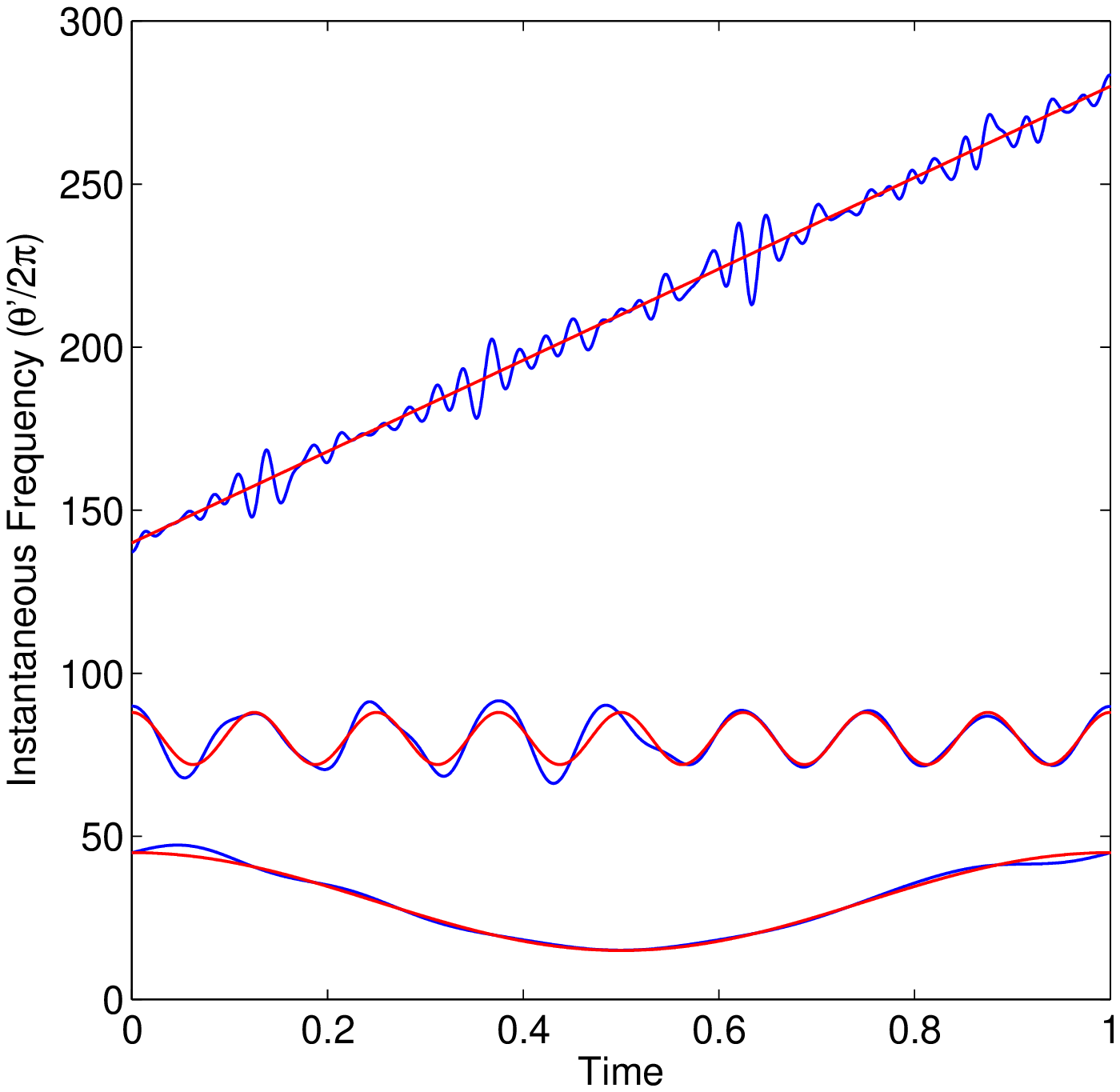}
     \end{center}
    \caption{ \label{tri-IF}Upper row: left: the signal defined in \myref{data-tri} without noise; right: Instantaneous frequencies;
red: exact frequencies; blue: numerical results. Lower row: the same as the upper row except that white noise $2X(t)$ was added to the original
signal, the corresponding SNR is -0.8 dB.}
\end{figure}

\vspace{3mm}
\textbf{Real Data: Length-of-Day Data}
\vspace{3mm}

Next, we apply our method to the Length-of-Day data, see
Fig. \ref{data-LOD}. The data we adopt here was produced
 by Gross \cite{Gross01}, covering the period from 20 January
1962 to 6 January 2001, for a total of 14,232 days (approximate 39 years).
In our previous paper \cite{HS11}, we also studied this data set.
Due to the high computational cost associated with the $l^1$ minimization,
we can not decompose the entire data set. Instead, we decompose
a segment of the data that contains 700 consecutive days.
Thanks to the low computational cost of the FFT-based nonlinear
matching pursuit method, we can now study the entire data set without
any compromise.

\begin{figure}

    \begin{center}

\includegraphics[width=0.9\textwidth]{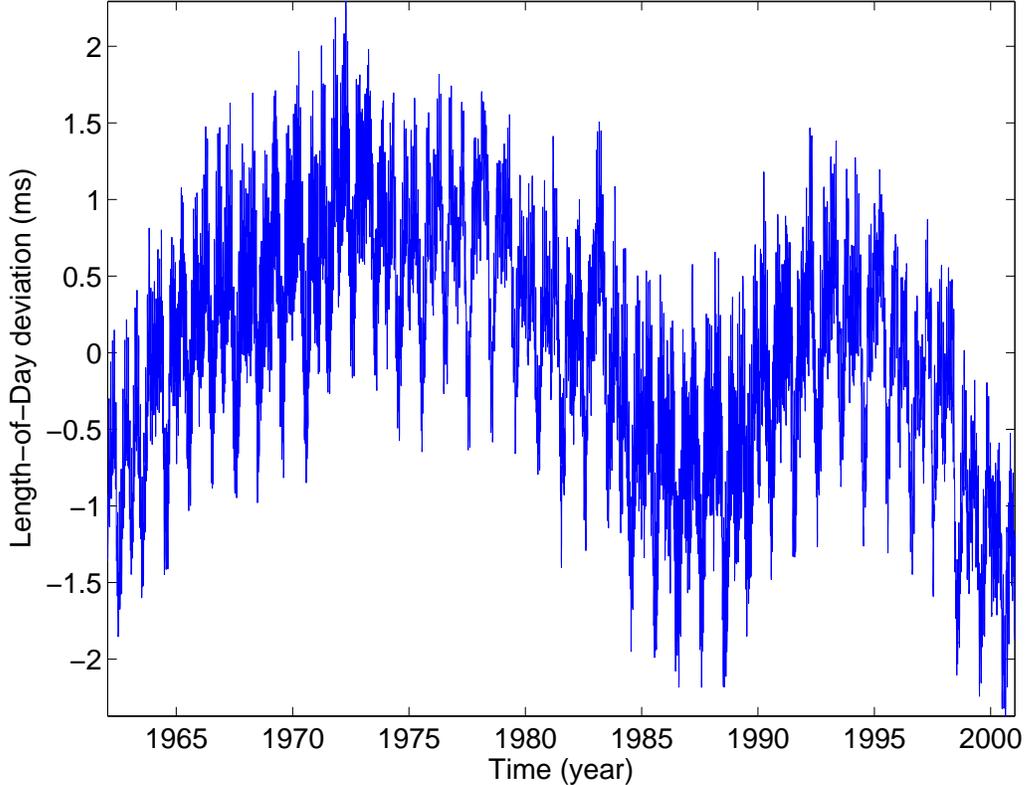}
     \end{center}
     \caption{\label{data-LOD} The daily Length-of-Day data from Jan 20, 1962 to Jan 6, 2001.}
\end{figure}

\begin{figure}

    \begin{center}
\includegraphics[width=0.9\textwidth]{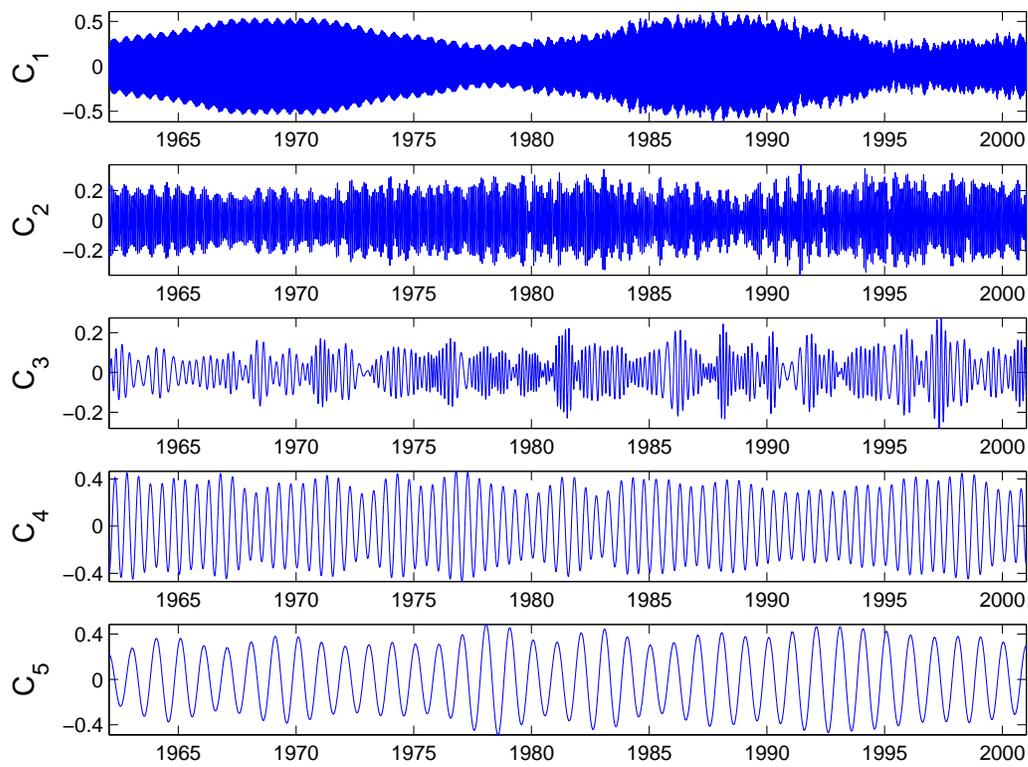}
     \end{center}
     \caption{\label{IMF-LOD} The first 5 IMFs with highest frequencies given by our FFT-based method.}
\end{figure}

Fig. \ref{IMF-LOD} displays the first 5 IMFs extracted by the FFT-based method.
These IMFs are sorted by their frequencies from high to low. We note that the results obtained
by our method do not suffer from the mode mixing phenomenon that is
present in the EMD decomposition.
Moreover, each component is enforced to be an IMF by the construction
of our dictionary. Thus, there is no need to do shifting or
post-processing as was done in the EMD or the EEMD method.
And the IMFs we got are qualitatively match those obtained by EEMD with post-processing \cite{WH09}.

It is interesting to note that the each IMF that we obtain has
a clear physical interpretation. For example, the period of $C_1$ is
around 14 days, corresponding to the semi-monthly tides.
The period of $C_2$ is about 28 days, corresponding to the monthly tides.
Similarly, the period of $C_4$ is about half a year, corresponding to
the  semi-annual cycle and $C_5$ corresponds to the annual cycle.

\subsection{Numerical results for the $l^1$ regularized nonlinear matching pursuit}
\label{non-periodic}

The nonlinear matching pursuit method based on the Fourier Transform
works well only for periodic data. For non-periodic data or data with
poor scale separation, the results
obtained by this method tend to produce some oscillations
near the boundary. This so-called ``end effect'' is also present
in the EMD method and other data analysis methods. In our method, the
``end effect'' comes from the use of the Fourier Transform in the
algorithm. To remove this ``end effect'' error, we need to use the
$l^1$ regularized nonlinear matching pursuit described in Section 3.2
with $V(\theta)$ being the overcomplete Fourier basis defined 
in \myref{2-fold-fourier}.

\subsubsection{Numerical results for non-periodic data}

In this subsection, we perform a numerical experiment to test 
the effectiveness of our $l^1$ regularized nonlinear matching pursuit
for non-periodic data. We first consider the following data in 
our experiment:
\begin{eqnarray}
\label{chirp-nonperiodic}
&&\theta_1=20\pi (t+1)^2+1,\quad \theta_2=161.4\pi t+4(1-t)^2\sin(16\pi t),\nonumber\\
&&f(t)=\frac{1}{1.5+\sin(1.5\pi t)}+(2t+1)\cos\theta_1+(2-t)^2\cos\theta_2.
\end{eqnarray}
In this numerical example, the parameter $\gamma$ is chosen to be 1.
From Fig. \ref{chirp-global}, we observe that the $l^1$ regularized 
nonlinear matching pursuit seems to produce considerably smaller error 
near the boundary for this non-periodic signal.
\begin{figure}

    \begin{center}
	\includegraphics[width=0.5\textwidth]{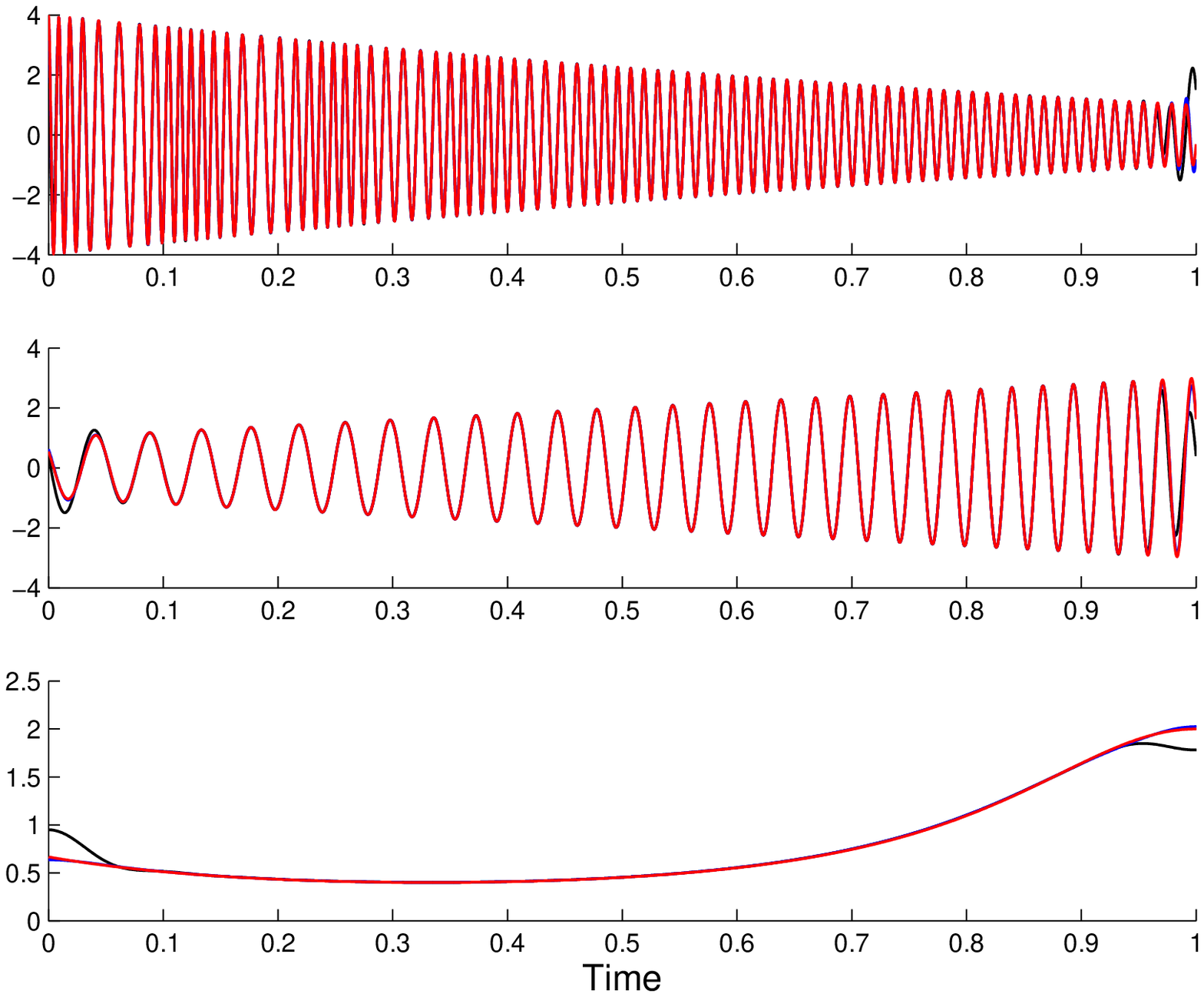}
	\includegraphics[width=0.42\textwidth]{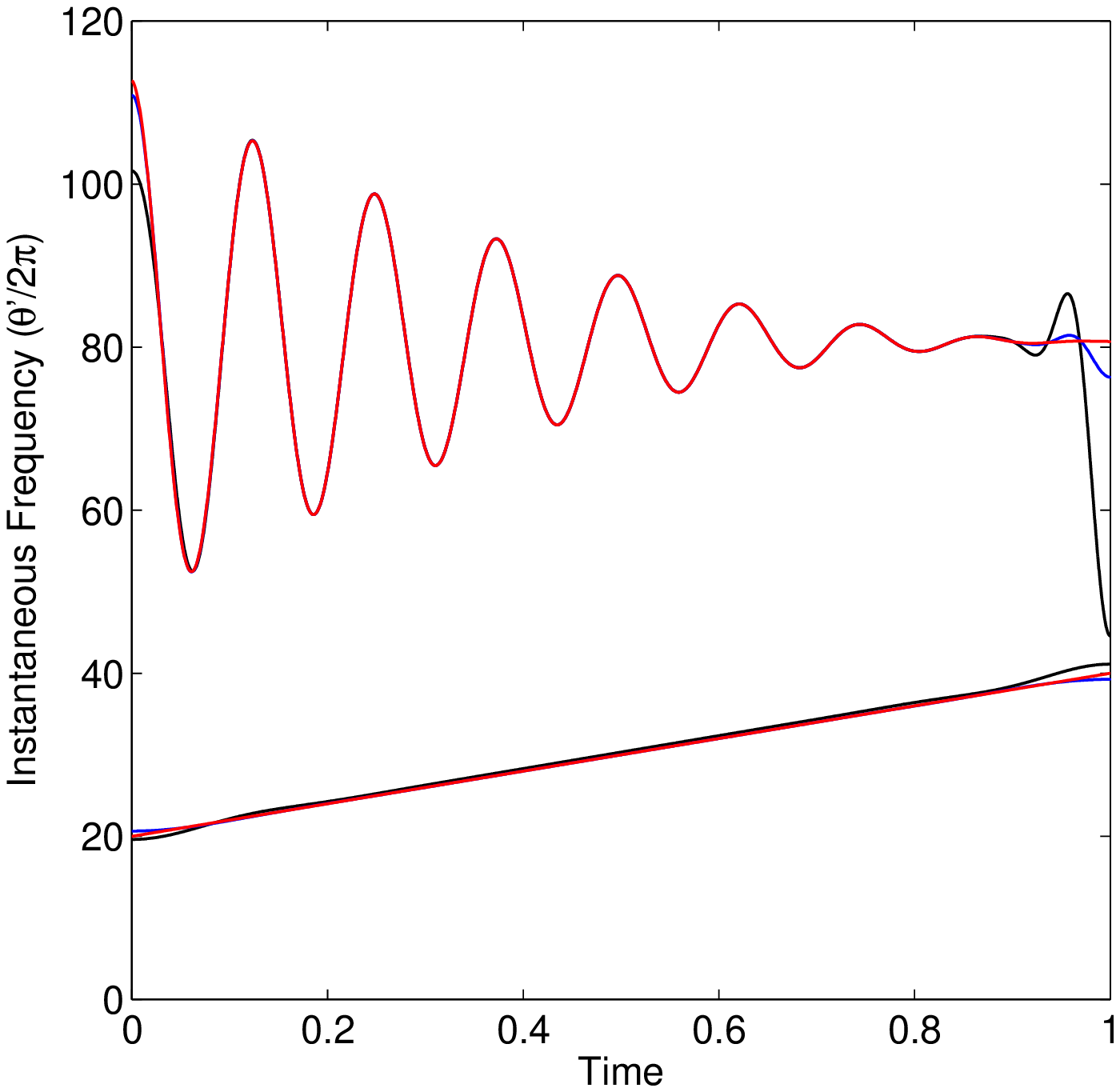}
     \end{center}
     \caption{\label{chirp-global} IMF (left) and Instantaneous frequency (right) of the signal
in \myref{chirp-nonperiodic} obtained from different methods. Red: exact;
Blue: $l^1$ regularized nonlinear matching pursuit; Black: FFT-based algorithm.}
\end{figure}

\subsubsection{Numerical results for data with incomplete or sparse samples}

The $l_1$ regularized nonlinear matching pursuit can also handle the 
incomplete data and the data with sparse samples. We illustrate this
property of our method through a few examples.

The first example is an incomplete signal given by
\myref{incomplete-small}.
\begin{eqnarray}
\label{incomplete-small}
  \theta(t)=120\pi t +10\cos(4\pi t),\;
 a(t)=2+\cos(2\pi t),\;
f(t)=a(t)\cos\theta(t), \; t\in [0,0.4]\cup [0.6,1].
\end{eqnarray}
For this signal, we have only eighty percent of the original data and miss
twenty percent of the data in the gap interval $[0.4,0.6]$. In 
Fig. \ref{data-incomplete}, we plot the recovered signal in the gap
interval $[0.4,0.6]$ (see the middle panel). The recovered signal 
matches the original signal almost perfectly in the gap interval. 
The recovered instantaneous frequency also matches the exact instantaneous
frequency with high accuracy.

In Fig. \ref{data-incomplete-large}, we perform the same numerical 
experiment by enlarging the interval of missing data from 
$(0.4,0.6)$ to $(0.3, 0.7)$, i.e. we miss forty percent of the data.
Even for this more challenging example, our method still gives quite 
reasonable reconstruction of the original data in the region of missing 
data. The recovered instantaneous frequency still approximates the 
exact instantaneous frequency with reasonable accuracy, 
especially away from the region of missing data.

\begin{figure}
    \begin{center}
	\includegraphics[width=0.3\textwidth]{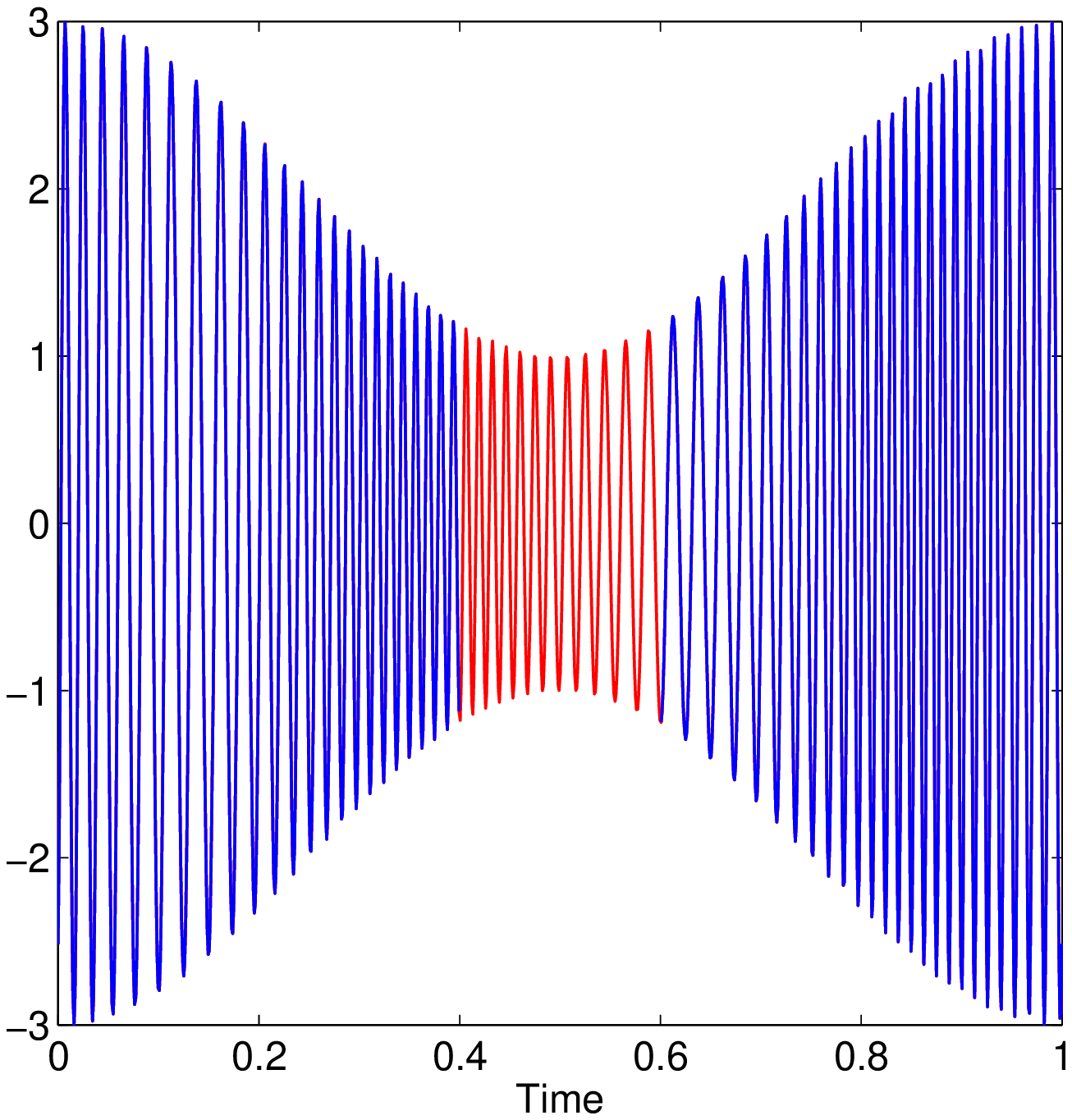}
	\includegraphics[width=0.3\textwidth]{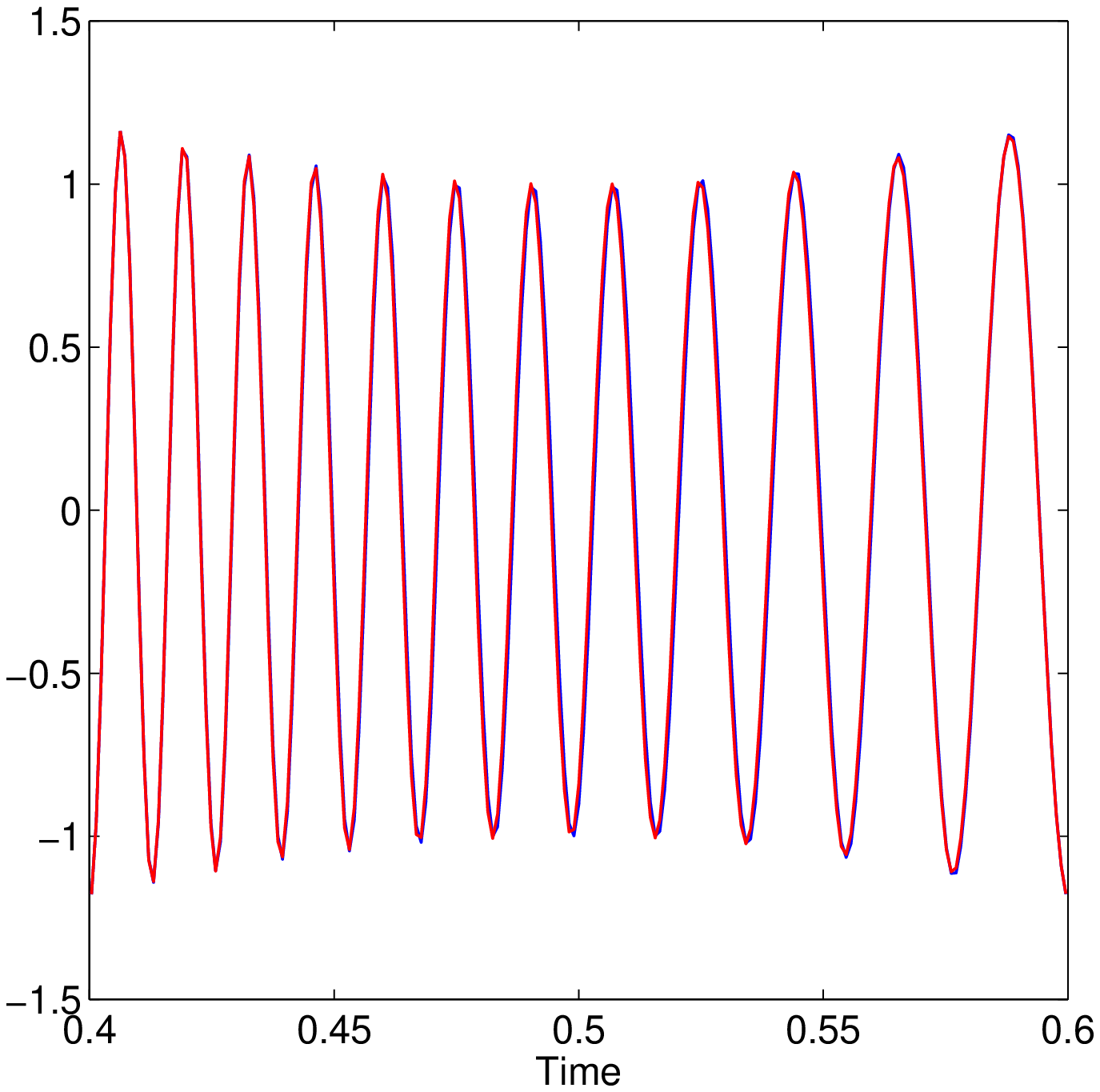}
	\includegraphics[width=0.3\textwidth]{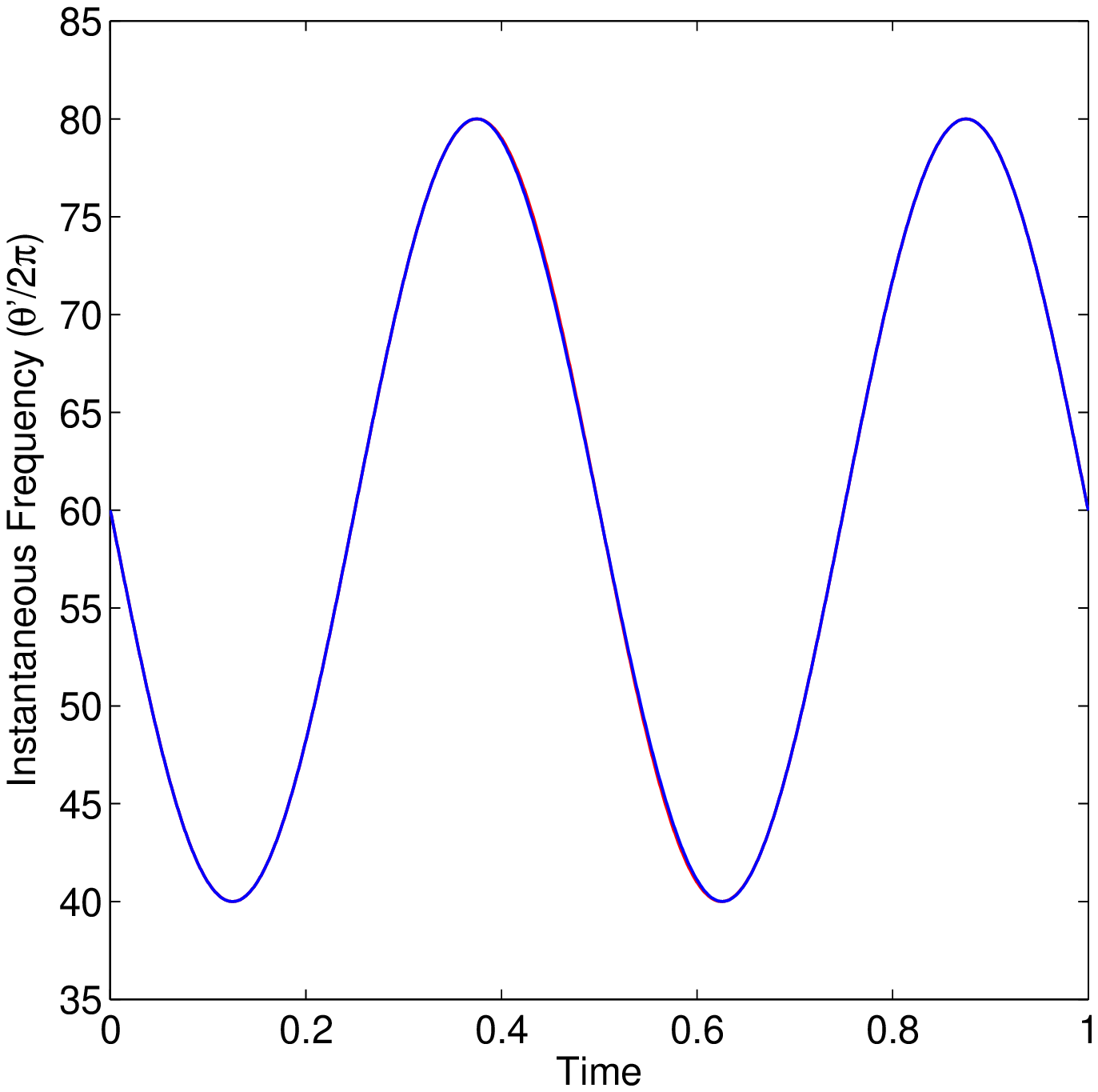}
     \end{center}
     \caption{\label{data-incomplete} Left: blue: the original incomplete data, the gap is $(0.4,0.6)$; red: the
missing data recovered by our method; Middle: the recovered missing data, red: exact; blue: numerical.
Right: the instantaneous frequencies, red: exact; blue: numerical.}
\end{figure}
\begin{figure}
    \begin{center}
	\includegraphics[width=0.3\textwidth]{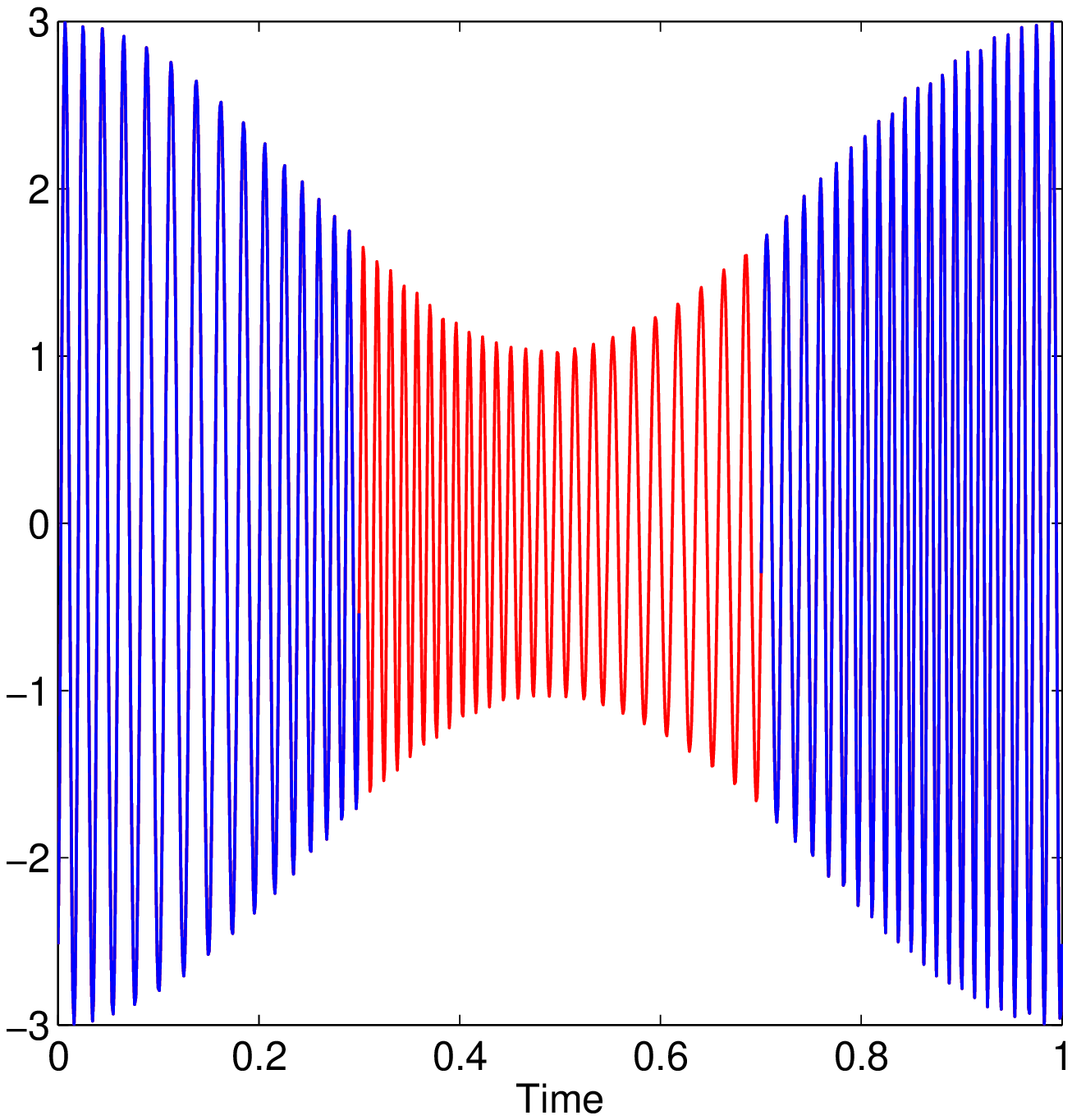}
	\includegraphics[width=0.3\textwidth]{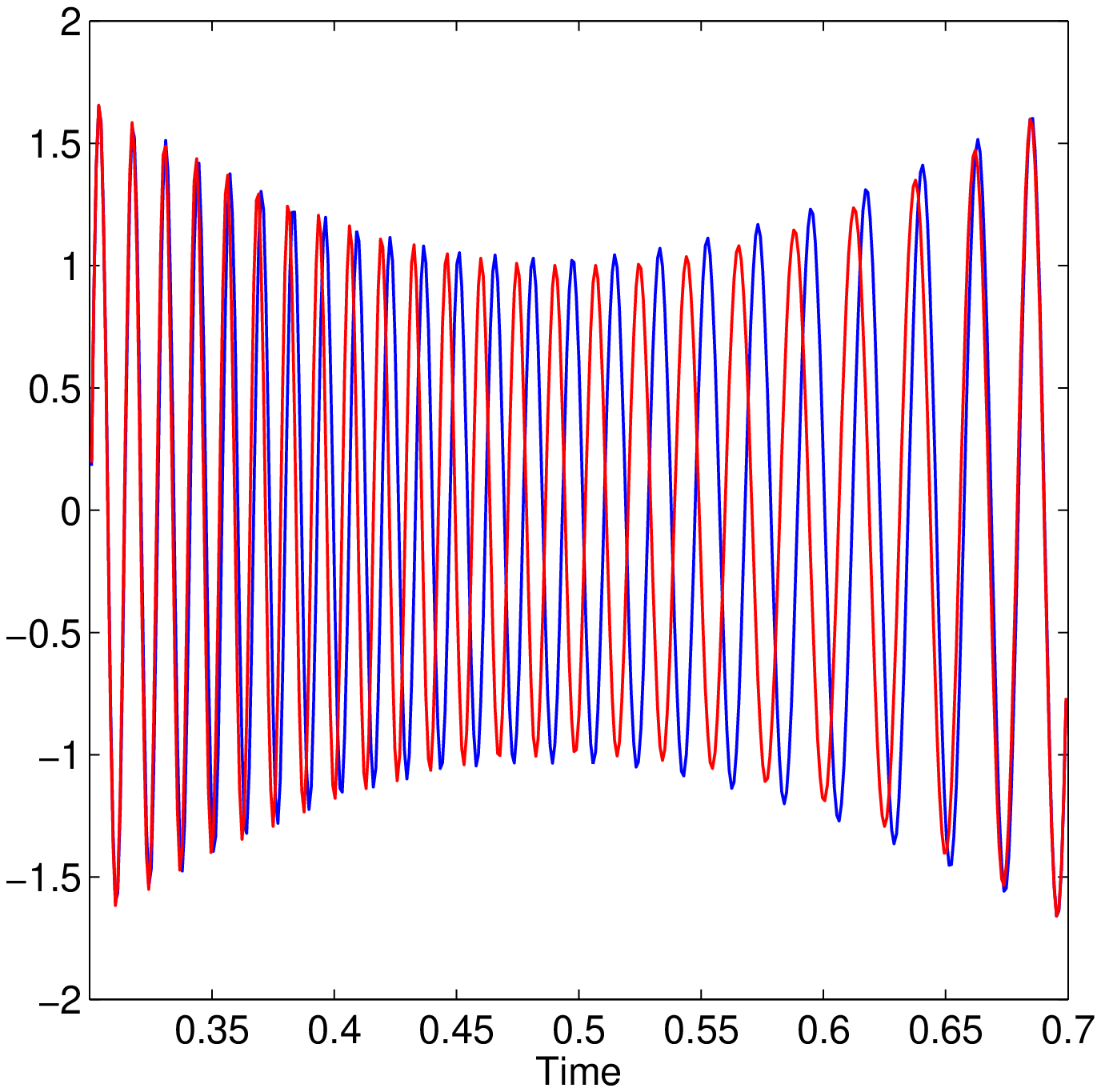}
	\includegraphics[width=0.3\textwidth]{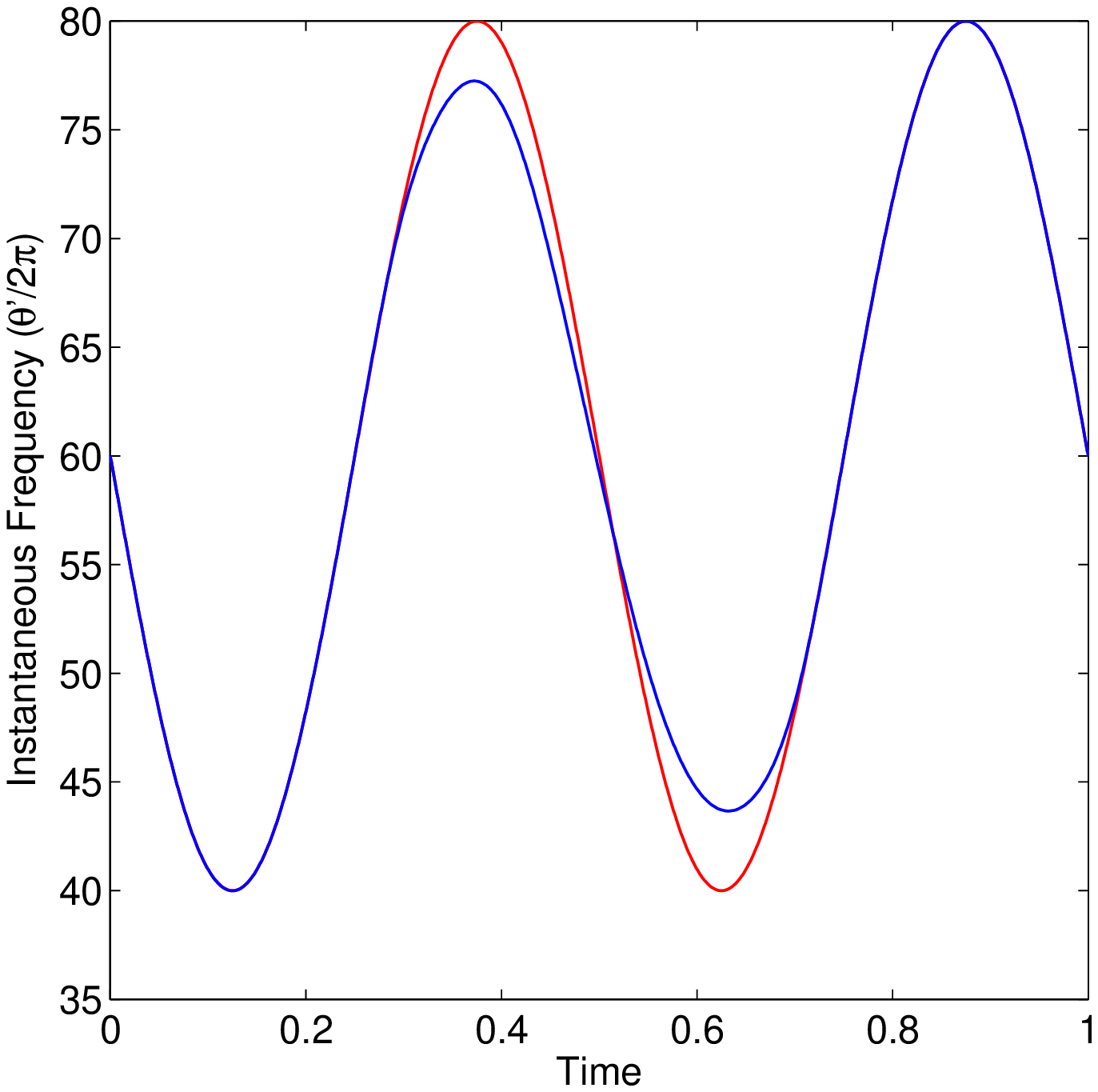}
     \end{center}
     \caption{\label{data-incomplete-large}  Left: blue: the original incomplete data, the gap is $(0.3,0.7)$; red:
the missing data recovered by our method; Middle: the recovered missing data, red: exact; blue: numerical.
Right: the instantaneous frequencies, red: exact; blue: numerical.}
\end{figure}

Finally, we consider an example with insufficient samples. The signal 
is generated as follows: 
\begin{eqnarray}
\label{data-sparse-formula}
  \theta(t_i)=120\pi t_i +10\cos(2\pi t_i),\quad
 a(t_i)=2+\cos(2\pi t_i),\quad
f(t_i)=a(t_i)\cos\theta(t_i),\; t_i\in [0,1].
\end{eqnarray}
and $i=1,2,\cdots, N$. The location $t_i$ is chosen randomly in $[0,1]$.
In this example, the number of samples is $64$. This means that we have 
about one sample point within one period of the signal on average. 

Fig. \ref{data-sparse} gives the results obtained by our method. In 
the case of insufficient samples without noise, the recovered signal 
and the original signal are almost indistinguishable.
\begin{figure}

    \begin{center}
	\includegraphics[width=0.45\textwidth]{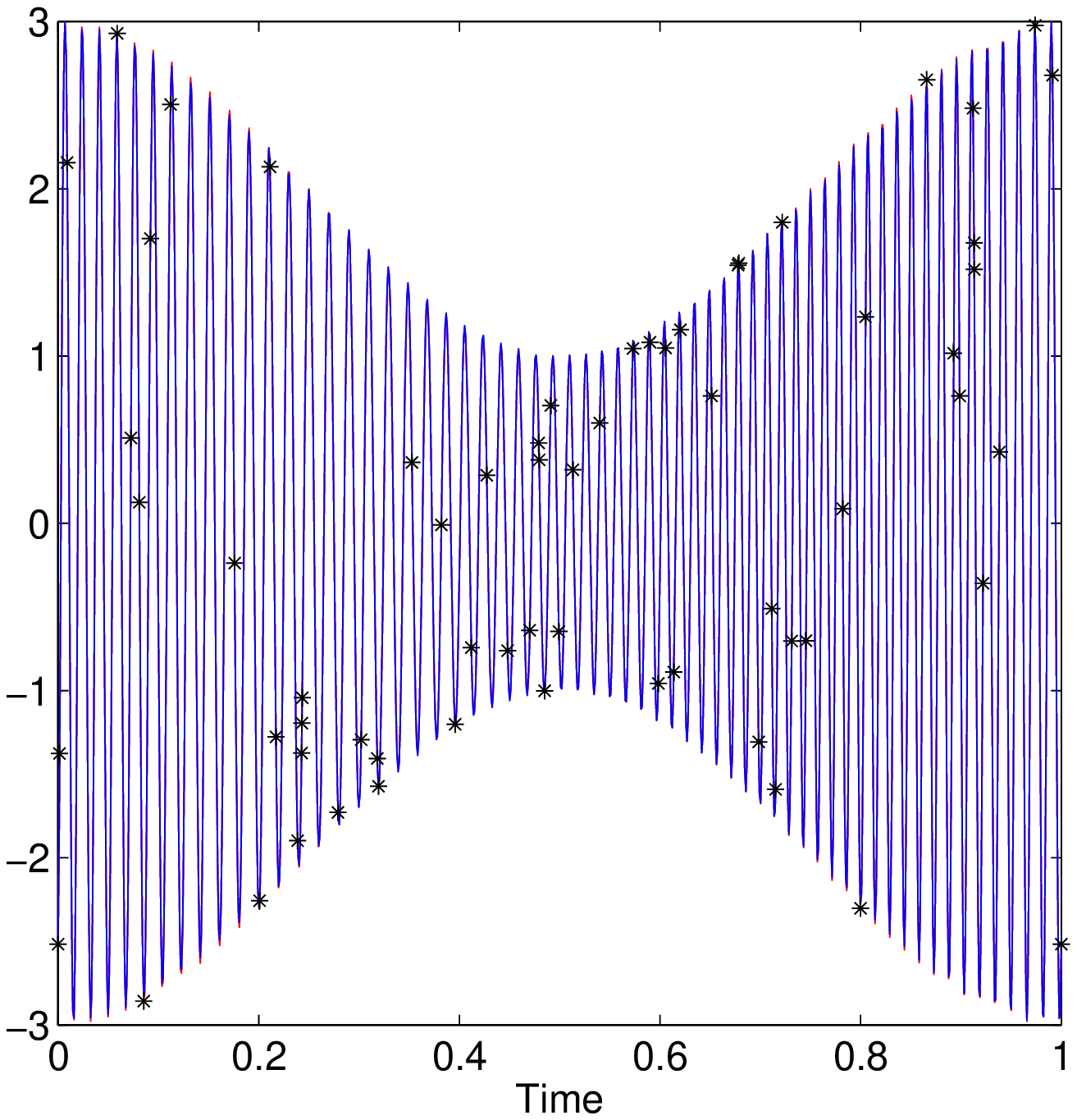}
	\includegraphics[width=0.45\textwidth]{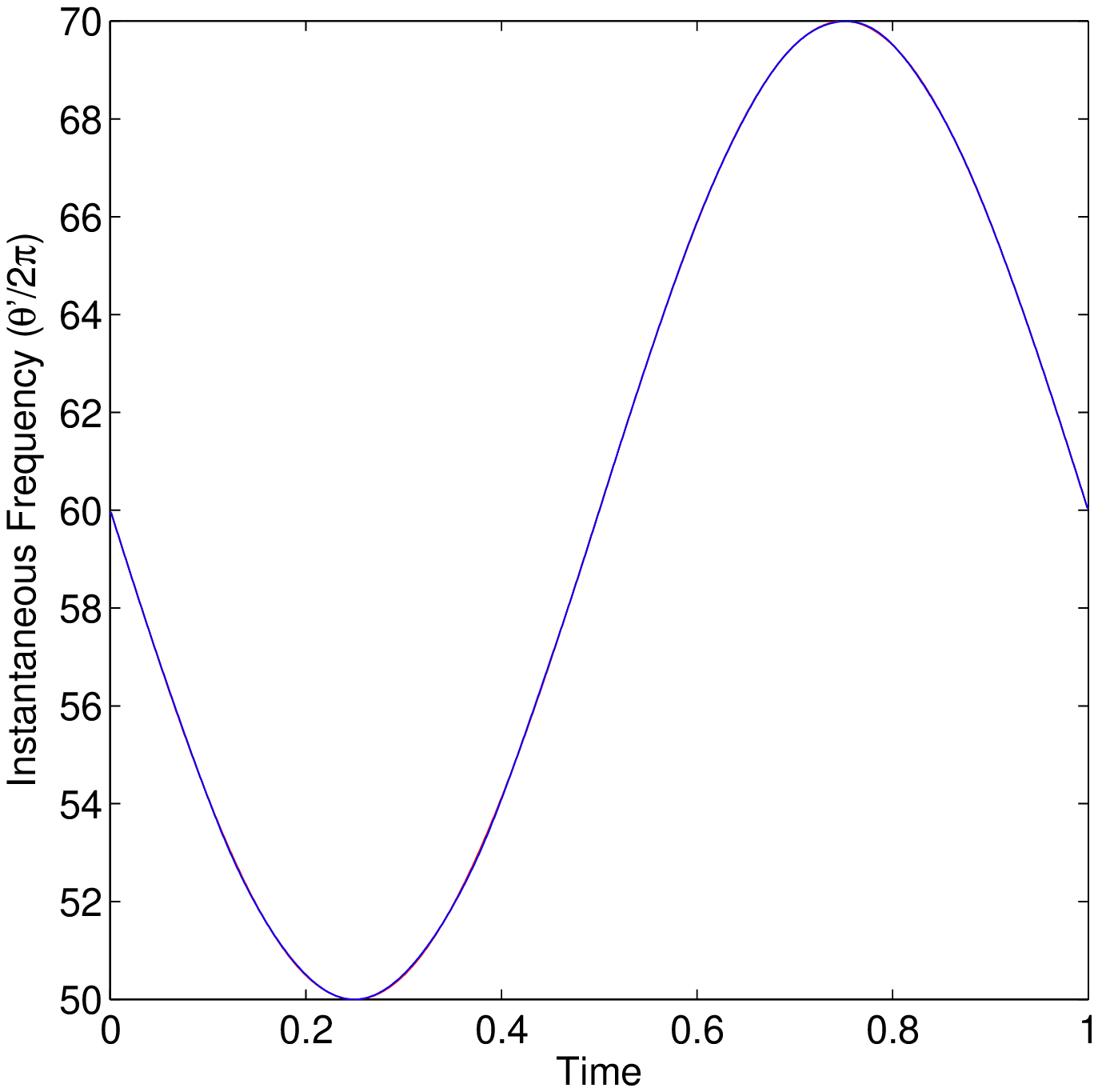}
     \end{center}
     \caption{\label{data-sparse} Left: original samples, red: exact; blue: recovered; '*' represent the sample points.
Right: instantaneous frequency, red: exact; blue: numerical.}
\end{figure}

We now add Gaussian noise to the original samples and apply our method 
to this noisy data. The result is given in Fig. \ref{data-sparse-noise}. 
In this case, the noise $0.2X(t)$ is added to the origianl signal $f(t)$ given 
in \myref{data-sparse-formula}.
We can see that both the recovered signal and the instantaneous frequency 
still have reasonable accuracy. This shows that our method is stable 
with respect to noise perturbation even for sparse under-sampled data.
\begin{figure}
    \begin{center}
	\includegraphics[width=0.45\textwidth]{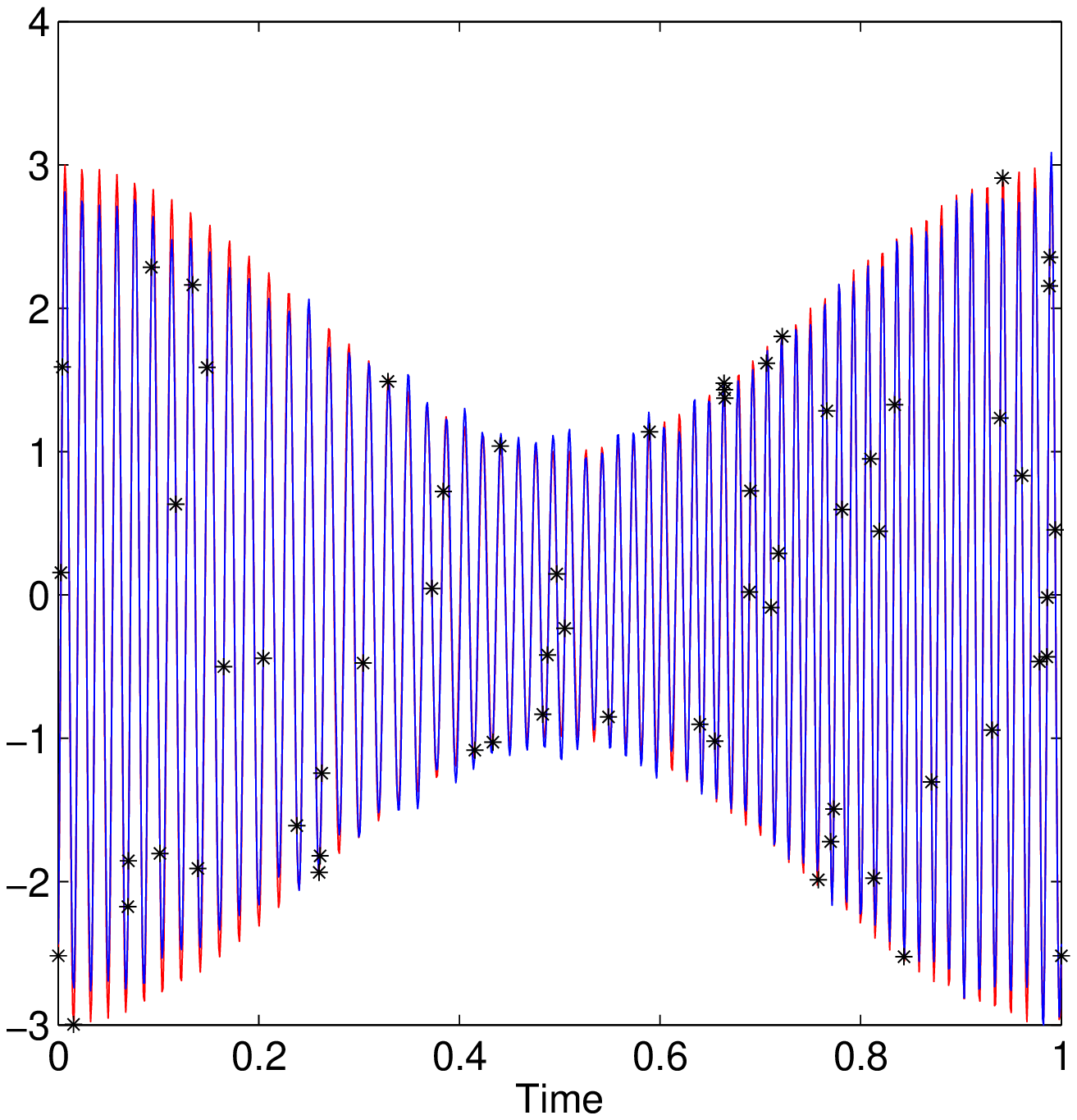}
	\includegraphics[width=0.45\textwidth]{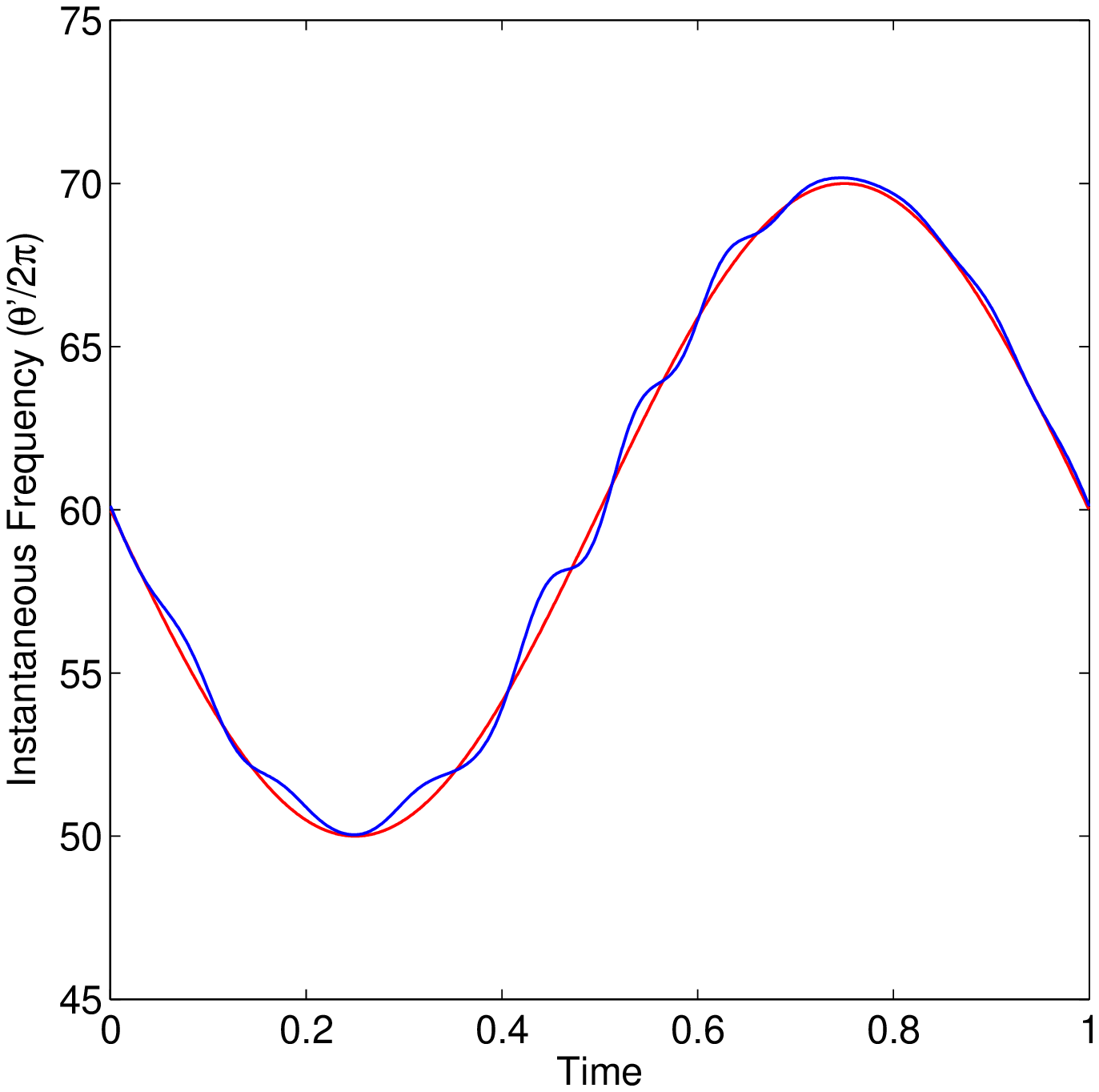}
     \end{center}
     \caption{\label{data-sparse-noise}Left: original samples, red: exact; blue: recovered from the noised data, 
$f(t_i)+0.2X(t_i)$; 
'*' represent the sample points.
Right: instantaneous frequencies, red: exact; blue: numerical.}
\end{figure}

In the future,
we plan to perform some theoretical study of our method for data with 
incomplete or sparse samples and carry out more numerical experiments 
for some complicated or real data.

\section{Generalizations for the $l^1$ regularized nonlinear matching pursuit}

The iterative algorithm based on $l^1$ regularized nonlinear matching 
pursuit can also be generalized to deal with more complicated data, such 
as the data with poor scale-separation and the data with intra-wave 
frequency modulation. In this section, the power of this kind of 
algorithm is shown through several numerical examples, more details can 
be presented in our subsequent papers.

\subsection{Numerical results for data with poor scale-separation property}

In the previous sections, we show that for data with a good 
scale-separation property, the algorithm based on $l^1$ regularized 
nonlinear matching 
pursuit can give an accurate decomposition. Now, we give a brief 
discussion how to decompose data with poor scale-separation property.

Let $f$ be a signal that has the
following sparse decomposition over dictionary $\mathcal{D}$:
\begin{eqnarray}
  f(t)=\sum_{k=1}^M a_k\cos\theta_k,\quad a_k\cos\theta_k\in \mathcal{D},
\end{eqnarray}
where $\mathcal{D}$ is defined in \myref{dic-D}. But now the 
instantaneous frequencies $\theta'_k(t)$ are not well separated, 
so $f(t)$ does not satisfy the scale-separation condition defined
in Section \ref{analysis}.

It is well known that for data consisting of components with 
interfering frequencies, the matching pursuit with a Gabor
dictionary may not give a sparse decomposition \cite{Mallat09}. 
Since our method is based on the matching pursuit, it is not 
surprising that it may not be able to generate the sparsest 
decomposition either.

To illustrate, we consider the following signal consisting of 
two IMFs whose instantaneous frequencies intersect each other.
 The signal is generated by the analytical formula given below.
\begin{eqnarray}
\label{data-cross}
f(t)=\cos(20\pi t+40\pi t^2+\sin(2\pi t))+\cos(40\pi t) .
\end{eqnarray}

\begin{figure}

    \begin{center}

\includegraphics[width=0.3\textwidth]{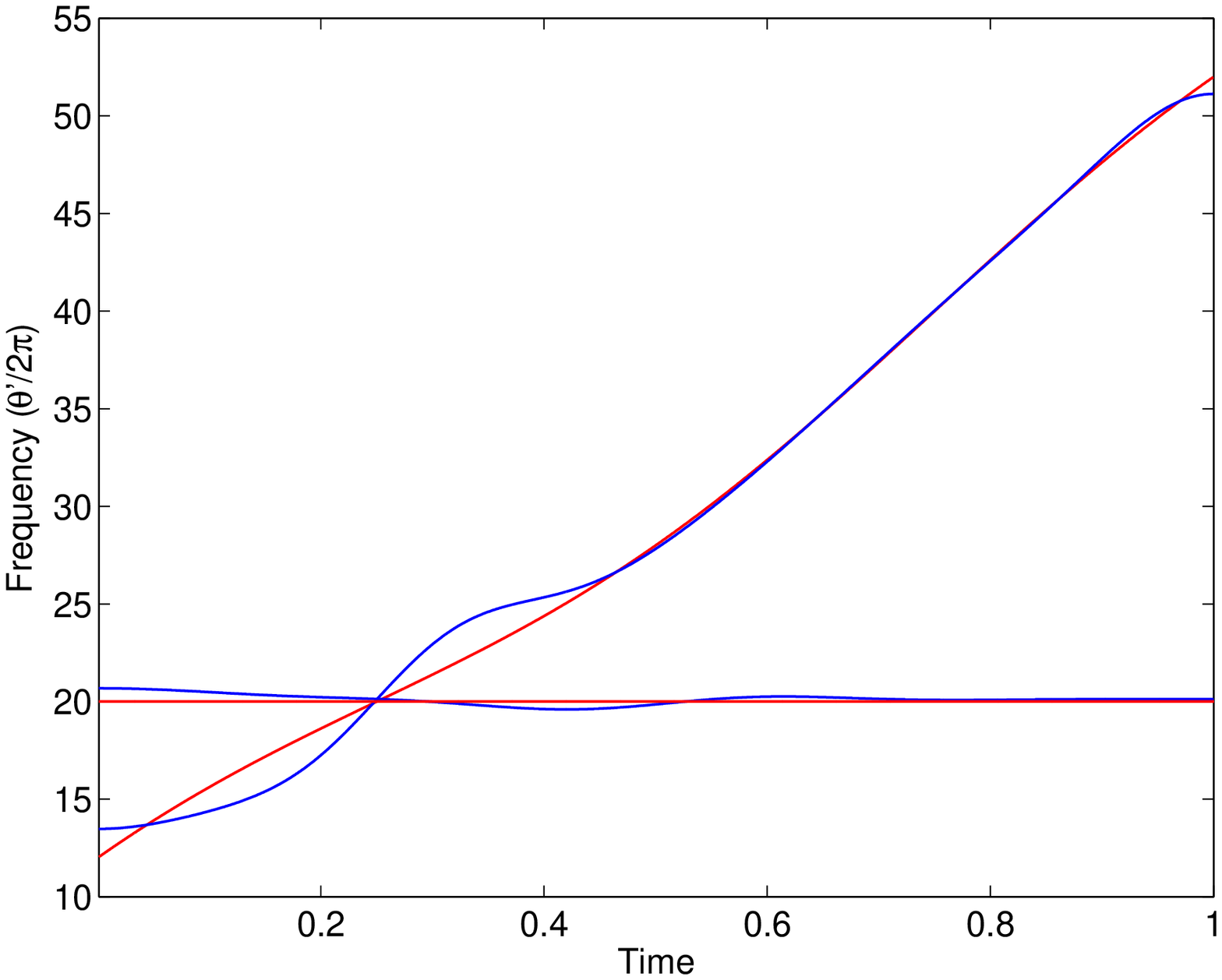}
\includegraphics[width=0.3\textwidth]{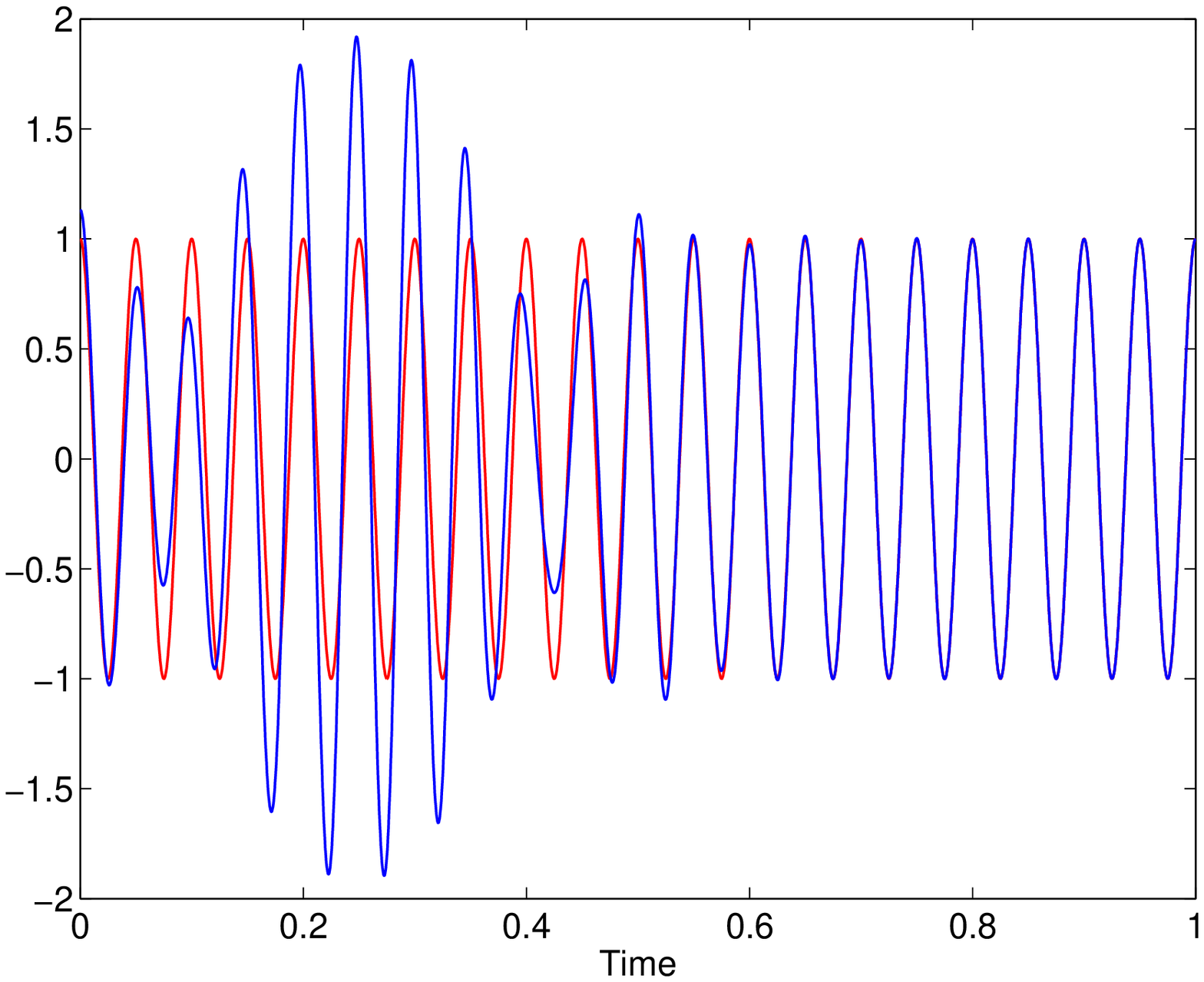}
\includegraphics[width=0.3\textwidth]{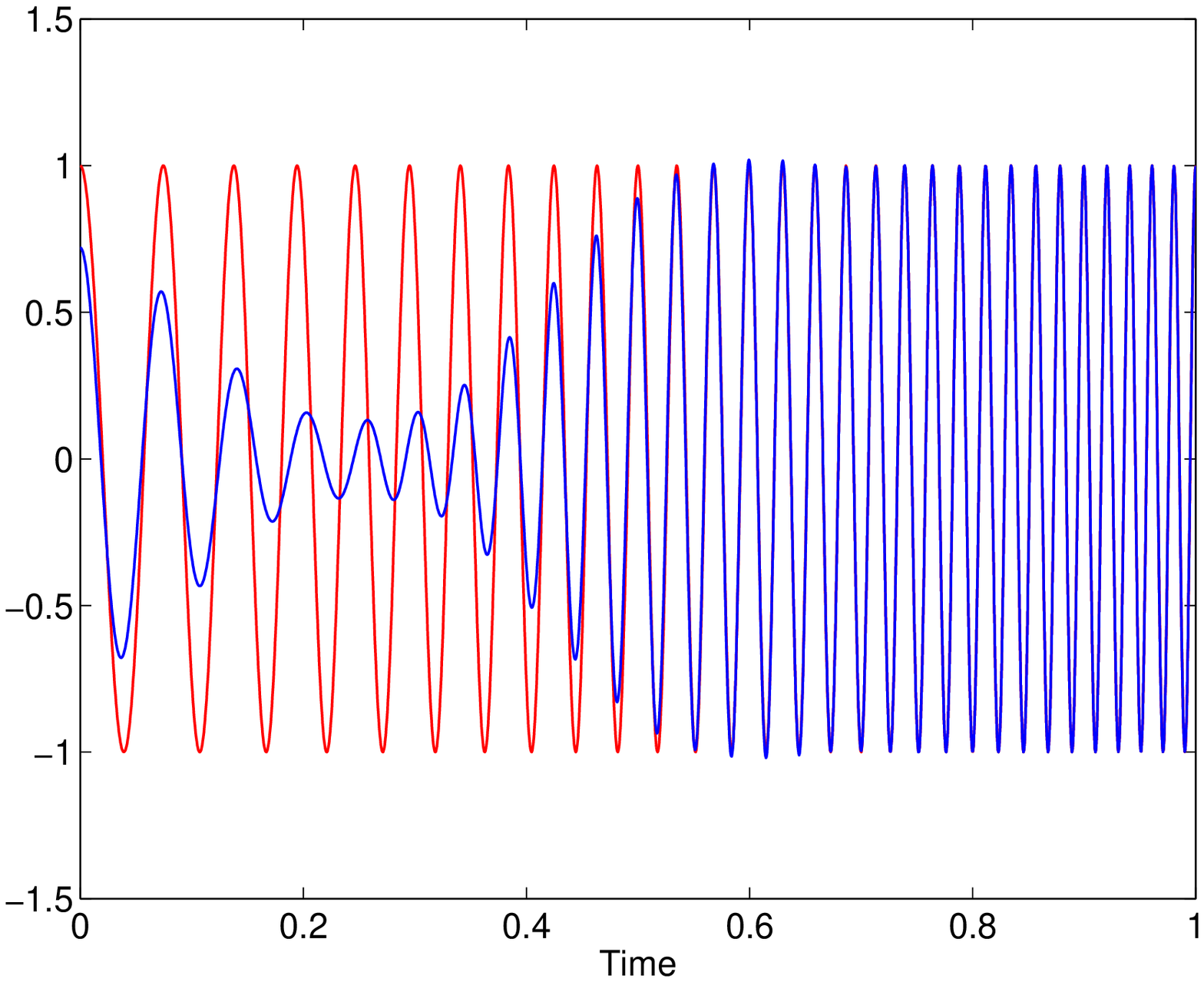}

     \end{center}
    \caption{ \label{cross-IF}Left: Instantaneous frequencies;
red: exact frequencies; blue: numerical results. Middle and Right: IMFs extracted by the previous nonlinear matching pursuit.}
\end{figure}

Fig. \ref{cross-IF} plots the instantaneous frequencies and IMFs 
recovered by the nonlinear matching pursuit given in the previous section. Near the point of
intersection, both the instantaneous frequencies and IMFs produce 
noticeable errors. The good news is that the instantaneous frequency
recovered by our method is still in phase with the exact one. 
Furthermore, the accuracy is quite reasonable
in the region far away from the point of intersection.
This shows that our method has
a temporal locality property, which is important in many
physical applications.

To further improve the accuracy of our decomposition when there are
a number of instantaneous frequencies that are not well separated, we 
need to decompose these components simultaneously since these IMFs 
have strong correlation. Assume that we can learn from our 
$l^1$ regularized nonlinear matching pursuit that there are $M$ 
components of IMFs whose instantaneous frequencies are not well 
separated, we modify our decomposition method to solve the following
optimization problem:
\begin{eqnarray}
  \label{opt-mul}
  \min \left (\gamma\sum_{k=1}^M \|\widehat{a}_k\|_{l^1}+\|f(t)-\sum_{k=1}^M a_k\cos\theta_k\|_{l^2}^2\right ),\quad s.t. \quad a_k\cos\theta_k\in \mathcal{D},
\end{eqnarray}
where $\gamma >0$ is a regularized parameter and $\widehat{a}_k$ is the representation of $a_k$ in $V(\theta_k)$ space.

This problem is much more difficult to solve than the original one, 
since the different components may have strong correlation. 
Based on the $l^1$ regularized nonlinear matching pursuit that we 
introduced in the previous sections, we have developed a new 
method to solve the above optimization problem. The detail 
of this method will be reported in our subsequent paper. 
Here we give an example to demonstrate that this new method has
the capability to deal with the signal with poor scale separation.

Fig \ref{cross-IF-mul} gives the results obtained by our new method 
for the signal given in \myref{data-cross}. We can see that both 
the instantaneous frequencies and IMFs match the exact ones 
pretty well. These results are much better
than those given in Fig. \ref{cross-IF}.
\begin{figure}

    \begin{center}

\includegraphics[width=0.3\textwidth]{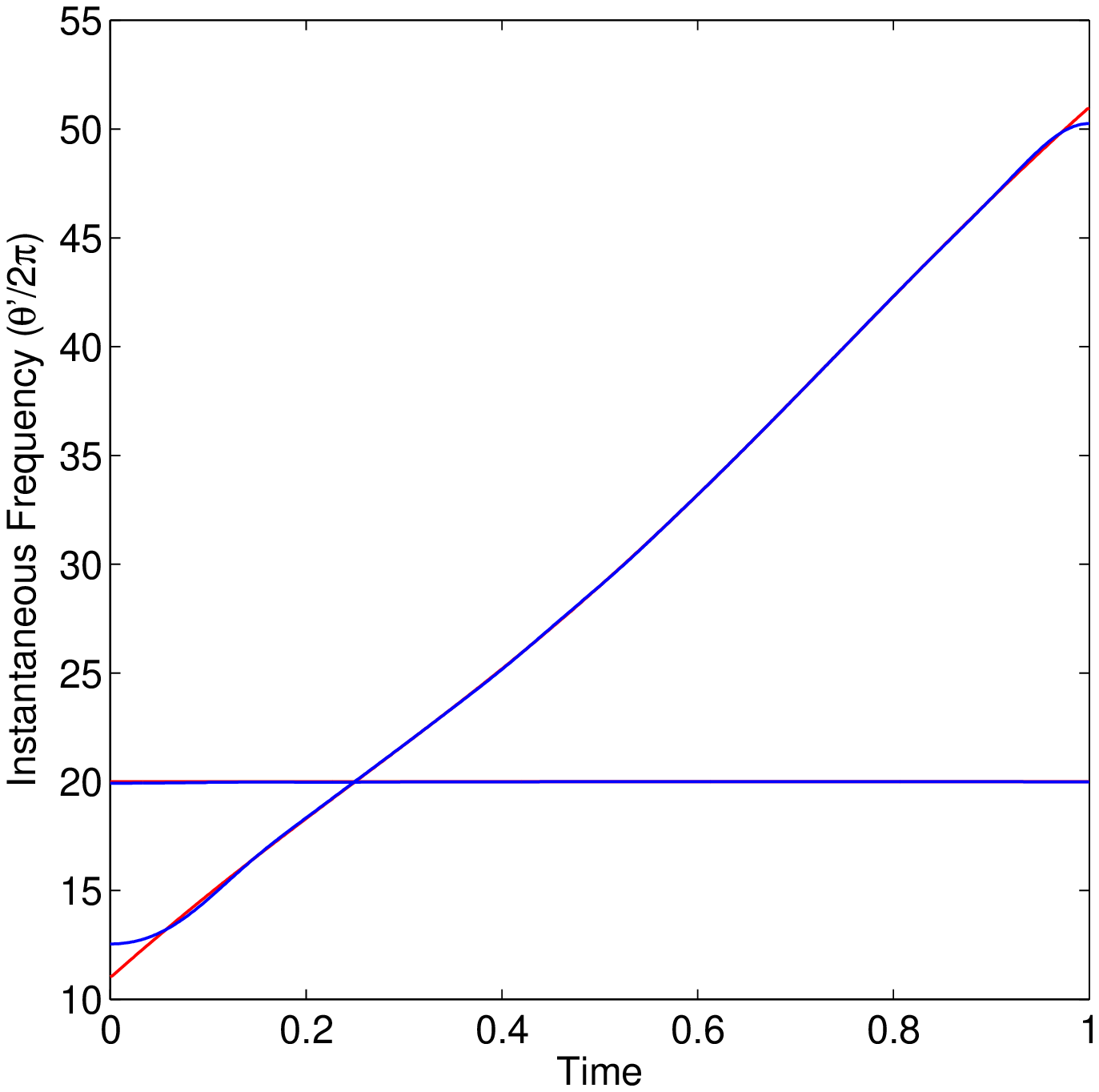}
\includegraphics[width=0.3\textwidth]{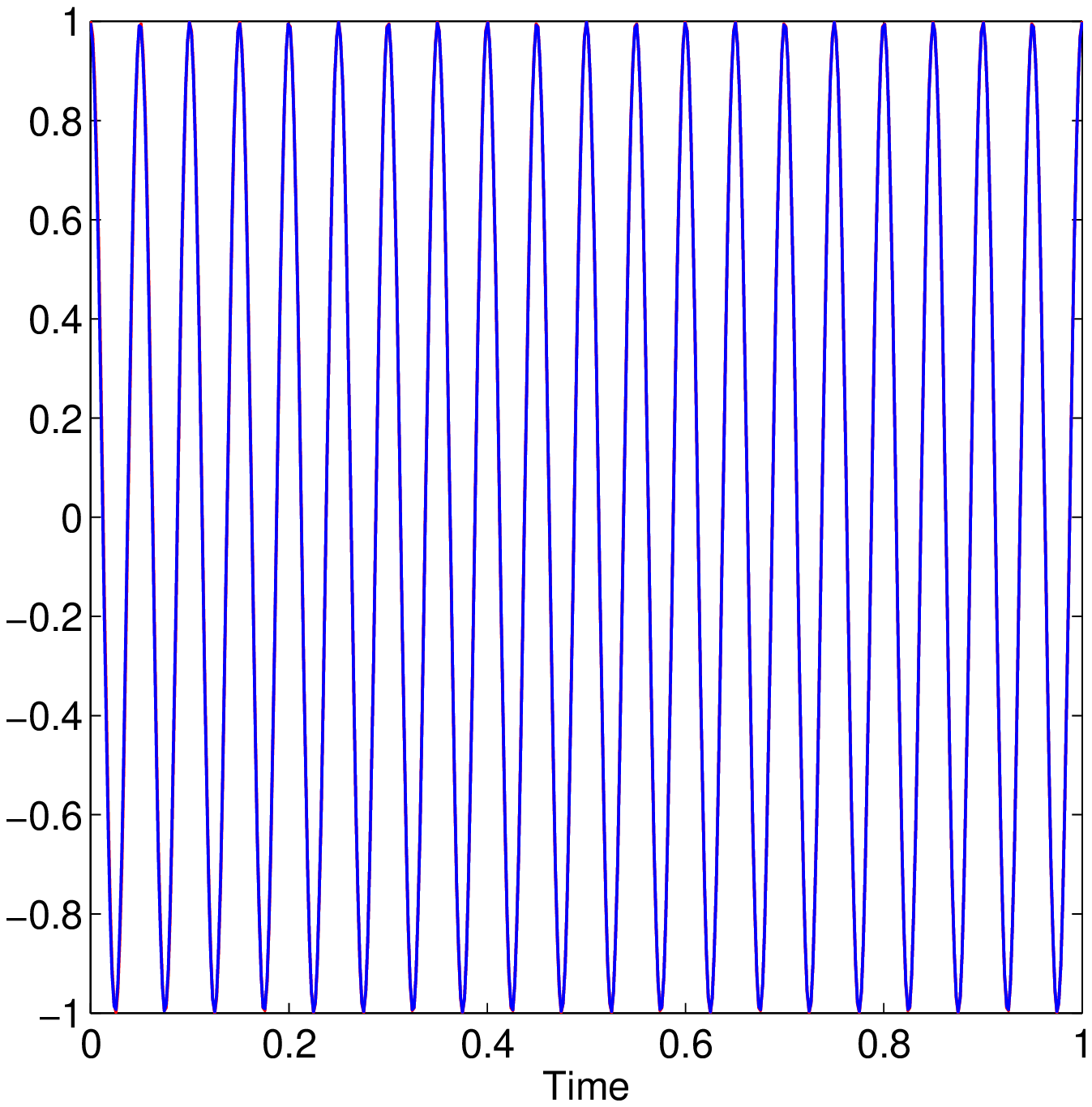}
\includegraphics[width=0.3\textwidth]{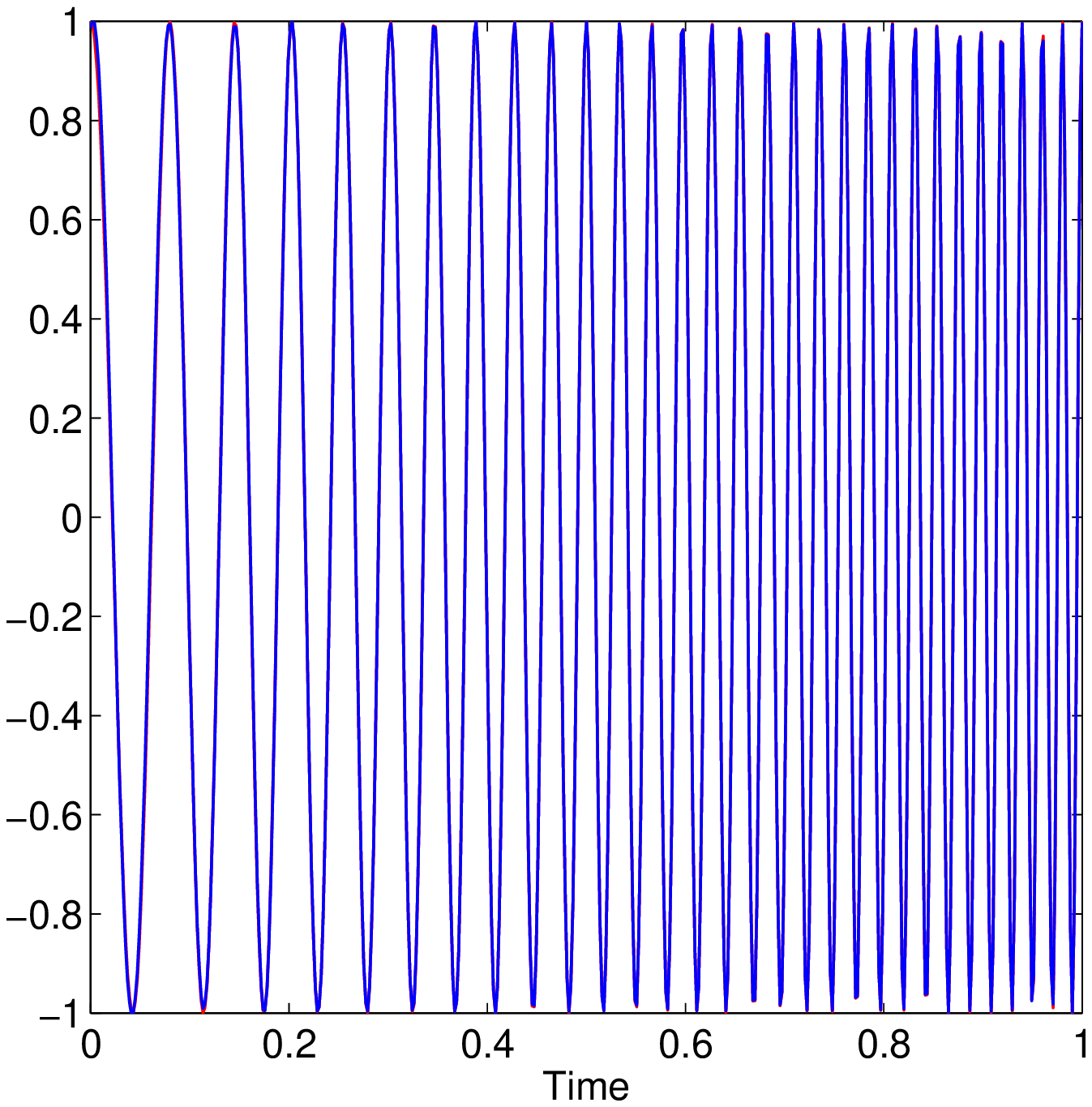}

     \end{center}
    \caption{ \label{cross-IF-mul} Left: Instantaneous frequencies; Middle (first component) and right (second component): IMFs obtained by 
extracting two IMFs together.
red: exact results; blue: numerical results.}
\end{figure}

In the previous example we consider, the scale separation is destroyed 
near the region where instantaneous frequencies of two different IMFs
intersect each other. In the following example, we consider another 
example with poor scale separation, but for a different reason. 
In this case, the frequencies of different IMFs are well separated, 
but the instantaneous frequency of an IMF is not smooth any more.
Specifically, we consider the following example:
\begin{eqnarray}
\label{data-jump}
&&f=\frac{1}{1.5+\cos 2\pi t}+\cos\theta_1+\cos\theta_2,\quad \theta_1(0)=\theta_2(0)=0,\\
&&  \theta'_1=\left\{\begin{array}{cc} 40\pi,& 0\le t\le 0.3,\\
60\pi,& 0.3<t\le 1.\end{array}\right.,\quad\quad
  \theta'_2=\left\{\begin{array}{cc} 140\pi,& 0\le t\le 0.6,\\
160\pi,& 0.6<t\le 1.\end{array}\right.\nonumber.
\end{eqnarray}
In this example, the instantaneous frequencies are discontinuous at 
$t=0.3$ and $t=0.6$ respectively. Even for such signal, our $l^1$ 
regularized nonlinear matching pursuit method still gives a reasonable
decomposition even if the signal is polluted with noise.
\begin{figure}

    \begin{center}

\includegraphics[width=0.45\textwidth]{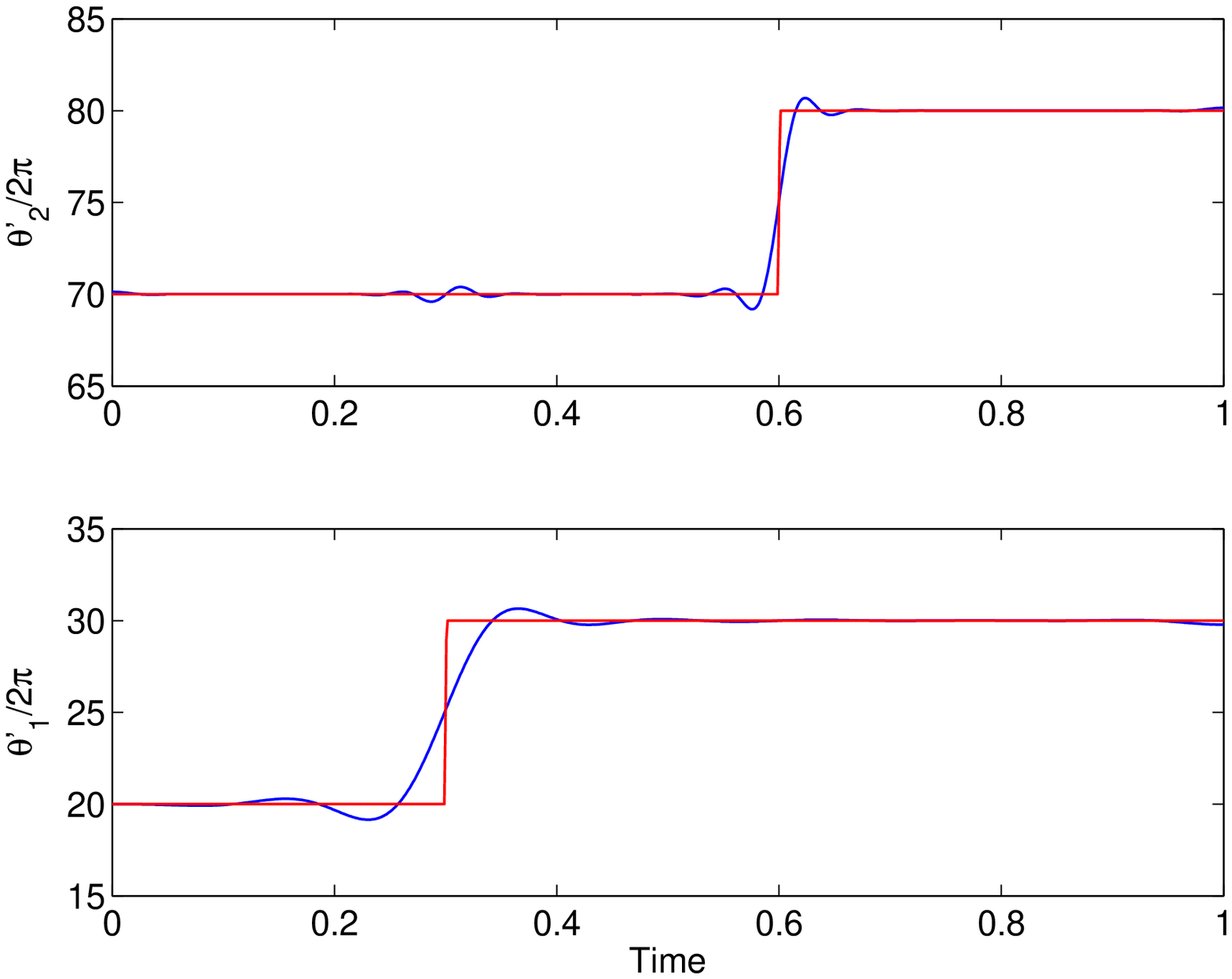}
\includegraphics[width=0.45\textwidth]{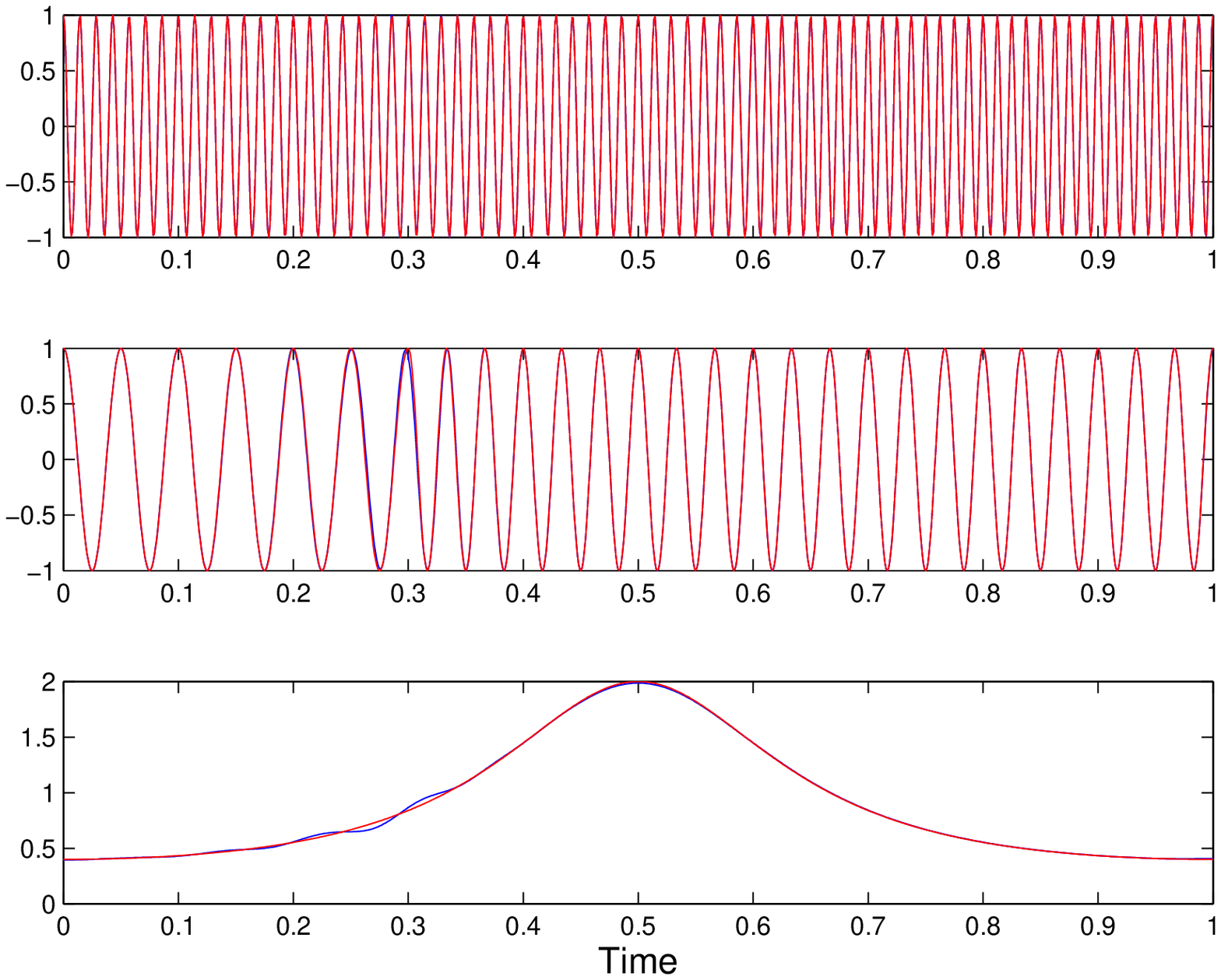}
     \end{center}
    \caption{ \label{result-jump} Instantaneous frequencies and IMFs given by the
$l^1$ regularized nonlinear matching pursuit for the data given in \myref{data-jump}. 
Left: Instantaneous frequencies; Right: IMFs obtained by our method.
red: exact results; blue: numerical results.}
\end{figure}

\begin{figure}

    \begin{center}

\includegraphics[width=0.45\textwidth]{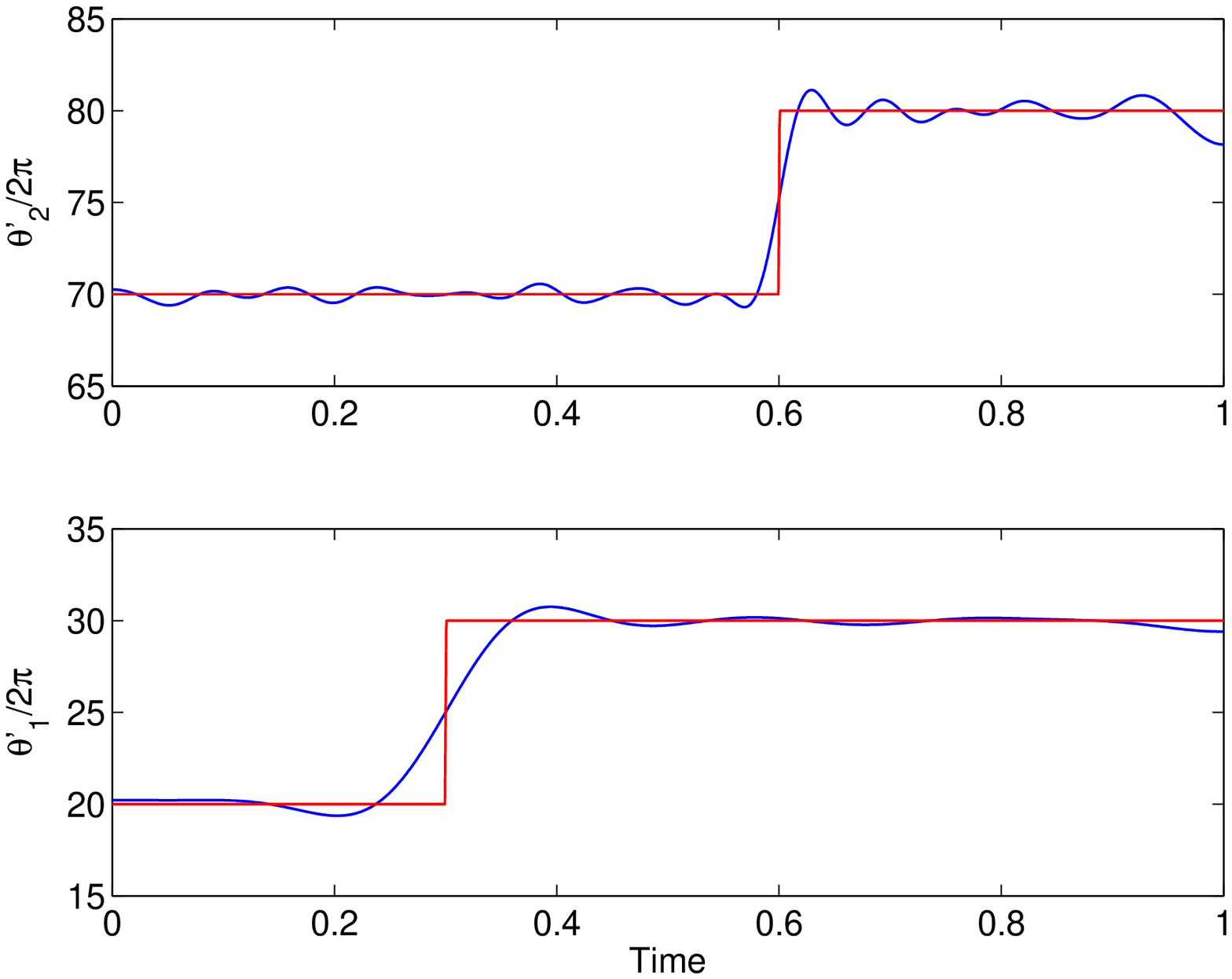}
\includegraphics[width=0.45\textwidth]{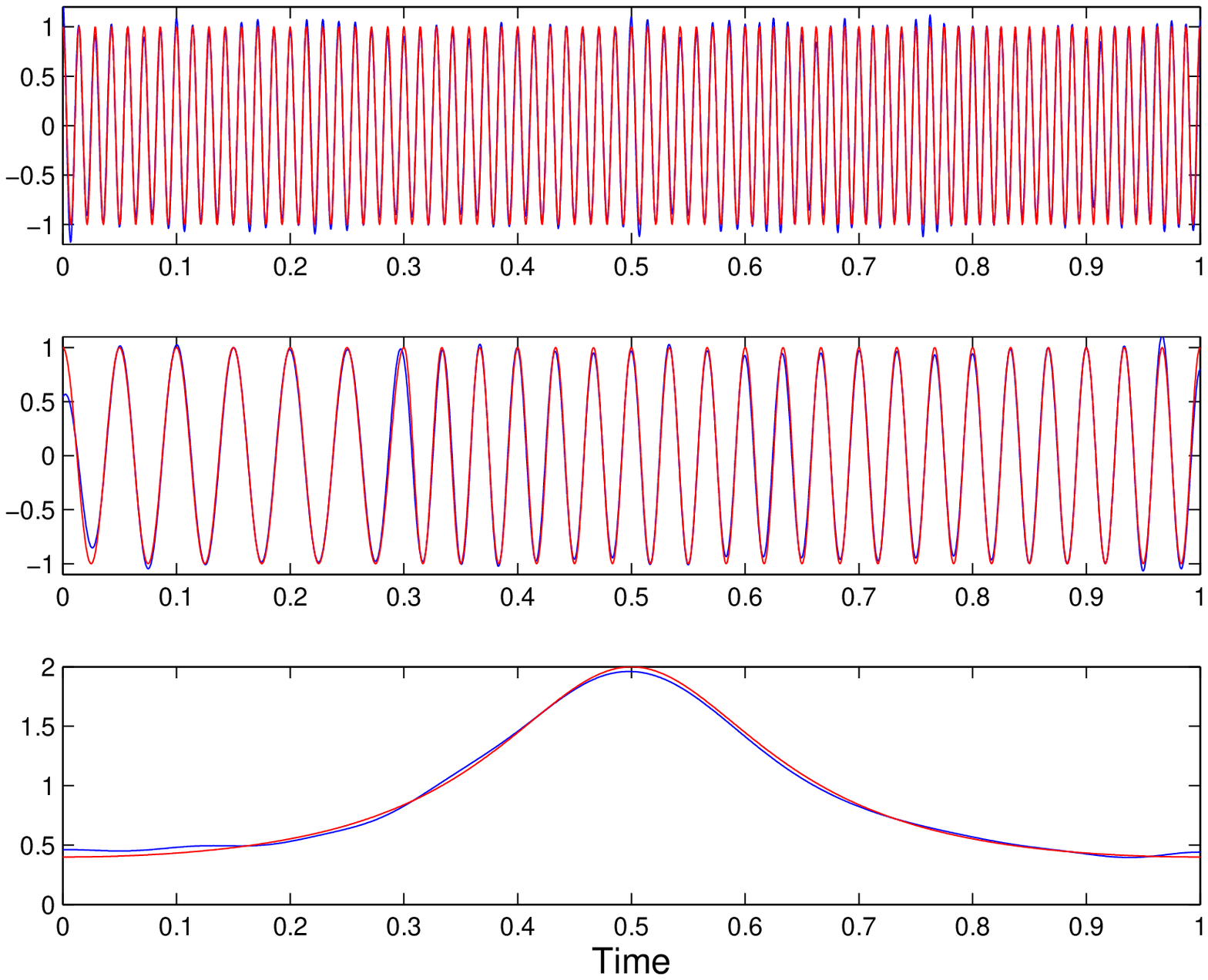}
     \end{center}
    \caption{ \label{result-jump-noise} 
Instantaneous frequencies and IMFs given by 
$l^1$ regularized nonlinear matching pursuit for the data $f(t)+0.3X(t)$, where $f$ is given in \myref{data-jump}
and $X(t)$ is Gaussian noise with standard derivation $\sigma^2=1$.
Left: Instantaneous frequencies; Right: IMFs obtained by our method.
red: exact results; blue: numerical results.}
\end{figure}
As we can see from Fig. \ref{result-jump}, 
there are some oscillations due to the Gibbs phenomenon
near the points of discontinuity $(t=0.3,0.6)$.  IMFs also have some 
errors near these two points. But in the region away from these two 
points, the numerical results match very well with the exact ones. 
These results also suggest that our $l^1$ regularized nonlinear 
matching pursuit has temporal locality property. The error is 
confined in a small region near the points where scale separation 
is poor. This property will be discussed further in Section 6. 
When the noise is added to this signal, the results we obtain still 
have a reasonable accuracy, see Fig. \ref{result-jump-noise}.

\subsection{Numerical results for data with intra-wave frequency modulation}

In some physical applications such as the stokes waves or some
nonlinear dynamic systems, we need to deal with data that have 
intra-wave frequency modulation. For this type of signals, they 
have the sparse decomposition:
\begin{eqnarray}
  f(t)=\sum_{k=1}^M a_k\cos\theta_k,
\end{eqnarray}
where $a_k$ are smooth envelopes, but their instantaneous frequencies, 
$\theta'_k$, are not smooth any more. Typically,
the phase function has the form
$\theta_k=\phi_k+\e \cos\left(\omega_k \phi_k\right)$, where 
$\phi_k$ is a smooth function, $\e>0$ is a small number and 
$\omega_k\ge 1$ is an integer. In the study of some nonlinear waves 
or dynamical systems, $\omega_k$ is an important parameter related to
the characteristic of the nonlinearity of the system.

An essential difficulty for this type of data is that the instantaneous 
frequency, $\theta'_k$, is as oscillatory as $\cos\theta_k$ or even more 
oscillatory. In the 
nonlinear matching pursuit method that we proposed in the sections, we 
assume that $a_k$ and $b_k$ are smoother than $\cos\theta_k$ and use 
these two coefficients to update $\theta_k$. In the case where the 
signal has strong intra-wave modulation, $\theta'_k$ is as oscillatory 
as $\cos\theta_k$. Thus our current method would not be able to give 
a good approximation of $\theta'_k$ unless we are given a very good 
initial guess for $\theta_k$. To overcome this difficulty, we 
introduce a shape function, $s_k$, to replace 
cosine function. The idea is to absorb the high frequency intra-wave
modulation into the shape function $s_k$. This will ensure that 
$\theta'_k$ is smoother than $s_k(\theta_k)$. This idea has been 
proposed by Hau-tieng Wu in \cite{Wu11}, but efficient algorithm to 
compute the shape function has not appear in the literature. 

Note that $s_k$ is not 
known {\it a priori} and is adapted to the signal. We need to learn 
$s_k$ from the physical signal. This consideration naturally motivates 
us to modify the construction of the dictionary as follows:
\begin{eqnarray}
  \label{dic-intra}
  \mathcal{M}=\left\{a_ks_k(\theta_k):\quad a_k, \theta'_k\in V(\theta_k), \;s_k\; \mbox{is $2\pi$-period function}\right\},
\end{eqnarray}
where $s_k$ is an unknown $2\pi$-periodic `shape function' and is 
adapted to the signal. If we choose $s_k$ to be cosine function, 
then $\mathcal{M}=\mathcal{D}$.

We have developed an efficient method based on our $l^1$ regularized 
nonlinear matching pursuit to recover the components with intra-wave 
frequency modulation by looking for the sparsest decomposition 
using this new dictionary $\mathcal{M}$. By exploring the fact that
$s_k$ is a periodic function of $\theta_k$, we can identify certain
low rank structure of the signal. This structure enables us to extract 
the shape function from the signal. Once we get an approximation of 
the shape function $s_k$, we can use $s_k$ to update $\theta_k$.
This process continues until it converges. The detail of 
this method will appear in another paper. Here, we give several 
numerical examples to demonstrate the capability of our method.

The first example is the solution of the Duffing equation.
This is an important example to demonstrate the importance of the 
intra-wave frequency modulation. Our method can handle this signal 
even with noise.

The Duffing equation is a nonlinear ODE which has the following form:
\begin{eqnarray}
  \label{duffing-1}
  \frac{d^2u}{dt^2}+u+\e u^{1+\omega}=\gamma\cos(\beta t).
\end{eqnarray}
The parameters, $\e,\gamma,\omega$, that we use here to generate the 
solution in Fig. \ref{duffing}, are the same as those in
the paper \cite{Huang98}, $\e=-1,\,\gamma=0.1,\,\beta=\frac{1}{25}$
and $\omega = 2$. The initial condition is $u(0)=u'(0)=1$.

In Fig. \ref{duffing}, we plot the shape function that we obtain from 
the solution of the Duffing equation. In this example, we can express
$s_k(\theta_k)$ in terms of $\cos\widetilde{\theta_k}$, from which
we can recover the instantaneous frequency of the signal. More 
interestingly, from the Fourier coefficients of the shape function, 
we can recover the information regarding the nonlinearity
of the Duffing equation, which is $\omega=2$ in this case. 
We refer to a recent work of Prof. Norden Huang for more discussions
of how to extract the degree of nonlinearity using EMD 
(see Dr. Huang's lecture in the IMA Hot Topic Workshop on Trend 
and Instantaneous Frequency, September 7-9, 2011, IMA).

\begin{figure}

    \begin{center}

\includegraphics[width=0.3\textwidth]{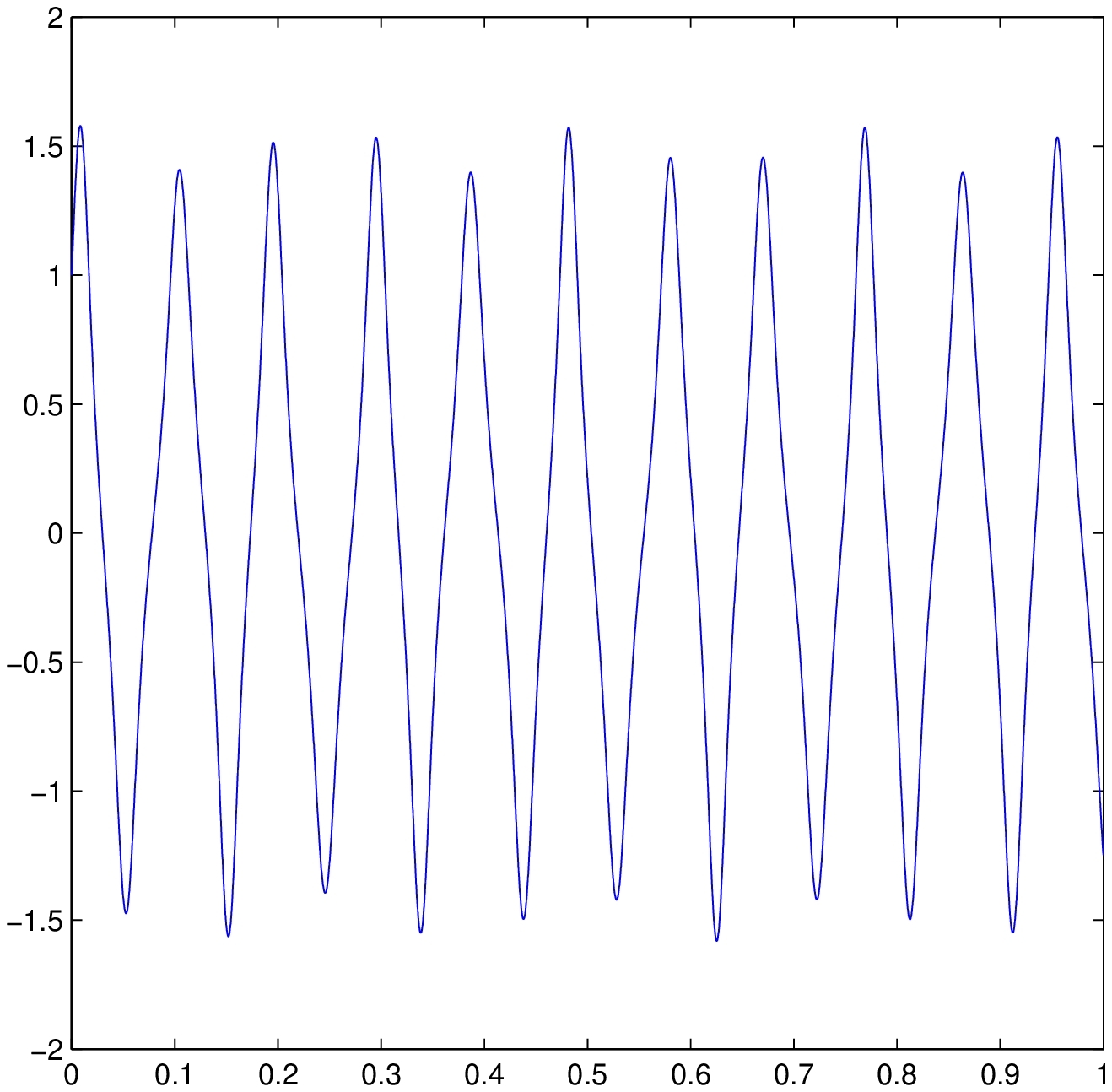}
\includegraphics[width=0.3\textwidth]{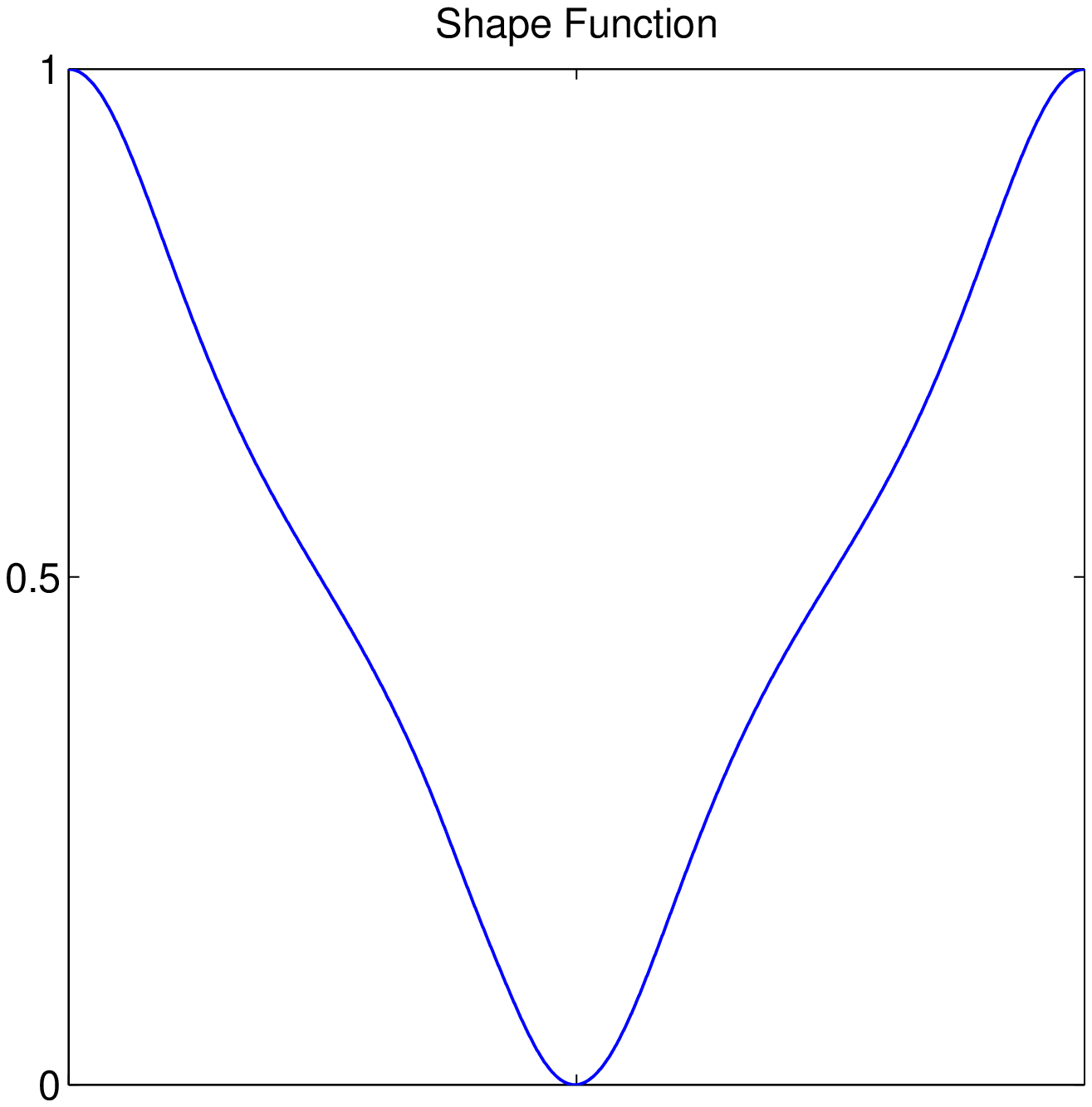}
\includegraphics[width=0.3\textwidth]{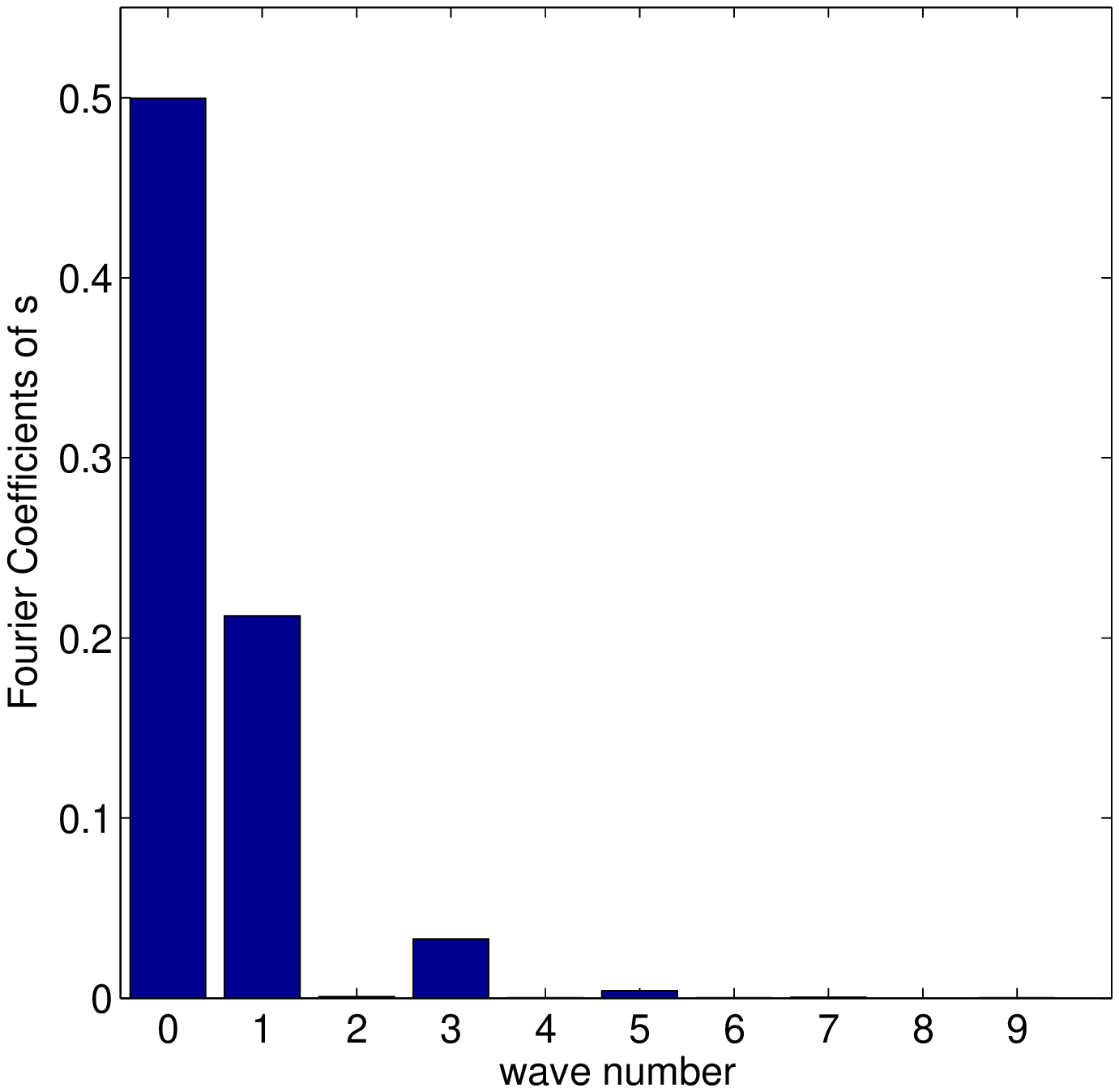}

     \end{center}
    \caption{ \label{duffing}Left: the solution of duffing equation;
Middle: the shape function $s$; Right: the Fourier coefficients of $s$.}
\end{figure}

We also add Gaussian noise $X(t)$ with covariance 
$\sigma^2=1$ to the original solution of the Duffing equation.
Fig. \ref{duffing-noise} shows the corresponding results. We can see 
that the shape function extracted from the noisy signal still keeps 
the main characteristics of the shape function extracted from the 
signal without noise. We can also clearly extract the
degree of nonlinearity, $\omega = 2$, even with such large noise
perturbation to the solution of the Duffing equation.
\begin{figure}

    \begin{center}

\includegraphics[width=0.3\textwidth]{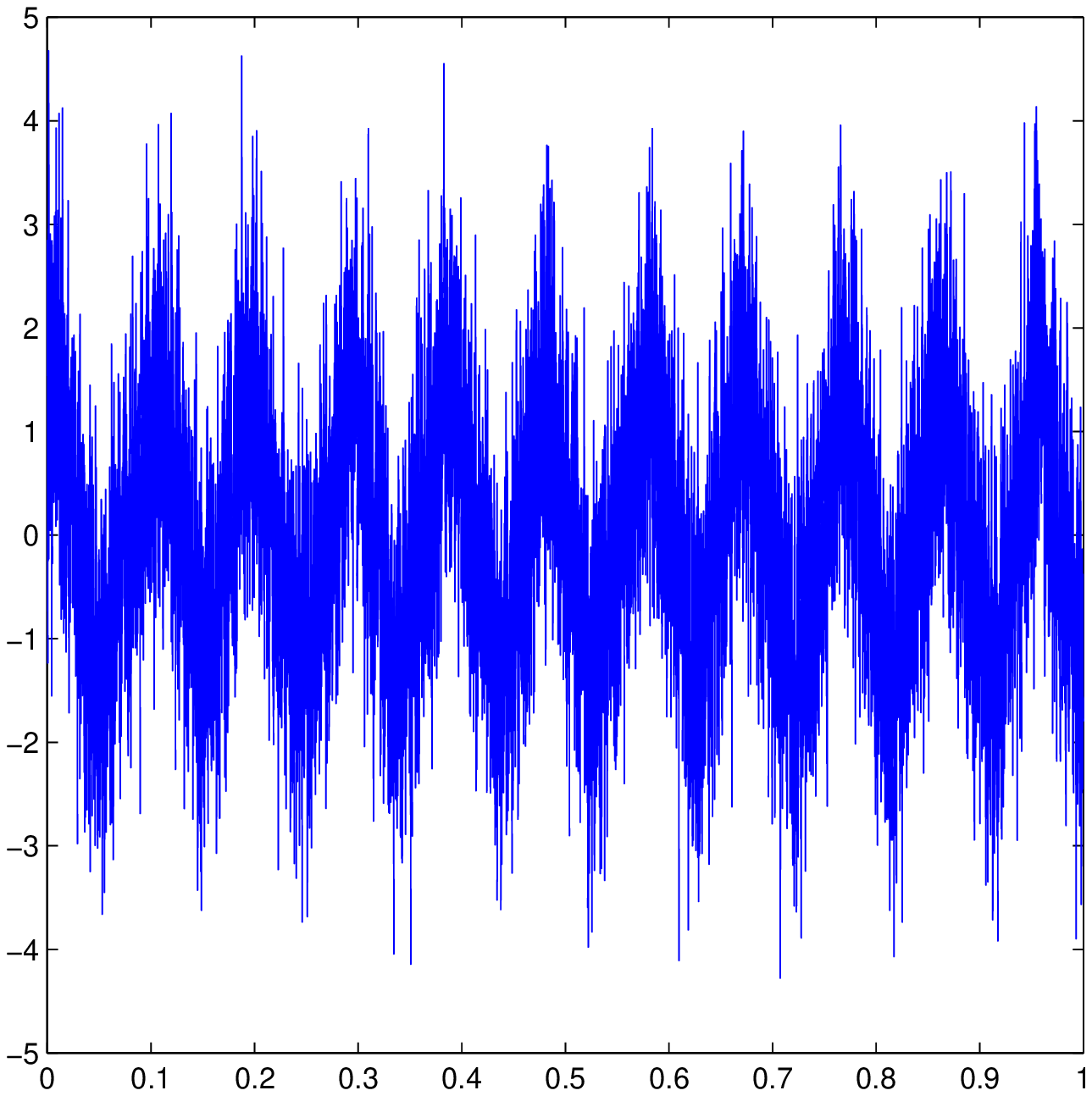}
\includegraphics[width=0.3\textwidth]{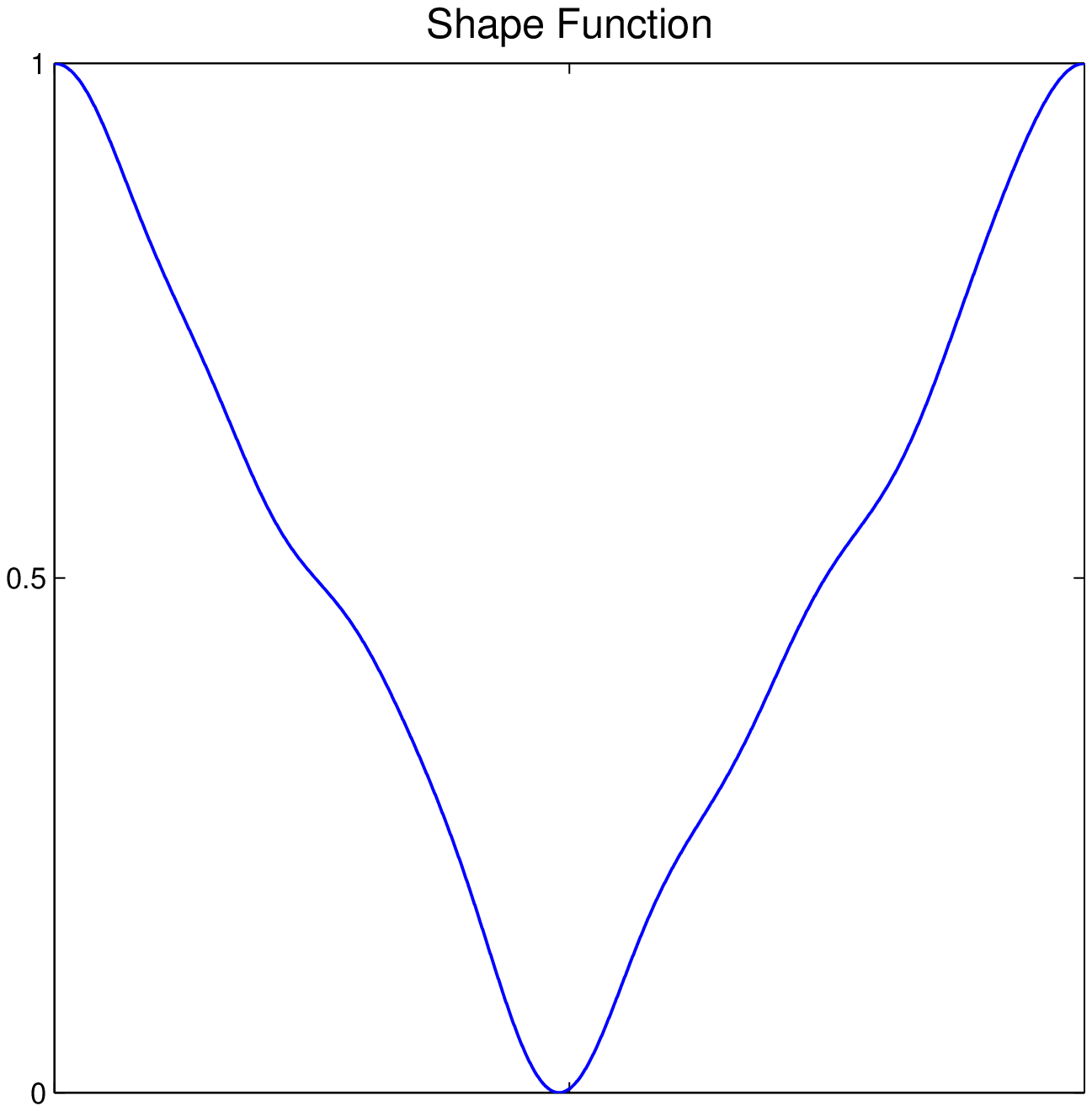}
\includegraphics[width=0.3\textwidth]{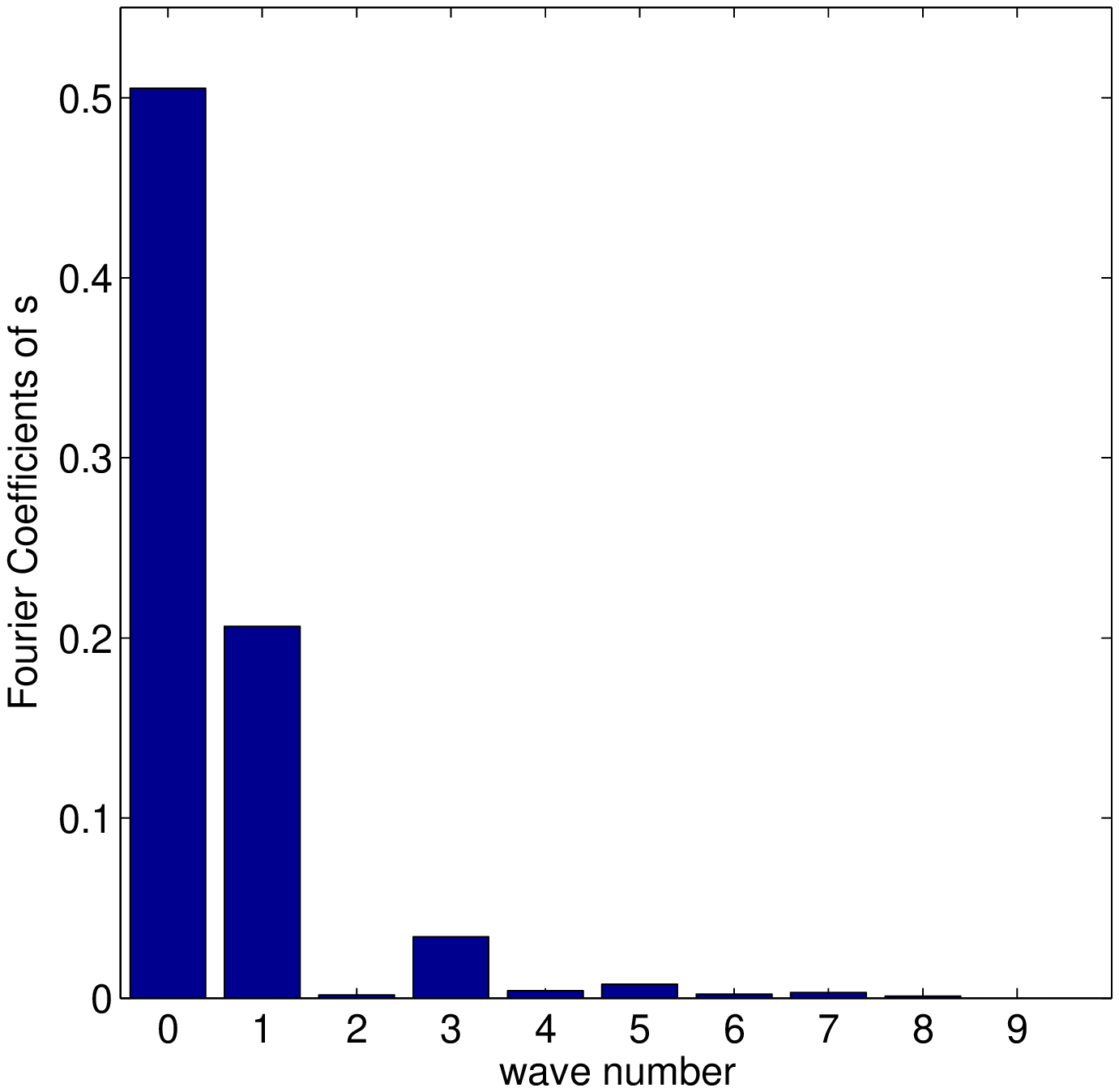}

     \end{center}
    \caption{ \label{duffing-noise}Left: the solution of duffing equation with noise $X(t)$;
Middle: the shape function $s$; Right: the Fourier coefficients of $s$.}
\end{figure}

\section{Some preliminary error analysis for the data with scale separation}
\label{analysis}

In this section, we perform some preliminary error analysis for our
nonlinear matching pursuit method. To guarantee uniqueness of the
decomposition, we need to impose certain scale separation property
for the data that we try to decompose. Before we state our result, we
first define what we mean by scale separation for a given signal.

\begin{definition}
  [Scale-separation]
\label{scale-seperation}
One function $f(t)=a(t)\cos\theta(t)$ is said to satisfy a
scale-separation property with a separation factor $\e >0$,
if $a(t)$ and $\theta(t)$ satisfy the following conditions:
\begin{eqnarray*}
 && a(t)\in C^1(\mathbb{R}),\; \theta\in C^2(\mathbb{R}),\\
&&\inf_{t\in \mathbb{R}} \theta'(t)>0,\quad
M=\sup_{t\in \mathbb{R}}|\theta''(t)|<\infty\\
&& \left|\frac{a'(t)}{\theta'(t)}\right|,\; \left|\frac{\theta''(t)}{\left(\theta'(t)\right)^2}\right|\le \e,\quad \forall t\in
\mathbb{R}.
\end{eqnarray*}
\end{definition}
\begin{definition}
  [Well-separated signal]
\label{well-seperated}
A signal $f: \mathbb{R}\rightarrow \mathbb{R}$ is said to be
well-separated with separation factor $\e$ and frequency
ratio $d>1$ if it can be written as
\begin{eqnarray*}
  f(t)=\sum_{k=1}^Ka_k(t)\cos\theta_k(t)
\end{eqnarray*}
where all $f_k(t)=a_k(t)\cos\theta_k(t)$ satisfies the
scale-separation property with separation factor $\e$, and
their phase function $\theta_k$ satisfies
\begin{eqnarray}
  \label{seperation-IMF}
\theta_k'(t)\ge d \theta_{k-1}'(t),\quad \forall t\in \mathbb{R}.
\end{eqnarray}
\end{definition}
\begin{theorem}
\label{main}
  Let $f(t)$ be a function satisfying the scale-separation property with
separation factor $\e$ and frequency ratio $d$ as defined in
Definition \ref{well-seperated}. Choose a low-pass filter $\phi$
such that its Fourier Transform $\widehat{\phi}$ has support in
$[-\Delta, \Delta]$ with $\Delta < \frac{d-1}{d+1/2}$ and
$\widehat{\phi}(k)=1,\;\forall k\in [-\Delta/2,\Delta/2]$.
If in the $n$th step, the approximate phase function
$\theta_{k_0}^n$ satisfies the following condition:
\begin{eqnarray}
 \left|\frac{\left(\theta_{k_0}^n\right)'(t)}{\theta_{k_0}'(t)}-1\right|<\frac{\Delta}{2},\quad
  \left|\frac{\left(\theta_{k_0}^n\right)''(t)}{\left(\left(\theta_{k_0}^n\right)'(t)\right)^2}\right|\le \e,
\end{eqnarray}
then the accuracy of the phase function in the next step is order $\e$, i.e.
\begin{eqnarray}
  \left|\theta_{k_0}^{n+1}(t)-\theta_{k_0}(t)\right|=O(\e).
\end{eqnarray}
\end{theorem}
In order to prove the above theorem, we need the following lemma:
\begin{lemma}
  \label{conv-lemma}
For any $a(t)\in C^1(\mathbb{R})$, $\theta\in C^2(\mathbb{R})$, we have
\begin{eqnarray}
\left|\int a(\tau)e^{-i\theta(\tau)}\; \phi(\tau-t) d\tau-a(t)e^{-i\theta(t)}\widehat{\phi}\left(\theta'(t)\right)\right|
\le \sup |a'(t)|I_1+\frac{1}{2}|a(t)|\sup |\theta''(t)|I_2 ,
\end{eqnarray}
where $I_n=\int|t^n\phi(t)|dt$.
\end{lemma}
\begin{proof} The proof follows from the following direct calculations:
\begin{eqnarray}
  \label{eq-cos1}
&&\left|\int a(\tau)e^{-i\theta(\tau)}\; \phi(\tau-t) d\tau-a(t)e^{-i\theta(t)}\widehat{\phi}\left(\theta'(t)\right)\right|\nonumber\\
&=&\left|\int \left(a(\tau)-a(t)\right)e^{i\theta(\tau)}\; \phi(\tau-t) d\tau +
a(t)\int \left(e^{-i\theta(\tau)}-e^{-i\left(\theta(t)-\theta'(t)(\tau-t)\right)}\right)\; \phi(\tau-t) d\tau\right|\nonumber\\
&=&\left|\int \left(a(\tau)-a(t)\right)e^{i\theta(\tau)}\; \phi(\tau-t) d\tau +
a(t)\int \left(e^{-i\left(\theta(\tau)-\theta(t)-\theta'(t)(\tau-t)\right)}-1\right)e^{-i\left(\theta(t)-\theta'(t)(\tau-t)\right)}\;
\phi(\tau-t) d\tau\right|\nonumber\\
&\le& \sup |a'(t)|\int |\tau\phi(\tau)|d\tau+|a(t)|\left|\int \left(e^{-\frac{1}{2}i\theta''(s(\tau))(\tau-t)^2}-1\right)
 e^{-i\left(\theta(t)+\theta'(t)(\tau-t)\right)}\phi(\tau-t) d\tau\right|\nonumber\\
&\le& \sup |a'(t)|\int |\tau\phi(\tau)|d\tau+|a(t)|\int\left| \frac{1}{2}\theta''(s(\tau))(\tau-t)^2
 \phi(\tau-t)\right| d\tau\nonumber\\
&\le& \sup |a'(t)|\int |\tau\phi(\tau)|d\tau+\frac{1}{2}|a(t)|\sup |\theta''(t)|\int\left|\tau^2
 \phi(\tau)\right| d\tau\nonumber\\
&=& \sup |a'(t)|I_1+\frac{1}{2}|a(t)|\sup |\theta''(t)|I_2
\end{eqnarray}
\end{proof}

\begin{remark}
We remark that since we typically deal with data of finite
support and extend them periodically to the whole domain, the
estimates for $I_1$ and $I_2$ in the above lemma are effectively
taken only in the finite support of the data.
\end{remark}

\begin{corollary}
If the Fourier Transform of the low-pass filter $\phi$ is symmetric,
i.e. $\widehat{\phi}(k)=\widehat{\phi}(-k)$, then we have
\label{coro}
  \begin{eqnarray}
\label{esti-cos}
\left|\int a(\tau)\cos\left(\theta(\tau)\right) \phi(\tau-t) d\tau-a(t)\cos\theta(t)\widehat{\phi}\left(\theta'(t)\right)\right|
\le \sup |a'(t)|I_1+\frac{1}{2}|a(t)|\sup |\theta''(t)|I_2\\
\label{esti-sin}
\left|\int a(\tau)\sin\left(\theta(\tau)\right) \phi(\tau-t) d\tau-a(t)\sin\theta(t)\widehat{\phi}\left(\theta'(t)\right)\right|
\le \sup |a'(t)|I_1+\frac{1}{2}|a(t)|\sup |\theta''(t)|I_2 .
\end{eqnarray}
\end{corollary}
Now we can prove Theorem \ref{main}. Here we only give a sketch of the
proof, the detail of the proof is deferred to the Appendix.

\begin{proof}\textbf{of Theorem \ref{main}}: In order to simplify the
notation, we denote $\overline{\theta}=\theta_{k_0}^n$ and
and define $\overline{(\cdot)}$ as the mapping from $t$ to
$\overline{\theta}$, i.e.
$\overline{f}(\overline{\theta})=f(t),\quad \forall f$.
In our algorithm, we update $\theta_{k_0}^{n+1}$ by the following
step
\begin{eqnarray}
  \theta_{k_0}^{n+1}=\overline{\theta}-\arctan\left(\frac{b(t)}{a(t)}\right),\quad
  a(t)=A(\overline{\theta}(t)),\quad b(t)=B(\overline{\theta}(t)),
\end{eqnarray}
where
\begin{eqnarray}
 A(\gamma)=2\int \overline{f}(\overline{\theta})
\cos(\overline{\theta})\phi\left(\overline{\theta}-\gamma\right) d\overline{\theta},\quad
B(\gamma)=2\int \ol{f}(\ot) \sin(\overline{\theta})\phi\left(\overline{\theta}-\gamma\right) d\overline{\theta} .
\end{eqnarray}
We will prove that $a(t)$ and $b(t)$ satisfy the following estimates:
\begin{eqnarray}
  a(t)&=&A\left(\overline{\theta}(t)\right)=a_{k_0}(t)\cos\left(\theta_{k_0}(t)-\overline{\theta}\right)+O(\e).\\
  b(t)&=&B\left(\overline{\theta}(t)\right)=a_{k_0}(t)\sin\left(\theta_{k_0}(t)-\overline{\theta}\right)+O(\e).
\end{eqnarray}
Then, we can get
\begin{eqnarray}
  \Delta \theta=\arctan\left(\frac{B(t)}{A(t)}\right)=\theta_{k_0}(t)-\overline{\theta}+O(\e),
\end{eqnarray}
which implies that
\begin{eqnarray}
  \left|\theta_{k_0}^{n+1}(t)-\theta_{k_0}(t)\right|=O(\e).
\end{eqnarray}
First, let us estimate $A(\gamma)$ as follows:
\begin{eqnarray}
  \label{eq-cos}
 A(\gamma)&=&2\int \ol{f}(\ot) \cos(\overline{\theta})\phi\left(\overline{\theta}-\gamma\right) d\overline{\theta}
=2\sum_{k=1}^n\int \ol{a}_k(\ot)\cos\theta_k(t) \cos(\overline{\theta})\phi\left(\overline{\theta}-\gamma\right) d\overline{\theta}
\nonumber\\
&=&\sum_{k=1}^n\int \ol{a}_k(\ot)\cos\left(\theta_k(t)+\overline{\theta}\right)\phi\left(\overline{\theta}-\gamma\right) d\overline{\theta}
+\sum_{k=1}^n\int \ol{a}_k(\ot)\cos\left(\theta_k(t)-\overline{\theta}\right)\phi\left(\overline{\theta}-\gamma\right) d\overline{\theta}
\nonumber\\
&=&\sum_{k=1}^n\int \ol{a}_k(\ot)\cos\left(\theta_k(t)+\overline{\theta}\right)\phi\left(\overline{\theta}-\gamma\right) d\overline{\theta}
+\sum_{k\ne k_0}\int \ol{a}_k(\ot)\cos\left(\theta_k(t)-\overline{\theta}\right)\phi\left(\overline{\theta}-\gamma\right) d\overline{\theta}
\nonumber\\
&&+\int \ol{a}_{k_0}(\ot)\cos\left(\theta_{k_0}(t)-\overline{\theta}\right)\phi\left(\overline{\theta}-\gamma\right) d\overline{\theta}
\nonumber\\
&=& I+II+III.
\end{eqnarray}
For I and II, the scale separation assumption of the data implies
that $\ol{a}_k(\ot)\cos\left(\theta_k(t) +\overline{\theta}\right)$
and $\ol{a}_k(\ot)\cos\left(\theta_k(t)-\overline{\theta}\right) $
($k\ne k_0$) are more oscillatory than the kernel
$\phi(\overline{\theta})$. Thus we expect that these two terms are
small after convolution with a smooth kernel. In fact,
we can prove that $I, II=O(\e)$ by Corollary \ref{coro}.

The estimate for III is more difficult. By our assumption, we
have
$ \left|\frac{\left(\theta_{k_0}^n\right)'(t)}{\theta_{k_0}'(t)}-1\right|<\frac{\Delta}{2}.$
This implies that
$\ol{a}_{k_0}(\ot)\cos\left(\theta_{k_0}(t)-\overline{\theta}\right)$ is smoother than the kernel $\phi(\overline{\theta})$.
Using Corollary \ref{coro}, we can prove that
$III= \ol{a}_{k_0}(\ot)\cos\left(\theta_{k_0}(t)-\overline{\theta}\right)+O(\e)$.
By combining these results, we get
\begin{eqnarray}
  a(t)=A\left(\overline{\theta}(t)\right)=a_{k_0}(t)\cos\left(\theta_{k_0}(t)-\overline{\theta}\right)+O(\e).
\end{eqnarray}
Similarly, we can prove the estimate for $b(t)$
\begin{eqnarray}
  b(t)=B\left(\overline{\theta}(t)\right)=a_{k_0}(t)\sin\left(\theta_{k_0}(t)-\overline{\theta}\right)+O(\e).
\end{eqnarray}
This completes the proof.
\end{proof}
\begin{remark}
  Under the same assumption, we can prove that the error of the instantaneous frequency is also order $\e$, i.e.
  \begin{eqnarray}
    \left|\left(\theta_{k_0}^{n+1}\right)'(t)-\theta_{k_0}'(t)\right|=O(\e).
  \end{eqnarray}
The argument is almost the same as the above proof, except that
the calculation is a little more involved.
\end{remark}
From the above theorem, we can see that the accuracy of our method depends
on the factor of scale separation. This is also consistent with our
numerical results. In the following numerical example, we compare the
IMFs obtained by our method for two
different signals. One has poor scale separation, the other one has
better scale separation. The signals are given by \myref{signal-ss}.
The signal $f_2$ has a better scale separation property than $f_1$
since its instantaneous frequency is twice of that of $f_1$. As shown
in Fig. \ref{comp-ss}, the error we obtain for $f_2$ is considerably
smaller than that of $f_1$.
\begin{eqnarray}
\label{signal-ss}
  a_0(t)=a_1(t)=\frac{1}{1.1+\cos(2\pi t)},\quad \theta=10\sin(2\pi t)+40\pi t.\nonumber\\
  f_1(t)=a_0(t)+a_1(t)\cos\theta(t),\quad  f_2(t)=a_0(t)+a_1(t)\cos(2\theta(t)) .
\end{eqnarray}
\begin{figure}

    \begin{center}
	\includegraphics[width=0.45\textwidth]{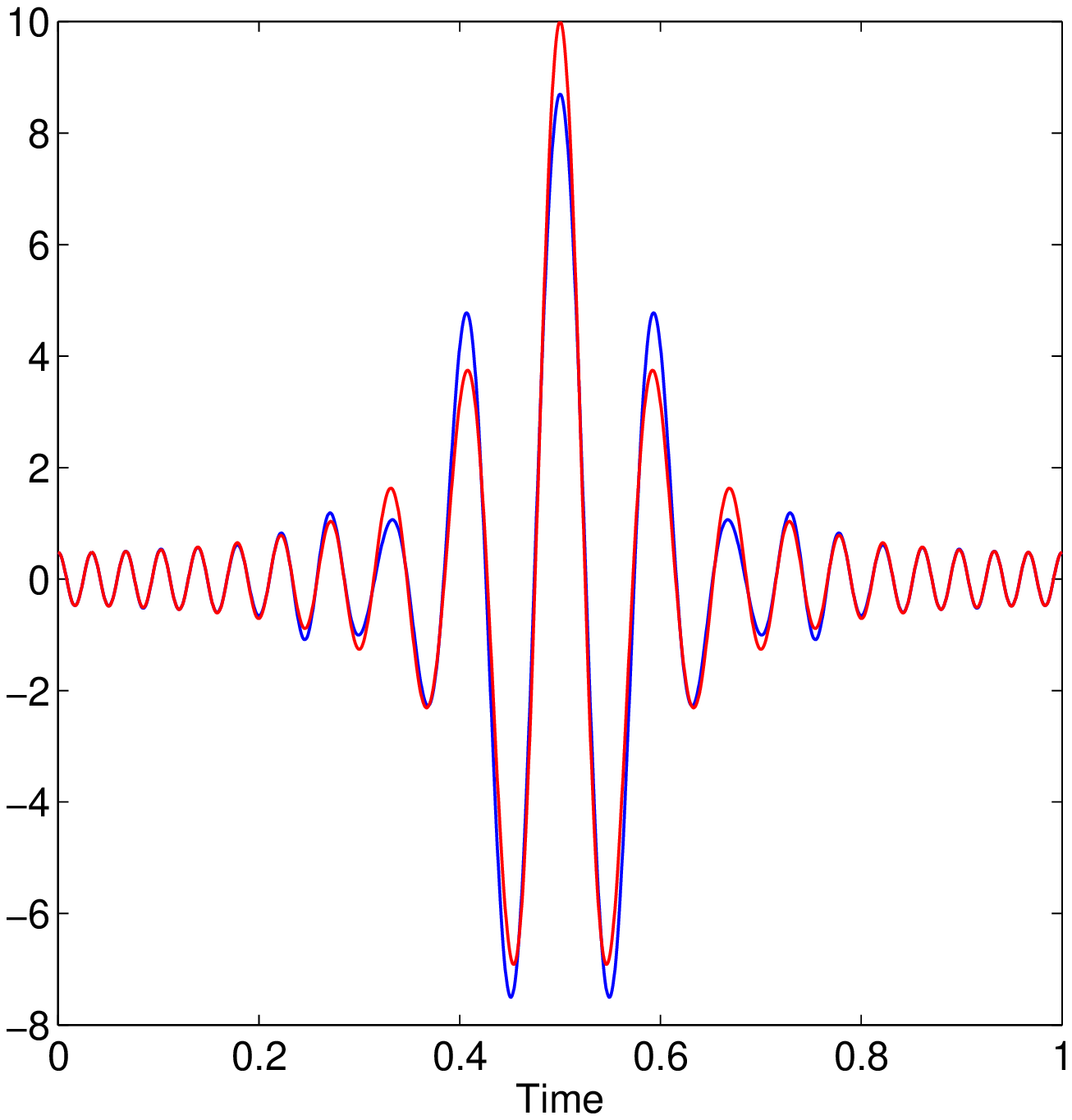}
	\includegraphics[width=0.45\textwidth]{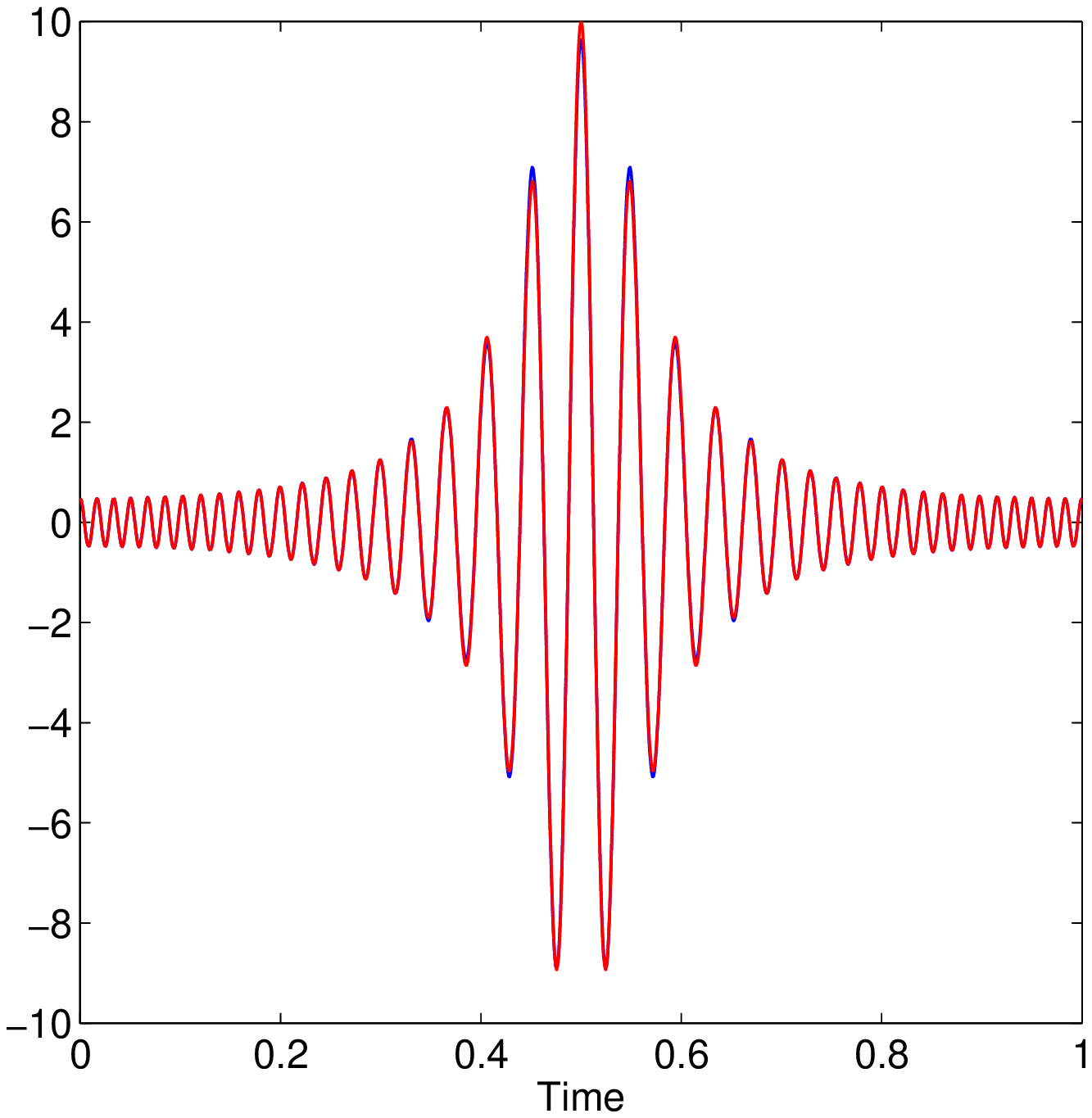}
     \end{center}
     \caption{\label{comp-ss} IMFs with
different factors of scale separation.
Left: poor scale separation; Right: good scale separation.}
\end{figure}

We would like to point out that the error estimate given by
Theorem \ref{main} is highly over-estimated. In Fig. \ref{comp-ss},
we can see that even for the signal with a poor scale separation property,
the instantaneous frequency we obtain is still reasonably accurate,
although the corresponding scale separation factor
$\e\approx 1.4$ is quite big
according to Definition \ref{scale-seperation}).

In the estimate of Lemma \ref{conv-lemma}, instead of taking the
supreme over $\mathbb{R}$, we can take supreme over a finite
interval, since we can choose a low-pass filter $\phi$ that decays
exponentially fast. Then, we can get a more local estimate:
\begin{eqnarray}
\left|\int a(\tau)e^{i\theta(\tau)}\; \phi(\tau-t) d\tau-a(t)e^{i\theta(t)}\widehat{\phi}\left(\theta'(t)\right)\right|
\le \sup_{t\in S_\phi} |a'(t)|I_1+\frac{1}{2}|a(t)|\sup_{t\in S_\phi}
 |\theta''(t)|I_2 ,
\end{eqnarray}
where $S_\phi=\left\{t\in \mathbb{R}: |\phi(t)|>\e\right\}$.

This seems to suggest that for the signal that does not have a good scale
separation property in some region, its influence on the accuracy
of the decomposition is limited to that region. This is the temporal
locality property we have mentioned in Section 5.1. This property
 can be also seen
in Fig. \ref{comp-ss}. In this example, the scales are not well
separated in the center of the interval. However, the scales are
better separated near the two ends. We can see that the error near
the boundary is much smaller than that in the center.

From this analysis, we can also see that the low-pass filter with
smooth Fourier spectrum (such as the cosine function given by
\myref{cutoff-cosine}) would perform better than that with
discontinuous spectrum (such as the stair function given by
\myref{cutoff-jump}) in terms of maintaining the temporal locality
property of the decomposition. The low-pass filter with discontinuous
spectrum decays much slower in the time domain due to the Gibbs
phenomena. This is why we use the cosine low-pass filter
instead of the stair one.

Theorem \ref{main} tells us that if we have a good initial guess,
then there is no need to do iterations. But in most cases, we have
only a rough initial guess, and the condition in Theorem \ref{main} may
not be satisfied. In this case, the iterative procedure in our algorithm
improves the result gradually as the number of iterations increases.
It will provide a good approximation after a number of iterations.
In Fig. \ref{illustrate-iteration}, we show how the iteration can
improve the approximation of the instantaneous frequency for a
simple chirp signal: $f(t)=\cos(10\pi(3t+1)^2)$.

In this example, the initial guess for the instantaneous frequency
is a constant,
$\theta_0=80\pi t$. Near the intersection of $\theta_0' = 80\pi $ and the
exact instantaneous frequency, the initial guess is relatively good.
After one step, we can get a better approximation in this local region.
This new estimate gives us a better guess for the instantaneous
frequency in a slightly larger interval that contains the good interval
given by the initial guess. Then in the next step, we can get a
good approximation in an even larger interval.
Gradually, we can get an accurate approximation of the
instantaneous frequency in the whole interval.
Fig. \ref{illustrate-iteration} plots the approximate instantaneous
frequency in different steps. As the number of iterations increases,
the region in which we have a good approximation becomes larger and
larger. Finally, the iterative algorithm produces an accurate
instantaneous frequency in the entire domain.

\begin{figure}

    \begin{center}
	\includegraphics[width=0.6\textwidth]{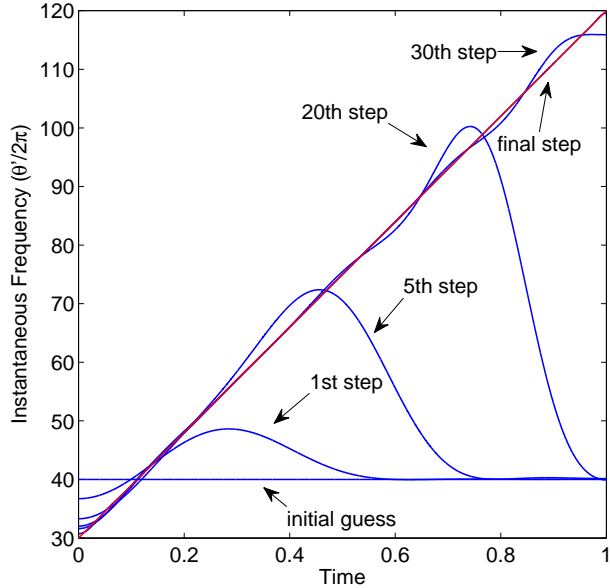}
     \end{center}
     \caption{\label{illustrate-iteration} Instantaneous frequency at different step.}
\end{figure}

\section{Conclusion}

In this paper, we introduce a new data-driven time-frequency analysis 
method based
on the nonlinear matching pursuit. The adaptivity of our decomposition
is obtained by looking for the
sparsest representation of signals in the time-frequency domain
from a largest possible dictionary that consists of all possible
candidates for intrinsic mode functions (IMFs). Solving this
nonlinear optimization problem is in general very difficult.
We propose a nonlinear matching pursuit method to solve this
nonlinear optimization problem by generalizing the classical matching
pursuit for the $l^0$ optimization problem. One important advantage of
this nonlinear matching pursuit method is it can be implemented 
very efficiently. Further, this approach is very stable to noise.
For data with good scale separation property, our method gives
an accurate decomposition up to the boundary.

We have also carried out some preliminary theoretical study
for the nonlinear optimization method proposed in this paper.
In the case when the signal satisfies certain scale separation
conditions, we show that our iterative algorithm converges to an
approximate decomposition with the accuracy determined by
the scale separation factor of the signal.

There are some remaining issues to be studied in the future, 
such as data with poor scale separation property, data with intra-wave
frequency modulation, the so-called 'end effect' of data, and data with
incomplete or sparse samples and so on. We have addressed these issues
to some extent in this paper, but much more work need to be done to
resolve these challenging issues.

Another direction
is to generalize this adaptive data analysis method to high dimensional
data. In some physical applications such as propagation of nonlinear
ocean waves, each wave form has a dominating propagation direction.
In this case, our method has a natural generalization by adopting a
multi-dimensional phase function.

\vspace{10mm}
\noindent
\textbf{\Large Appendix}
\vspace{3mm}

\begin{proof}\textbf{of Theorem \ref{main}}:
In order to simplify the notation, we denote
$\overline{\theta}=\theta_{k_0}^n$ and use
$\overline{(\cdot)}$ to represent the mapping from
$t$ to $\overline{\theta}$, i.e.
$\overline{f}(\overline{\theta})=f(t),\quad \forall f$.

According to our algorithm, we update $\theta_{k_0}^{n+1}$ as follows:
\begin{eqnarray}
  \theta_{k_0}^{n+1}=\overline{\theta}-\arctan\left(\frac{b(t)}{a(t)}\right),
\quad
  a(t)=A(\overline{\theta}(t)),\quad b(t)=B(\overline{\theta}(t)),
\end{eqnarray}
where
\begin{eqnarray}
 A(\gamma)=2\int \ol{f}(\ot) \cos(\overline{\theta})\phi\left(\overline{\theta}-\gamma\right) d\overline{\theta},\quad
B(\gamma)=2\int \ol{f}(\ot) \sin(\overline{\theta})\phi\left(\overline{\theta}-\gamma\right) d\overline{\theta}.
\end{eqnarray}
We first estimate $A(\gamma)$ as follows:
\begin{eqnarray}
  \label{eq-cos2}
 A(\gamma)=2\int \ol{f}(\ot) \cos(\overline{\theta})\phi\left(\overline{\theta}-\gamma\right) d\overline{\theta}
=2\sum_{k=1}^n\int \ol{a}_k(\ot)\cos\theta_k(t) \cos(\overline{\theta})\phi\left(\overline{\theta}-\gamma\right) d\overline{\theta} .
\nonumber
\end{eqnarray}
For $k\ne k_0$, we have
\begin{eqnarray}
&&2\int \ol{a}_k(\ot)\cos\theta_k(t) \cos(\overline{\theta})\phi\left(\overline{\theta}-\gamma\right) d\overline{\theta}
\nonumber\\
&=&\int \ol{a}_k(\ot)\left(\cos\left(\theta_k(t)+\overline{\theta}\right)+\cos\left(\theta_k(t)-\overline{\theta}\right)
\right)\phi\left(\overline{\theta}-\gamma\right) d\overline{\theta} .
\end{eqnarray}
Since
\begin{eqnarray}
  \left|\frac{d \ol{a}_k(\ot)}{d\overline{\theta}}\right|=\left|\frac{a_k'(t)}{\overline{\theta}'(t)}\right|
\le  \e\left|\frac{\theta_k'(t)}{\overline{\theta}'(t)}\right|,
\end{eqnarray}
we obtain
\begin{eqnarray}
  \left|\frac{d^2}{d\overline{\theta}^2}\left(\theta_k(t)\pm \overline{\theta}\right)\right|
=\left|\frac{\theta_k''(t)\overline{\theta}'(t)-\theta_k'(t)\overline{\theta}''(t)}{\left(\overline{\theta}'(t)\right)^3}\right|
\le \e \left|\frac{\left(\theta_k'(t)\right)^2-\theta_k'(t)\overline{\theta}'(t)}{\left(\overline{\theta}'(t)\right)^2}\right|,
\quad \mbox{for}\;k\ne k_0.
\end{eqnarray}
Now we apply Corollary \ref{coro} for $k\ne k_0$ to obtain
\begin{eqnarray}
&&2\int \ol{a}_k(\ot)\cos\theta_k(t) \cos(\overline{\theta})\phi\left(\overline{\theta}-\phi\right) d\overline{\theta}
\nonumber\\
&=&\ol{a}_k(\ot)\cos\left(\theta_k(t)+\overline{\theta}\right)\widehat{\phi}\left(\frac{\theta_k'(t)}{\overline{\theta}'(t)}+ 1\right)
+\ol{a}_k(\ot)\cos\left(\theta_k(t)-\overline{\theta}\right)\widehat{\phi}\left(\frac{\theta_k'(t)}{\overline{\theta}'(t)}- 1\right)
+O(\e).\quad\quad
\end{eqnarray}
Using the condition $\Delta<\frac{d-1}{d+1/2}$, $\theta_k'(t)>d\theta'_{k-1}(t)$ and
 $\left|\frac{\overline{\theta}'(t)}{\theta_{k_0}'(t)}-1\right|<\frac{\Delta}{2}$, we get
 \begin{eqnarray}
   \frac{\theta_k'(t)}{\overline{\theta}'(t)}-1&=&\frac{\theta_k'(t)}{\theta_{k_0}(t)}
\frac{\theta_{k_0}'(t)}{\overline{\theta}'(t)}-1>d(1-\Delta/2)-1>\Delta,\quad \mbox{if}\; k>k_0,\\
   \frac{\theta_k'(t)}{\overline{\theta}'(t)}-1&=&\frac{\theta_k'(t)}{\theta_{k_0}(t)}
\frac{\theta_{k_0}'(t)}{\overline{\theta}'(t)}-1<(1+\Delta/2)/d-1<-\Delta,\quad \mbox{if}\; k<k_0,\\
\frac{\theta_k'(t)}{\overline{\theta}'(t)}+1&>&1>\Delta .
 \end{eqnarray}
Since the support of $\widehat{\phi}$ is within $[-\Delta, \Delta]$, we have
\begin{eqnarray}
2\int \ol{a}_k(\ot)\cos\theta_k(t) \cos(\overline{\theta})\phi\left(\overline{\theta}-\phi\right) d\overline{\theta}
=O(\e) .
\end{eqnarray}
For $k=k_0$, we proceed as follows
\begin{eqnarray}
&&2\int \ol{a}_k(\ot)\cos\theta_k(t) \cos(\overline{\theta})\phi\left(\overline{\theta}-\phi\right) d\overline{\theta}
\nonumber\\
&=&\int \ol{a}_{k_0}(\ot)\left(\cos\left(\theta_{k_0}(t)+\overline{\theta}\right)+\cos\left(\theta_{k_0}(t)-\overline{\theta}\right)
\right)\phi\left(\overline{\theta}-\phi\right) d\overline{\theta} .
\end{eqnarray}
Similarly, by using the assumption
\begin{eqnarray}
  \left|\frac{d \ol{a}_{k_0}(\ot)}{d\overline{\theta}}\right|=\left|\frac{a_{k_0}'(t)}{\overline{\theta}'(t)}\right|
\le  \e\left|\frac{\theta_{k_0}'(t)}{\overline{\theta}'(t)}\right|,
\end{eqnarray}
we obtain the following estimates:
\begin{eqnarray}
  \left|\frac{d}{d\overline{\theta}}\left(\theta_{k_0}(t)+ \overline{\theta}\right)\right|
=\left|\frac{\theta_{k_0}'(t)}{\overline{\theta}'(t)}+ 1\right|>1>\Delta,  \\
\left|\frac{d}{d\overline{\theta}}\left(\theta_{k_0}(t)- \overline{\theta}\right)\right|
=\left|\frac{\theta_{k_0}'(t)}{\overline{\theta}'(t)}- 1\right|<\frac{\Delta}{2},
\end{eqnarray}
\begin{eqnarray}
  \left|\frac{d^2}{d\overline{\theta}^2}\left(\theta_{k_0}(t)\pm \overline{\theta}\right)\right|
=\left|\frac{\theta_{k_0}''(t)\overline{\theta}'(t)-\theta_{k_0}'(t)\overline{\theta}''(t)}
{\left(\overline{\theta}'(t)\right)^3}\right|
\le \e \left|\frac{\left(\theta_{k_0}'(t)\right)^2-\theta_{k_0}'(t)\overline{\theta}'(t)}{\left(\overline{\theta}'(t)\right)^2}.
\right|.
\end{eqnarray}
By applying Corollary \ref{coro} again, we get
\begin{eqnarray}
&&2\int \ol{a}_k(\ot)\cos\theta_k(t) \cos(\overline{\theta})\phi\left(\overline{\theta}-\phi\right) d\overline{\theta}
\nonumber\\
&=&\ol{a}_{k_0}(\ot)\cos\left(\theta_{k_0}(t)+\overline{\theta}\right)\widehat{\phi}\left(\frac{\theta_{k_0}'(t)}
{\overline{\theta}'(t)}+ 1\right)
+\ol{a}_{k_0}(\ot)\cos\left(\theta_{k_0}(t)-\overline{\theta}\right)\widehat{\phi}\left(\frac{\theta_{k_0}'(t)}{\overline{\theta}'(t)}
- 1\right)
+O(\e)\nonumber\\
&=&\ol{a}_{k_0}(\ot)\cos\left(\theta_{k_0}(t)-\overline{\theta}\right)+O(\e).
\end{eqnarray}
Finally, we get the following estimate for $a(t)$,
\begin{eqnarray}
  \label{esti-cos-final}
  a(t)=A\left(\overline{\theta}(t)\right)=a_{k_0}(t)\cos\left(\theta_{k_0}(t)-\overline{\theta}\right)+O(\e).
\end{eqnarray}
For $b(t)$, we can obtain a similar estimate
\begin{eqnarray}
  \label{esti-cos-final1}
  b(t)=B\left(\overline{\theta}(t)\right)=a_{k_0}(t)\sin\left(\theta_{k_0}(t)-\overline{\theta}\right)+O(\e).
\end{eqnarray}
Thus, we have
\begin{eqnarray}
  \Delta\theta=\arctan\left(\frac{B(t)}{A(t)}\right)=\theta_{k_0}(t)-\overline{\theta}+O(\e),
\end{eqnarray}
which implies that
\begin{eqnarray}
  \left|\theta_{k_0}^{n+1}(t)-\theta_{k_0}(t)\right|=O(\e).
\end{eqnarray}
This completes the proof.
\end{proof}

\vspace{0.2in}
\noindent
{\bf Acknowledgments.}

This work was in part supported by the AFOSR MURI grant
FA9550-09-1-0613.
We would like to thank Professors Norden E. Huang and Zhaohua Wu
for many stimulating discussions on EMD/EEMD and topics related to
the research presented here. We would also like to thank Professors
Ingrid Daubechies, Stanley Osher, and Zuowei Shen for their
interest in this work and for a number of valuable discussions.
Prof. Hou would like to express his
gratitude to the National Central University (NCU) for their support
and hospitality during his visits to NCU in the past two years.

\end{document}